\pdfoutput=1
\RequirePackage{ifpdf}
\ifpdf 
\documentclass[pdftex]{sigma}
\else
\documentclass{sigma}
\fi

\numberwithin{equation}{section}

\newtheorem{Theorem}{Theorem}[section]
\newtheorem{Corollary}[Theorem]{Corollary}
\newtheorem{Conjecture}[Theorem]{Conjecture}
\newtheorem{Lemma}[Theorem]{Lemma}
\newtheorem{Proposition}[Theorem]{Proposition}
 { \theoremstyle{definition}
\newtheorem{Definition}[Theorem]{Definition}

\newtheorem{Example}[Theorem]{Example}
\newtheorem{Remark}[Theorem]{Remark} }

\usepackage{tikzit}

\tikzstyle{label}=[inner sep=0.1mm, font={\footnotesize}]
\tikzstyle{circ}=[fill=black, draw=black, shape=circle, line width=0.25mm, tikzit shape=circle, inner sep=0.35mm]
\tikzstyle{nodes}=[fill=white, draw=none, shape=circle, inner sep=0.7pt, font={\scriptsize}, minimum size=11pt]
\tikzstyle{label-2}=[fill=none, draw=none, shape=circle, font={\footnotesize}]

\tikzstyle{hyper}=[-, fill={rgb,255: red,48; green,255; blue,214}, fill opacity=0.6, draw={rgb,255: red,18; green,229; blue,85}, tikzit fill={rgb,255: red,48; green,255; blue,214}]
\tikzstyle{hyper1}=[-, fill={rgb,255: red,230; green,138; blue,9}, draw={rgb,255: red,255; green,128; blue,0}, fill opacity=0.5]
\tikzstyle{hyper2}=[-, fill={rgb,255: red,128; green,179; blue,255}, draw={rgb,255: red,46; green,87; blue,115}, fill opacity=0.55]
\tikzstyle{new edge style 0}=[-, fill={rgb,255: red,245; green,255; blue,39}, draw={rgb,255: red,168; green,170; blue,22}, fill opacity=0.6]
\tikzstyle{blue}=[-, draw={rgb,255: red,49; green,149; blue,255}, line width=1.5pt]
\tikzstyle{red}=[-, draw=red, line width=1.5pt]
\tikzstyle{blue thick}=[-, tikzit draw=blue, draw=blue, line width=1.7pt]

\usepackage{breqn}
\usetikzlibrary{shapes.geometric,intersections,decorations.markings}
\usetikzlibrary{calc,intersections,through,backgrounds}
\usetikzlibrary{patterns}

\newlength\friezelen
\usepackage{array}
\newcolumntype{Q}{>{\centering}p{\friezelen}<{}}

\usepackage{mathtools}
\usetikzlibrary{arrows}
\tikzset {->-/.style={decoration={markings, mark=at position .5 with {\arrow{latex}}}, postaction={decorate}}}

\usepackage{etoolbox}

\DeclareMathOperator{\wt}{wt}

\newcommand{\new}[1]{{#1}}

\newcommand{\calS}{\mathcal{S}}
\newcommand{\calC}{\mathcal{C}}
\newcommand{\Z}{\mathbb{Z}}

\begin{document}
\allowdisplaybreaks

\newcommand{\arXivNumber}{2107.14785}

\renewcommand{\PaperNumber}{089}

\FirstPageHeading

\ShortArticleName{Rooted Clusters for Graph LP Algebras}

\ArticleName{Rooted Clusters for Graph LP Algebras}

\Author{Esther BANAIAN~$^{\rm a}$, Sunita CHEPURI~$^{\rm b}$, Elizabeth KELLEY~$^{\rm c}$ and Sylvester W.~ZHANG~$^{\rm d}$}

\AuthorNameForHeading{E.~Banaian, S.~Chepuri, E.~Kelley and S.W.~Zhang}

\Address{$^{\rm a)}$~Department of Mathematics, Aarhus University, 8000 Aarhus, Denmark}
\EmailD{\href{mailto:banaian@math.au.dk}{banaian@math.au.dk}}

\Address{$^{\rm b)}$~Department of Mathematics, Lafayette College, Easton, PA 18042, USA}
\EmailD{\href{mailto:chepuris@lafayette.edu}{chepuris@lafayette.edu}}
\Address{$^{\rm c)}$~Department of Mathematics, University of Illinois Urbana-Champaign,\\
\hphantom{$^{\rm c)}$}~Urbana, IL 61801, USA}
\EmailD{\href{mailto:kelleye@illinois.edu}{kelleye@illinois.edu}}
\Address{$^{\rm d)}$~School of Mathematics, University of Minnesota, Minneapolis, MN 55455, USA}
\EmailD{\href{mailto:swzhang@umn.edu}{swzhang@umn.edu}}

\ArticleDates{Received October 13, 2021, in final form November 17, 2022; Published online November 24, 2022}

\Abstract{LP algebras, introduced by Lam and Pylyavskyy, are a generalization of cluster algebras. These algebras are known to have the Laurent phenomenon, but positivity remains conjectural. Graph LP algebras are finite LP algebras encoded by a graph. For the graph LP algebra defined by a tree, we define a family of clusters called \emph{rooted clusters}. We prove positivity for these clusters by giving explicit formulas for each cluster variable. We also give a combinatorial interpretation for these expansions using a generalization of $T$-paths.}

\Keywords{Laurent phenomenon algebra; cluster algebra; graph LP algebra; $T$-path}

\Classification{05E15; 05C70}

\section{Introduction}\label{sec:intro}

Cluster algebras were introduced by Fomin and Zelevinsky~\cite{FZ-02}. Though the original motivation for these objects was the study of total positivity, they have since been found to have connections to a wide variety of mathematical areas, including the representation theory of quivers \cite{BMRRT-06, K-10, R-10}, algebraic geometry and mirror symmetry \cite{B-17,GHK-15}, discrete integrable systems \cite{FH-13, K-10-II}, Poisson geometry \cite{GSV-05,GSV-10}, Teichm\"uller theory \cite{FG-09, FG-10, FT-18, GSV-05}, other areas of combinatorics \cite{BMDKTY-18,C-04, CFZ-02,IT-09}, and mathematical physics \cite{C-04,F-12}.

Cluster algebras are commutative rings with a family of distinguished generators called \emph{cluster variables}. The cluster variables occur in overlapping subsets of fixed size called \emph{clusters}. Given a cluster $\calC$, we can obtain a unique distinct cluster $\calC'$ by a process called \emph{mutation} where one cluster variable in $\calC$ is replaced with a different cluster variable. The two cluster variables involved in this process are related by a \emph{binomial exchange relation}; that is, their product can be expressed as a binomial in terms of the other variables in $\calC$ (or, equivalently, in $\calC'$).

Cluster algebras have several important features, including that:
\begin{enumerate}\itemsep=0pt
 \item[(1)] \emph{Laurent phenomenon.} Given a fixed choice of cluster $\calC=(x_1,\dots,x_n)$, every cluster vari\-able can be written as a Laurent polynomial in $x_1,\dots,x_n$.
 \item[(2)] \emph{Positivity.} The Laurent polynomial in (1) has positive coefficients.
\end{enumerate}

Lam and Pylyavskyy introduced Laurent phenomenon (LP) algebras as a generalization of cluster algebras~\cite{LP-12}. In an LP algebra, the restriction that exchange relations be binomial is relaxed to allow arbitrary irreducible polynomials. The precise definition of these algebras is reviewed in Section~\ref{sec:LP}. Lam and Pylyavskyy proved that the Laurent phenomenon holds for LP algebras and conjectured that positivity holds as well.

Graph LP algebras are a particularly nice class of LP algebras whose exchange relations can be encoded in a graph. Lam and Pylyavskyy explored graph LP algebras in depth in~\cite{LP-16} and gave simple descriptions of all of the clusters along with several formulas for computing the cluster variables. However, positivity for graph LP algebras remains conjectural. In this paper, we describe progress towards that conjecture.

We begin by introducing \emph{rooted clusters} for graph LP algebras. Our first main result, which we prove by giving explicit formulas for every cluster variable in terms of each rooted cluster, is positivity for such clusters.

\begin{Theorem}\label{thm:main1}
If $\Gamma$ is a tree and $\calC$ is a rooted cluster for $\Gamma$, then every cluster variable in the graph LP algebra associated to $\Gamma$ can be expressed as a Laurent polynomial with positive coefficients in the elements of $\calC$.
\end{Theorem}

We then introduce a generalization of Schiffler's $T$-paths for type $A$ cluster algebras~\cite{S-08} for our setting.

\begin{Theorem}
\label{thm:main2}
Let $\Gamma$ be a tree and $\calC$ be a rooted cluster for $\Gamma$. If $S$ is a connected subset of vertices of $\Gamma$, then the cluster variable $Y_S$ has the combinatorial expansion formula \[Y_S=\sum_{\substack{\text{complete hyper}\\T\text{-paths }\alpha\text{ for }S}}\wt(\alpha).\]
\end{Theorem}

We will begin in Section~\ref{sec:LP} by giving more background on LP algebras and then specifically graph LP algebras. In Section~\ref{sec:rooted-clusters}, we introduce rooted clusters. Section~\ref{sec:formulas} gives formulas for the cluster variables in terms of a rooted cluster $\calC$ and contains the proof of Theorem~\ref{thm:main1}. We begin Section~\ref{sec:T-path} with background on $T$-paths for type $A$ cluster algebras and then define hyper $T$-paths. This section culminates with the proof of Theorem~\ref{thm:main2}. We conclude with a few thoughts about future work.

\section{Preliminaries}\label{sec:prelim}

\subsection{LP algebras}\label{sec:LP}

LP algebras were defined by Lam and Pylyavskyy in~\cite{LP-12}. We state the full definition of LP algebras in this section for the sake of completeness, but this paper will focus on a particular subset: graph LP algebras. Our initial definition of graph LP algebras, in Definition~\ref{defn:graph-LP}, uses the following definition of an LP algebra, but we will later give an equivalent and simpler definition in Theorem~\ref{thm:graph-LP-defn}.

\begin{Definition}[{cf.\ \cite[Section~2.1]{LP-12}}]
\label{defn:LP-alg}
Let $R$ be a coefficient ring over $\Z$ that is a unique factorization domain. Let $\mathcal{F}$ be the field of rational functions in $X_1,\dots,X_n$ over $\text{Frac}(R)$ where $X_1,\dots,X_n$ are indeterminates. A \emph{seed} is a collection $\{(x_i,F_i)\}_{1\leq i\leq n}$ where
\begin{itemize}\itemsep=0pt
 \item $x_1,\dots,x_n$ is a transcendence basis for $\mathcal{F}$ over $\text{Frac}(R)$.
 \item $F_1,\dots, F_n$ are polynomials in $R[x_1,\dots,x_n]$ such that
 \begin{enumerate}\itemsep=0pt
 \item[(LP1)] $F_i$ is irreducible in $R[x_1,\dots,x_n]$ and is not divisible by any variable $x_j$, and
 \item[(LP2)] $F_i$ does not involve $x_i$.
 \end{enumerate}
\end{itemize}
The individual elements $x_1, \dots, x_n$ are known as \emph{cluster variables}, the entire set $\{ x_1, \dots, x_n \}$ as a \emph{cluster}. The functions $F_1,\dots,F_n$ are the \emph{exchange polynomials}. The \emph{rank} of the seed $\{ (x_i, F_i) \}_{1 \leq i \leq n}$ is $n$.
\end{Definition}
Given a seed $t = \{(x_i,F_i)\}_{1\leq i\leq n}$, let $ \mathcal{L} = \mathcal{L}(t) := R\big[x_1^{\pm 1}, \dots, x_n^{\pm 1}\big]$. The LP algebra associated to the seed $t$ lives within this Laurent polynomial ring. Precisely defining this LP algebra, however, requires some additional framework.

For each seed, the collection of exchange polynomials defines a collection of \emph{exchange Laurent polynomials} $\{ \hat{F_1}, \dots, \hat{F_n} \}$ such that
\begin{itemize}\itemsep=0pt
 \item $\hat{F_i} := x_1^{a_1} \cdots x_{i-1}^{a_{i-1}}x_{i+1}^{a_{i+1}}\cdots x_n^{a_n}F_i$ for some $a_1, \dots, a_{i-1},a_{i+1},\dots,a_n \in \mathbb{Z}_{\leq 0}$, and
 \item for $i \neq j$, $\hat{F_i}\big|_{x_j \leftarrow F_j/x}$ lies in $R\big[x_1^{\pm 1}, \dots, x_{j-1}^{\pm 1},x^{\pm 1}, x_{j+1}^{\pm 1}, \dots, x_n^{\pm 1}\big]$ and is not divisible by $F_j$ (as an element of this ring).
\end{itemize}
The collection of exchange Laurent polynomials $\big\{ \hat{F_1}, \dots, \hat{F_n} \big\}$ is well-defined and is uniquely determined by the original collection of exchange polynomials $\{ F_1,\dots,F_n \}$. The exchange Laurent polynomials allow us to state a definition of \emph{mutation}.

\begin{Definition}
For a seed $t = \{ (x_i,F_i) \}_{1 \leq i \leq n}$, \emph{mutation in direction $k$} produces a new seed $t' = \mu_k(\{ (x_i,F_i) \}_{1 \leq i \leq n}) = \{ (x_i',F_i') \}_{1 \leq i \leq n}$ where the new cluster variables are given by the exchange relation
 \[
 x_i' := \begin{cases} \hat{F_k}/x_k, & i = k, \\ x_i, & i \neq k. \end{cases}
 \]
The new exchange polynomials are determined according to the following cases:
 \begin{itemize}\itemsep=0pt
 \item $F_k' := F_k$.
 \item If $i \neq k$ and $F_i$ does not depend on $x_k$, then define $F_i'$ as any polynomial which satisfies $F_i' \ltimes F_i$ (i.e., $F_i'$ and $F_i$ differ multiplicatively by a unit in $R$), where $F_i'$ is now considered as an element of $\mathcal{L}' = \mathcal{L}(t')$.
 \item If $i \neq k$ and $F_i$ does depend on $x_k$, then define
 \[ G_i := F_i\big|_{x_k \leftarrow (\hat{F_k}|_{x_i \leftarrow 0})/x_k'} \]
 and define $H_i$ as $G_i$ with all common factors with $ \hat{F_k}\big|_{x_i \leftarrow 0}$ removed. Note that $H_i$ is only defined up to multiplication by a unit in $R$. Now, define
 \[ F_i' := MH_i, \]
 where $M$ is a Laurent monomial in $x_1', \dots, x_{i-1}', x_{i+1}', \dots, x_n'$ whose coefficient is a unit in $R$, such that $F_i'$ satisfies (LP2) and is not divisible by any variable in $R[x_1',\dots,x_n']$. Such a monomial always exists, but there may be many choices for the coefficient of $M$. Therefore, $F_i'$ is defined only up to multiplication by a unit in~$R$.
 \end{itemize}
\end{Definition}
One can verify that mutation produces a collection which meets the definition of a seed; for details, see \cite[Section~2.2]{LP-12}. Note that this definition of mutation is not exactly involutive, since mutation is non-deterministic. If we obtain the seed $t'$ by mutating the seed $t$ in direction~$k$, however, it is always possible to recover $t$ by mutating~$t'$ in direction~$k$.

\begin{Definition}
A \emph{Laurent phenomenon algebra} $(\mathcal{A}, \mathcal{S})$ consists of a collection of seeds $\mathcal{S}$ and a~subring $\mathcal{A} \subset \mathcal{F}$ generated by all the cluster variables which appear in seeds in~$\mathcal{S}$. The collection~$\mathcal{S}$ must satisfy the following conditions.
\begin{itemize}\itemsep=0pt
 \item Any pair of seeds in $\mathcal{S}$ can be obtained from each other via a sequence of mutations.
 \item For any seed $(x_i,F_i) \in \mathcal{S}$ and direction $k \in [n]\new{:=\{1,2,\dots,n\}}$, there exists another seed $(x_i',F_i') \in \mathcal{S}$ which can be obtained by mutating $(x_i,F_i)$ in direction $k$.
\end{itemize}
\end{Definition}

\subsection{Graph LP algebras}\label{sec:graph-LP}

For every undirected graph $\Gamma$, we obtain a graph LP algebra $\mathcal{A}_\Gamma$. The initial seed for $\mathcal{A}_\Gamma$ is encoded by the edges of the graph.

\begin{Definition}\label{defn:graph-LP}
Let $\Gamma$ be an undirected graph on $[n]$ and $R=\Z[A_1,\dots,A_n]$. Then the \emph{graph LP algebra} $\mathcal{A}_\Gamma$ is the LP algebra generated by initial seed $ \big\{\big(X_i,A_i+\sum_{i\text{ adjacent to }j}X_j\big)\big\}_{1\leq i\leq n}$.
\end{Definition}

\begin{Remark}There is a similar definition of LP algebras from directed graphs, described in~\cite{LP-16}, for which the results of Lam and Pylyavskyy discussed later in this section still hold.
\end{Remark}

Lam and Pylyavskyy prove that these LP algebras have a particularly nice structure using nested collections.

\begin{Definition}\label{defn:nested-coll}
Let $\Gamma$ be an undirected graph on $[n]$. A family of subsets of $[n]$, $\calS=\{S_1,\dots,S_k\}$, is a \emph{nested collection} if
\begin{itemize}\itemsep=0pt
 \item for any $i,j\leq k$, either $S_i\subseteq S_j$, $S_j\subseteq S_i$, or $S_i\cap S_j=\varnothing$, and
 \item if $S_{i_1},\dots S_{i_\ell}$ are pairwise disjoint, then $S_{i_1},\dots S_{i_\ell}$ are exactly the connected components of $\bigcup_{j=1}^\ell S_{i_j}$.
\end{itemize}
We say $\calS$ is a \emph{maximal nested collection on $S$} if $\bigcup_{i=1}^k S_i=S$ and there is no $S'\subseteq S$ such that $\{S_1,\dots,S_k,S'\}$ is a nested collection.
\end{Definition}

If $\Gamma$ is a graph on $[n]$ and $\calS$ is a maximal nested collection on $S=[n]$, we will generally say that $\calS$ is a maximal nested collection without specifying $S$.

\begin{Example}\label{ex:nested-coll}
Let $\Gamma$ be the following graph:
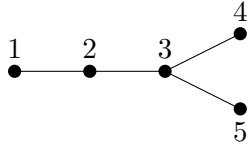
\begin{center}
\begin{tikzpicture}
\node at (-1,0.3) {1};
\node at (0,0.3) {2};
\node at (1,0.3) {3};
\node at (2,0.8) {4};
\node at (2,-0.8) {5};
\draw [line width=0.25mm, fill=black] (-1,0) circle (0.75mm);
\draw [line width=0.25mm, fill=black] (0,0) circle (0.75mm);
\draw [line width=0.25mm, fill=black] (1,0) circle (0.75mm);
\draw [line width=0.25mm, fill=black] (2,0.5) circle (0.75mm);
\draw [line width=0.25mm, fill=black] (2,-0.5) circle (0.75mm);
\draw (-1,0) -- (0,0)
(0,0) -- (1,0)
(1,0) -- (2,0.5)
(1,0) -- (2,-0.5);
\end{tikzpicture}
\end{center}

Then $\calS=\{\{1\},\{3\},\{1,2,3,4\}\}$ is a nested collection on $S=\{1,2,3,4\}$. However, it is not maximal because adding the set $S'=\{1,2,3\}$ still yields a nested collection.

As a nonexample, consider $\calS=\{\{1\},\{1,2\},\{3\},\{1,2,3,4\}\}$. We can see that this is not a nested collection by looking at the disjoint sets $\{1,2\}$ and $\{3\}$. The union of these sets is $\{1,2,3\}$, which has only one connected component.
\end{Example}

\begin{Theorem}[{\cite[Theorem 1.1]{LP-16}}]\label{thm:graph-LP-defn}
Let $\Gamma$ be an undirected graph on $[n]$. Define the matrix $\mathfrak{N}=(\mathfrak{n}_{ij})$ by \[\mathfrak{n}_{ij}=\begin{cases}
{\displaystyle \new{\frac{ A_i+\sum_{i\text{ adjacent to }k}X_k}{X_i}}}, & i=j, \\
-1, & i\text{ adjacent to }j, \\
0, & \text{otherwise}.
\end{cases}\]
Then the graph LP algebra $\mathcal{A}_\Gamma$ has cluster variables $\{X_1,\dots,X_n\} \cup \{Y_S\, |\, S\subset[n]\text{ is connected}\}$ where $Y_S$ is the determinant of the submatrix of $\mathfrak{N}$ obtained by taking only rows and columns indexed by~$S$. The clusters for $\new{\mathcal{A}_\Gamma}$ are of the form $\{X_{i_1},\dots,X_{i_k}\}\cup\{Y_S\, |\, S\in \calS\}$ where $\calS$ is a~maximal nested collection on $[n]\setminus\{i_1,\dots,i_k\}$.
\end{Theorem}

In a slight abuse of notation, we will generally write $Y_{s_1\dots s_r}$ as shorthand for $Y_{\{s_1,\dots,s_r\}}$.

\begin{Example}\label{ex:graph-LP}
Let $\Gamma$ be the graph from Example~\ref{ex:nested-coll}. One example of a valid cluster for $\mathcal{A}_\Gamma$ is $X_2$, $Y_1$, $Y_5$, $Y_{35}$, $Y_{345}$. This is because $\{\{1\}, \{5\}, \{3,5\}, \{3,4,5\}\}$ is a maximal nested collection on $\{1,3,4,5\}$. In this case, the $\mathfrak{N}$ matrix is
\[
\mathfrak{N}=
\begin{bmatrix}
\dfrac{A_1+X_2}{X_1} & -1 & 0 & 0 & 0 \\
-1 & \dfrac{A_2+X_1+X_3}{X_2} & -1 & 0 & 0 \\
0 & -1 & \dfrac{A_3+X_2+X_4+X_5}{X_3} & -1 & -1 \\
0 & 0 & -1 & \dfrac{A_4+X_3}{X_4} & 0 \\
0 & 0 & -1 & 0 & \dfrac{A_5+X_3}{X_5} \\
\end{bmatrix}.
\]
We can use this to rewrite the $Y$-variables in our cluster. For example,
\begin{align*}
Y_{35}& =\left|
\begin{matrix}
\dfrac{A_3+X_2+X_4+X_5}{X_3} & -1 \\
-1 & \dfrac{A_5+X_3}{X_5} \\
\end{matrix}
\right|\\
& =\frac{A_3A_5+A_5X_2+A_5X_4+A_5X_5+A_3X_3+X_2X_3+X_3X_4}{X_3X_5}.
\end{align*}
\end{Example}

Lam and Pylyavskyy also completely describe the exchange relations for $\mathcal{A}_\Gamma$. Before stating these relations, we must first introduce some notation. If $S\subseteq[n]$ and $i\in[n]$, then we write
\begin{itemize}\itemsep=0pt
 \item $Si$ for $S\cup\{i\}$,
 \item $S\oplus i$ for the connected component of $Si$ that includes $i$, and
 \item $S\ominus i$ for $Si\setminus(S\oplus i)$.
\end{itemize}
For any $S\subseteq[n]$ and $i,j\in[n]$, we let $P_S^{ij} :=\sum_{p:i\to_S j}Y_{S\setminus p}$, where the summation runs over paths from~$i$ to~$j$ that contain only $i$,~$j$, and vertices in~$S$, and $S\setminus p$ denotes $S$ without the vertices used in the path~$p$.

\begin{Example}\label{ex:notation}
Consider the following graph $\Gamma$:
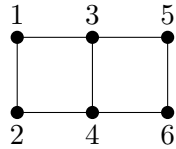
\begin{center}
\begin{tikzpicture}
\node at (-1,1.3) {1};
\node at (0,1.3) {3};
\node at (1,1.3) {5};
\node at (-1,-0.3) {2};
\node at (0,-0.3) {4};
\node at (1,-0.3) {6};
\draw [line width=0.25mm, fill=black] (-1,0) circle (0.75mm);
\draw [line width=0.25mm, fill=black] (0,0) circle (0.75mm);
\draw [line width=0.25mm, fill=black] (1,0) circle (0.75mm);
\draw [line width=0.25mm, fill=black] (-1,1) circle (0.75mm);
\draw [line width=0.25mm, fill=black] (0,1) circle (0.75mm);
\draw [line width=0.25mm, fill=black] (1,1) circle (0.75mm);
\draw (-1,0) -- (0,0)
(0,0) -- (1,0)
(1,0) -- (1,1)
(-1,1) -- (0,1)
(0,1) -- (1,1)
(0,0) -- (0,1)
(-1,0) -- (-1,1);
\end{tikzpicture}
\end{center}
Let $S=\{1,2,3,4\}$, $i=6$, $j=3$. Then there are three paths from $i$ to $j$:
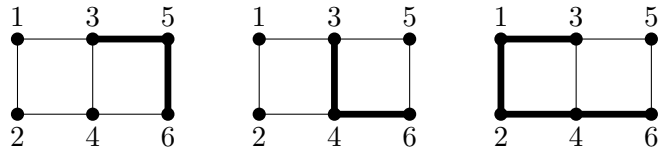
\begin{center}
\begin{tikzpicture}
\node at (-1,1.3) {1};
\node at (0,1.3) {3};
\node at (1,1.3) {5};
\node at (-1,-0.3) {2};
\node at (0,-0.3) {4};
\node at (1,-0.3) {6};
\draw [line width=0.25mm, fill=black] (-1,0) circle (0.75mm);
\draw [line width=0.25mm, fill=black] (0,0) circle (0.75mm);
\draw [line width=0.25mm, fill=black] (1,0) circle (0.75mm);
\draw [line width=0.25mm, fill=black] (-1,1) circle (0.75mm);
\draw [line width=0.25mm, fill=black] (0,1) circle (0.75mm);
\draw [line width=0.25mm, fill=black] (1,1) circle (0.75mm);
\draw (-1,0) -- (0,0)
(0,0) -- (1,0)
(1,0) -- (1,1)
(-1,1) -- (0,1)
(0,1) -- (1,1)
(0,0) -- (0,1)
(-1,0) -- (-1,1);
\draw[line width=0.9mm] (1,0) -- (1,1)
(0,1) -- (1,1);
\end{tikzpicture}
\hspace{0.2in}
\begin{tikzpicture}
\node at (-1,1.3) {1};
\node at (0,1.3) {3};
\node at (1,1.3) {5};
\node at (-1,-0.3) {2};
\node at (0,-0.3) {4};
\node at (1,-0.3) {6};
\draw [line width=0.25mm, fill=black] (-1,0) circle (0.75mm);
\draw [line width=0.25mm, fill=black] (0,0) circle (0.75mm);
\draw [line width=0.25mm, fill=black] (1,0) circle (0.75mm);
\draw [line width=0.25mm, fill=black] (-1,1) circle (0.75mm);
\draw [line width=0.25mm, fill=black] (0,1) circle (0.75mm);
\draw [line width=0.25mm, fill=black] (1,1) circle (0.75mm);
\draw (-1,0) -- (0,0)
(0,0) -- (1,0)
(1,0) -- (1,1)
(-1,1) -- (0,1)
(0,1) -- (1,1)
(0,0) -- (0,1)
(-1,0) -- (-1,1);
\draw[line width=0.9mm] (1,0) -- (0,0)
(0,1) -- (0,0);
\end{tikzpicture}
\hspace{0.2in}
\begin{tikzpicture}
\node at (-1,1.3) {1};
\node at (0,1.3) {3};
\node at (1,1.3) {5};
\node at (-1,-0.3) {2};
\node at (0,-0.3) {4};
\node at (1,-0.3) {6};
\draw [line width=0.25mm, fill=black] (-1,0) circle (0.75mm);
\draw [line width=0.25mm, fill=black] (0,0) circle (0.75mm);
\draw [line width=0.25mm, fill=black] (1,0) circle (0.75mm);
\draw [line width=0.25mm, fill=black] (-1,1) circle (0.75mm);
\draw [line width=0.25mm, fill=black] (0,1) circle (0.75mm);
\draw [line width=0.25mm, fill=black] (1,1) circle (0.75mm);
\draw (-1,0) -- (0,0)
(0,0) -- (1,0)
(1,0) -- (1,1)
(-1,1) -- (0,1)
(0,1) -- (1,1)
(0,0) -- (0,1)
(-1,0) -- (-1,1);
\draw[line width=0.9mm] (-1,0) -- (0,0)
(0,0) -- (1,0)
(-1,1) -- (0,1)
(-1,0) -- (-1,1);
\end{tikzpicture}
\end{center}

The first of these paths goes through 5, which is not in $S$. Therefore, it will not contribute to $P_S^{ij}$. The second path only goes through 4, which is in $S$. For this path we have $S\setminus p$ is $\{1,2,3,4\}\setminus\{3,4,6\}=\{1,2\}$, so it will contribute $Y_{12}$. The third path goes through 1, 2, and 4, all of which are in $S$. For this path $S\setminus p$ is $\{1,2,3,4\}\setminus\{3,1,2,4,6\}=\varnothing$, so it will contribute $Y_{\varnothing}=1$. Thus, we find that $P_S^{ij}=Y_{12}+1$.
\end{Example}

\begin{Proposition}[{\cite[Lemmas 4.7 and 4.11]{LP-16}}]\label{prop:exchange-rels}\quad
\begin{enumerate}\itemsep=0pt
\item[$(a)$] For $i\not\in S$, \[X_iY_{S\oplus i}=\frac{\sum_{j\not\in Si}P_S^{ij}X_j+\sum_{j\in Si}P_S^{ij}A_j}{Y_{S\ominus i}}.\]
\item[$(b)$] For $i,j\not\in S$ and $i\neq j$, \[Y_{S\oplus i}Y_{S\oplus j}=\frac{Y_{Sij}Y_S+P_S^{ij}P_S^{ji}}{Y_{S\ominus i}Y_{S\ominus j}}.\]
\end{enumerate}
\end{Proposition}

\subsection{Rooted clusters}\label{sec:rooted-clusters}

For the rest of this paper, we will be focusing on the case when $\Gamma$ is a tree. In this setting, we can define a special type of cluster we call a \emph{rooted cluster} that has desirable properties. There is one rooted cluster $\calC_v$ for each vertex $v$ of $\Gamma$. In order to define this cluster, we think of $\Gamma$ as being rooted at $v$. We then think of $\Gamma$ as a poset with the root $v$ being the maximal element and cover relations given by edges of $\Gamma$. This leads us to establish the following notation:
\begin{itemize}\itemsep=0pt
 \item Notice that if $i\neq v$, then $i$ is covered by exactly one vertex. We call this vertex $i^+$.
 \item The set of elements covered by $i$ is denoted $\Gamma_{\lessdot i}^v$. Similarly we have the sets $\Gamma_{\gtrdot i}^v$, $\Gamma_{< i}^v$, $\Gamma_{> i}^v$, $\Gamma_{\leq i}^v$, and $\Gamma_{\geq i}^v$ (note that $\Gamma_{\gtrdot i}^v=\{i^+\}$ if $i\neq v$).
\end{itemize}

\begin{Definition}\label{defn:rooted-cluster}
Let $\Gamma$ be a tree on $[n]$. Make $\Gamma$ into a rooted tree by choosing a vertex $v$ to be the root. Then for each vertex $x$ in $\Gamma$, let $I_x=\Gamma_{\leq x}^v$. The \emph{rooted cluster} $\calC_v$ is $\{I_x\}_{x\in[n]}$.
\end{Definition}

We verify that a rooted cluster is a maximal nested collection.

\begin{Lemma}
Given a tree, $\Gamma$, on $[n]$ and any vertex $v \in [n]$, the rooted cluster $\calC_v$ is a maximal nested collection on $[n]$.
\end{Lemma}

\begin{proof}
Because $I_v = [n]$, the collection of subsets in $\mathcal{C}_{v}$ clearly covers all vertices of $\Gamma$. Moreover, we have the containment $I_{i} \subset I_{v}$ for all vertices $i \neq v$. Thus, $I_v$ is compatible with all other subsets in $\mathcal{C}_{v}$.

Let $i$, $j$ be distinct vertices of $\Gamma$ such that $i,j \neq v$. Because $\Gamma$ is a tree, there is a unique path in $\Gamma$ between $i$ and $j$. If this path does not pass through the root~$v$, then either $i < j$ or vice versa. Without loss of generality, assume that $i < j$. Then, $I_i \subseteq I_j$. If the unique path between $i$ and $j$ does pass through~$v$, then $I_i \cap I_j = \varnothing$. Because the root $v$ is not in either $I_i$ or $I_j$, the union $I_{i} \cup I_{j}$ has exactly two disjoint connected components: $I_i$ and $I_j$. In general, for any collection of pairwise disjoint sets $I_{i_1}, \dots, I_{i_\ell}$, the union $\bigcup_{j=1}^\ell I_{i_j}$ is a disconnected graph whose connected components are exactly the subgraphs with vertices $I_{i_1}, \dots, I_{i_\ell}$.

Finally, suppose that $S \subseteq [n]$ is a subset that is compatible with all $I_i, i \in [n]$. We may assume $S$ is a connected subset of vertices in $\Gamma$ because otherwise we could simply consider its connected components. Let $x \in [n]$ be the element of $S$ which is closest to $v$; this element is unique since $S$ is connected. Since $S \cap I_x \neq \varnothing$, we must have either $S \subseteq I_x$ or $I_x \subseteq S$ in order for $S$ and $I_x$ to be compatible. Because $x$ is the largest element of $S$ and $I_x$ contains $x$ and everything smaller, $I_x\not\subseteq S$. So, $S \subseteq I_x$. If $S$ is properly contained in $I_x$, then there is at least \new{one} element $y$ less than $x$ such that $y \notin S$ but $y^+ \in S$. It follows that $I_y$ and $S$ are not compatible. Therefore, $S$ must contain everything less than $x$. Because $S$ does not contain any vertices of $\Gamma$ closer to the root than $x$, it follows that $S = I_x$.

Therefore, $\mathcal{C}_{v}$ satisfies the definition of a maximal nested collection, as stated in Defini\-tion~\ref{defn:nested-coll}.
\end{proof}

\begin{Example}\label{ex:rooted-cluster}
Consider the tree in Figure~\ref{fig:exampletree}. The cluster rooted at 1, $\calC_1$, consists of $I_1 = \{1,2,3,4,5,6,7,8\}$, $I_2 = \{2,3,4,5,6,7,8\}$, $I_3 = \{3\}$, $I_4 = \{4,6,7\}$, $I_5 = \{5,8\}$, $I_6 = \{6\}$, $I_7 = \{7\}$, and $I_8 = \{8\}$. One can check that these sets form a nested collection, and that it is not possible to add another set compatible with all others.

As a further example, given the same tree if we rooted at $4$ instead, then $I_2 = \{1,2,3,5,8\}$.
\end{Example}

\begin{figure}
\begin{center}
\begin{tikzpicture}
\node at (0.3,2) {1};
\node at (0.3,1) {2};
\node at (-1.3,0) {3};
\node at (0.3,0) {4};
\node at (1.3,0) {5};
\node at (-0.35,-1.3) {6};
\node at (0.35,-1.3) {7};
\node at (1,-1.3) {8};
\draw [line width=0.25mm, fill=black] (0,2) circle (0.75mm);
\draw [line width=0.25mm, fill=black] (0,1) circle (0.75mm);
\draw [line width=0.25mm, fill=black] (-1,0) circle (0.75mm);
\draw [line width=0.25mm, fill=black] (0,0) circle (0.75mm);
\draw [line width=0.25mm, fill=black] (1,0) circle (0.75mm);
\draw [line width=0.25mm, fill=black] (0.35,-1) circle (0.75mm);
\draw [line width=0.25mm, fill=black] (-0.35,-1) circle (0.75mm);
\draw [line width=0.25mm, fill=black] (1,-1) circle (0.75mm);
\draw (0,2) -- (0,1)
(0,1) -- (-1,0)
(0,1) -- (0,0)
(0,1) -- (1,0)
(0,0) -- (-0.35,-1)
(0,0) -- (0.35,-1)
(1,0) -- (1,-1);
\end{tikzpicture}
\end{center}
\caption{We can picture the cluster rooted at vertex 1 by taking subsets which are closed going down. For example, $I_2 = \{2,3,4,5,6,7,8\}$ and $I_5 = \{5,8\}$.}\label{fig:exampletree}
\end{figure}
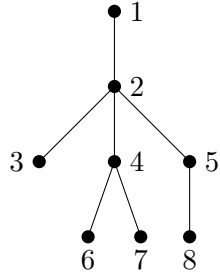

\section{Formulas}\label{sec:formulas}

Though Theorem~\ref{thm:graph-LP-defn} and Proposition~\ref{prop:exchange-rels} give formulas for all cluster variables, these formulas are not guaranteed to be in terms of the variables in any given cluster. In this section, we will prove formulas for each cluster variable in terms of the rooted cluster $\calC_v$. These formulas will allow us to prove positivity for this case.

We first state a fact that will be useful throughout this section.

\begin{Lemma}
\label{lem:in-cluster}
Let $\Gamma$ be a tree rooted at $v$ and $i$ be a vertex of $\Gamma$. Then \[Y_{\Gamma_{< i}^v}=\prod_{u\in\Gamma_{\lessdot i}^v}Y_{I_u}\]
and for $u\in\Gamma_{< i}^v$,
\[ Y_{\Gamma_{< i}^v \backslash \Gamma_{\geq u}^v} = \prod_{w \in \Gamma_{\leq i}^v \cap \Gamma_{\geq u}^{v}} ~\prod_{x \in \Gamma_{\lessdot w}^{v} \backslash \Gamma_{\geq u}^{v}} Y_{I_{x}}. \]
\end{Lemma}

\begin{proof}
The formula for $Y_{\Gamma_{< i}^v}$ follows from the definition for $Y$-variables indexed by disconnected sets.

Given $i$ and $u \in \Gamma_{< i}^{v}$, there is a chain $i = i_0 \gtrdot i_1 \gtrdot \cdots \gtrdot i_k = u$ in the vertex poset. This means $\Gamma_{< i}^v \backslash \Gamma_{\geq u}^v = \Gamma_{< i}^v \backslash \{ i_1, \dots, i_k \}.$ Consider some $y \in \Gamma_{< i}^v \backslash \Gamma_{\geq u}^v$ and let $0 \leq j \leq k$ be the largest index such that $y \in I_{i_j}$ (note that $y\in I_i=I_{i_0}$ so $j$ exists). Because $y \not\in \Gamma_{\geq u}^v$, we must have $y \in I_{j'}$ for some $j' \lessdot i_j$ where, if $j < k$, then $j' \neq i_{j+1}$.
Therefore, $\Gamma_{< i}^v \backslash \Gamma_{\geq u}^v = \bigsqcup_{j = 0}^k \bigsqcup_{x \in \Gamma_{\lessdot i_j} \backslash \{ i_{j+1} \}} I_{x}$, where we set $\{i_{k+1}\} = \varnothing$. The desired $Y$-variable identity follows.
\end{proof}

\subsection[X-variables]{$\boldsymbol{X}$-variables}\label{sec:X-formulas}

\begin{Lemma}\label{lem:X-formula-1}
Let $\Gamma$ be a tree rooted at $v$. For any vertex $i$ of $\Gamma$,
\begin{align*}
 X_i &= \begin{cases}
 \dfrac{Y_{\Gamma_{< i}^v}(X_{i^+}+A_i)+\sum_{u\in \Gamma_{< i}^v}Y_{\Gamma_{< i}^v\setminus\Gamma_{\geq u}^v}A_u}{Y_{I_i}} & \textrm{if } i \neq v, \\
 \dfrac{\sum_{u\in [n]}Y_{[n]\setminus\Gamma_{\geq u}^v}A_u}{Y_{[n]}} & \textrm{if } i = v.
 \end{cases}
\end{align*}
\end{Lemma}

\begin{proof}
These formulas follow directly from Proposition~\ref{prop:exchange-rels}(a) with $S=\Gamma_{< i}^v$.
\end{proof}

We can use this to find formulas for the $X$-variables in terms of the elements of a rooted cluster.

\begin{Proposition}\label{prop:X-formula-2}
Let $\Gamma$ be a tree rooted at $v$ and $i$ be a vertex of $\Gamma$. Then \[X_i=\frac{\sum_{u\in\Gamma_{\geq i}^v}\Big(\prod_{w\in\Gamma_{\geq i}^v\setminus\Gamma_{\geq u}^v} Y_{\Gamma_{< w}^v}\Big)\Big(\prod_{w\in\Gamma_{> u}^v} Y_{I_w}\Big)\Big(\sum_{w\in I_u}Y_{\Gamma_{< u}^v\setminus\Gamma_{\geq w}^v}A_w\Big)}{\prod_{u\in\Gamma_{\geq i}^v}Y_{I_u}}.\]
\end{Proposition}

\begin{proof}
We proceed by induction. The base case is $i=v$. Since $I_v=[n]$, in this case, the formula from the proposition reduces to the formula in the second part of Lemma~\ref{lem:X-formula-1}.

Now we will assume the proposition holds for $i^+$ and prove that it holds for $i$. From Lemma~\ref{lem:X-formula-1}, we know
\begin{align*}
X_i={} & \frac{Y_{\Gamma_{< i}^v}(X_{i^+}+A_i)+\sum_{u\in \Gamma_{< i}^v}Y_{\Gamma_{< i}^v\setminus\Gamma_{\geq u}^v}A_u}{Y_{I_i}}\\
={} & \frac{Y_{\Gamma_{< i}^v}\sum_{u\in\Gamma_{\geq i^+}^v}\Big(\prod_{w\in\Gamma_{\geq i^+}^v\setminus\Gamma_{\geq u}^v} Y_{\Gamma_{< w}^v}\Big)\Big(\prod_{w\in\Gamma_{> u}^v} Y_{I_w}\Big)\Big(\sum_{w\in I_u}Y_{\Gamma_{< u}^v\setminus\Gamma_{\geq w}^v}A_w\Big)}{Y_{I_i}\prod_{u\in\Gamma_{\geq i^+}^v}Y_{I_u}} \\
&{}+ \frac{\Big(\prod_{u\in\Gamma_{\geq i^+}^v}Y_{I_u}\Big)\Big(Y_{\Gamma_{< i}^v}A_i+\sum_{u\in \Gamma_{< i}^v}Y_{\Gamma_{< i}^v\setminus\Gamma_{\geq u}^v}A_u\Big)}{Y_{I_i}\prod_{u\in\Gamma_{\geq i^+}^v}Y_{I_u}}\\
={} & \frac{\sum_{u\in\Gamma_{\geq i^+}^v}\Big(\prod_{w\in\Gamma_{\geq i}^v\setminus\Gamma_{\geq u}^v} Y_{\Gamma_{< w}^v}\Big)\Big(\prod_{w\in\Gamma_{> u}^v} Y_{I_w}\Big)\Big(\sum_{w\in I_u}Y_{\Gamma_{< u}^v\setminus\Gamma_{\geq w}^v}A_w\Big)}{\prod_{u\in\Gamma_{\geq i}^v}Y_{I_u}} \\
&{}+ \frac{\Big(\prod_{u\in\Gamma_{> i}^v}Y_{I_u}\Big)\Big(\sum_{w\in I_i}Y_{\Gamma_{< i}^v\setminus\Gamma_{\geq w}^v}A_w\Big)}{\prod_{u\in\Gamma_{\geq i}^v}Y_{I_u}}\\
={} & \frac{\sum_{u\in\Gamma_{\geq i}^v}\Big(\prod_{w\in\Gamma_{\geq i}^v\setminus\Gamma_{\geq u}^v} Y_{\Gamma_{< w}^v}\Big)\Big(\prod_{w\in\Gamma_{> u}^v} Y_{I_w}\Big)\Big(\sum_{w\in I_u}Y_{\Gamma_{< u}^v\setminus\Gamma_{\geq w}^v}A_w\Big)}{\prod_{u\in\Gamma_{\geq i}^v}Y_{I_u}},
\end{align*}
where the second equality is from the inductive hypothesis.
\end{proof}

For example, in the tree in Figure~\ref{fig:exampletree},
\begin{align*}
X_4={} & \frac{1}{Y_{I_1}Y_{I_2}Y_{I_4}}\big( Y_{\Gamma_{< 2}^1}Y_{\Gamma_{< 4}^1}\big(A_1+Y_{\Gamma_{< 1}^1\setminus\Gamma_{\geq 2}^1}A_2+Y_{\Gamma_{< 1}^1\setminus\Gamma_{\geq 3}^1}A_3+Y_{\Gamma_{< 1}^1\setminus\Gamma_{\geq 4}^1}A_4 \\
&{}+Y_{\Gamma_{< 1}^1\setminus\Gamma_{\geq 5}^1}A_5+Y_{\Gamma_{< 1}^1\setminus\Gamma_{\geq 6}^1}A_6+Y_{\Gamma_{< 1}^1\setminus\Gamma_{\geq 7}^1}A_7+Y_{\Gamma_{< 1}^1\setminus\Gamma_{\geq 8}^1}A_8\big) \\
&{}+ Y_{\Gamma_{< 4}^1}Y_{I_1}\big(Y_{\Gamma_{< 2}^1\setminus\Gamma_{\geq 2}^1}A_2+Y_{\Gamma_{< 2}^1\setminus\Gamma_{\geq 4}^1}A_4+Y_{\Gamma_{< 2}^1\setminus\Gamma_{\geq 6}^1}A_6+Y_{\Gamma_{< 2}^1\setminus\Gamma_{\geq 7}^1}A_7\big) \\
&{}+ Y_{I_1}Y_{I_2}\big(Y_{\Gamma_{< 4}^1\setminus\Gamma_{\geq 4}^1}A_4+Y_{\Gamma_{< 4}^1\setminus\Gamma_{\geq 6}^1}A_6+Y_{\Gamma_{< 4}^1\setminus\Gamma_{\geq 7}^1}A_7\big)\big)\\
={} & \frac{1}{Y_{I_1}Y_{I_2}Y_{I_4}}\big( Y_{I_3}Y_{I_4}Y_{I_5}Y_{I_6}Y_{I_7}\big(A_1+Y_{I_3}Y_{I_4}Y_{I_5}A_2+Y_{I_4}Y_{I_5}A_3+Y_{I_3}Y_{I_5}Y_{I_6}Y_{I_7}A_4 \\
&{}+ Y_{I_3}Y_{I_4}Y_{I_8}A_5+Y_{I_3}Y_{I_5}Y_{I_7}A_6+Y_{I_3}Y_{I_5}Y_{I_6}A_7+Y_{I_3}Y_{I_4}A_8\big)\\
&{} +Y_{I_6}Y_{I_7}Y_{I_1}\big(Y_{I_3}Y_{I_4}Y_{I_5}A_2 + Y_{I_3}Y_{I_5}Y_{I_6}Y_{I_7}A_4+Y_{I_3}Y_{I_5}Y_{I_7}A_6+Y_{I_3}Y_{I_5}Y_{I_6}A_7\big)\\
&{} +Y_{I_1}Y_{I_2}\big(Y_{I_6}Y_{I_7}A_4+Y_{I_7}A_6+Y_{I_6}A_7\big)\big).
\end{align*}

By Lemma~\ref{lem:in-cluster}, we obtain the following immediate corollary:

\begin{Corollary}\label{cor:X-pos}
Let $\Gamma$ be a tree and $\mathcal{C}$ a rooted cluster for $\Gamma$. For any vertex $i$ of $\Gamma$, $X_{i}$ can be expressed as a Laurent polynomial in the elements of $\mathcal{C}$ with positive coefficients.
\end{Corollary}

\subsection[Y-variables]{$\boldsymbol{Y}$-variables}\label{sec:Y-formulas}

For the $Y$-variables, we begin by looking at sets of size~1.

\begin{Proposition}\label{prop:Y-single-formula}
Let $\Gamma$ be a tree rooted at $v$. For every vertex $i$ of $\Gamma$,
\[ Y_{\{i\}}=\frac{Y_{I_i}+\sum_{u\in\Gamma_{\lessdot i}^v}Y_{\Gamma_{< i}^v\setminus\{u\}}}{Y_{\Gamma_{< i}^v}}.\]
\end{Proposition}

\begin{proof}
If $\Gamma_{< i}^v=\varnothing$, then the formula in the proposition reduces to $Y_{\{i\}}=Y_{\{i\}}$.

If $\Gamma_{< i}^v\neq\varnothing$, recall from Theorem~\ref{thm:graph-LP-defn} that $Y_{\{i\}}=\frac{A_i+\sum_{u\text{ adjacent to }i}X_u}{X_i}$. If $i\neq v$, we can rewrite this as $Y_{\{i\}}=\frac{A_i+X_{i^+}+\sum_{u\in\Gamma_{\lessdot i}^v}X_u}{X_i}$. Using Proposition~\ref{lem:X-formula-1}, we can replace all the the $X$'s in the expression that come from elements of $\Gamma_{\lessdot i}^v$ to get
\[ Y_{\{i\}}=\frac{A_i+X_{i^+}+\sum_{u\in\Gamma_{\lessdot i}^v}\frac{Y_{\Gamma_{< u}^v}(X_i+A_u)+\sum_{w\in \Gamma_{< u}^v}Y_{\Gamma_{< u}^v\setminus\Gamma_{\geq w}^v}A_w}{Y_{I_u}}}{X_i}.\]
Giving everything a common denominator, we get
\begin{gather*}
 Y_{\{i\}}=\frac{Y_{\Gamma_{< i}^v}(A_i+X_{i^+})+\sum\limits_{u\in\Gamma_{\lessdot i}^v}\Big(\prod\limits_{w\in\Gamma_{\lessdot i}^v\setminus\{u\}}Y_{I_w}\Big)\Big(Y_{\Gamma_{< u}^v}(X_i+A_u)+\sum\limits_{w\in \Gamma_{< u}^v}Y_{\Gamma_{< u}^v\setminus\Gamma_{\geq w}^v}A_w\Big)}{Y_{\Gamma_{< i}^v}X_i},
 \end{gather*}
which simplifies to
\[ Y_{\{i\}}=\frac{Y_{\Gamma_{< i}^v}(A_i+X_{i^+})+\sum_{u\in\Gamma_{\lessdot i}^v}\Big(Y_{\Gamma_{< i}^v\setminus\{u\}}(X_i+A_u)+\sum_{w\in \Gamma_{< u}^v}Y_{\Gamma_{< i}^v\setminus\Gamma_{\geq w}^v}A_w\Big)}{Y_{\Gamma_{< i}^v}X_i}.\]
Pulling out all the $X_i$ terms, we get
\[ Y_{\{i\}}=\frac{Y_{\Gamma_{< i}^v}(A_i+X_{i^+})+\sum_{u\in\Gamma_{< i}^v}Y_{\Gamma_{< i}^v\setminus\Gamma_{\geq u}^v}A_u+X_i\sum_{u\in\Gamma_{\lessdot i}^v}Y_{\Gamma_{< i}^v\setminus\{u\}}}{Y_{\Gamma_{< i}^v}X_i}.\]
Applying Proposition~\ref{prop:exchange-rels}(a) with $S=\Gamma_{< i}^v$ proves the \new{proposition} in this case.

If $i=v$, we instead get that $Y_{\{v\}}=\frac{A_i+\sum_{u\in\Gamma_{\lessdot v}^v}X_u}{X_v}$. A similar computation to the above proves the \new{proposition} in this case.
\end{proof}

For example, in the tree in Figure \ref{fig:exampletree}, we calculate
\[
Y_{\{2\}} = \frac{Y_{I_2} + Y_{I_4}Y_{I_5} + Y_{I_3}Y_{I_6}Y_{I_7}Y_{I_5} + Y_{I_3}Y_{I_4}Y_{I_8}}{Y_{I_3}Y_{I_4}Y_{I_5}} = \frac{Y_{I_2}}{Y_{I_3}Y_{I_4}Y_{I_5}} + \frac{1}{Y_{I_3}} + \frac{Y_{I_6}Y_{I_7}}{Y_{I_4}} + \frac{Y_{I_8}}{Y_{I_5}}.
\]

Using this formula, we can prove a formula for general sets. It will be useful to define $f(i)=Y_{I_i}+\sum_{u\in\Gamma_{\lessdot i}^v}Y_{\Gamma_{< i}^v\setminus\{u\}}$, the numerator of the fraction from the above proposition. We begin with a few lemmas.

\begin{Lemma}\label{lem:Y-sum}
Let $\Gamma$ be a tree, $S$ be a connected subset of vertices, and $T=\{(a,b)\in S\times S\, |\, a=b^+\}$. Then
\[ Y_S=\sum_{n=0}^{\big\lfloor\frac{|S|}{2}\big\rfloor}(-1)^n\sum_{A\in A(n)}\bigg(\prod_{x\in (S\setminus A')}Y_{\{x\}}\bigg),\]
where $A(n)=\{A\subseteq 2^T\, |\, |A|=n\text{ and }\{a,b\}\cap\{c,d\}=\varnothing\text{ for all }(a,b),(c,d)\in A\text{ with }(a,b)\neq(c,d)\}$ and for any $A\in A(n)$, $A'=\{s\, |\, s\text{ is part of a pair in }A\}$.
\end{Lemma}

\begin{proof}
Recall from Theorem~\ref{thm:graph-LP-defn} that $Y_S$ is the determinant of matrix $\mathfrak{N}_S^S$. Consider a nonzero term in this determinant that comes from a permutation $\sigma$ where $\sigma(i)=j$. Writing $\sigma$ in cycle notation, there must be a cycle $(ijv_1\dots v_r)$. In order for this term to be nonzero, there must be edges in $\Gamma$ from $i$ to $j$, $j$ to $v-1$, and so on. Since $\Gamma$ is a tree, the only way this can happen is for $j=i$ or for the cycle in $\sigma$ to be $(ij)$. This tells us that any time there is a nonzero term in this determinant that uses entry $\mathfrak{n}_{ij}$ for $i\neq j$, it must also use entry $\mathfrak{n}_{ji}$. Since $\mathfrak{n}_{ij}=\mathfrak{n}_{ji}=-1$, this term is the product of all the other matrix entries given by the permutation.

Thus, we can break down the formula in the lemma as follows. We sum over $n$ the number of possible cycles in a permutation that gives a nonzero term in the determinant. Since all of the cycles are transpositions, the length of each permutation has the same parity as the number of cycles, giving us the $(-1)^n$. Once we have fixed $n$, we sum over all possible $n$-tuples of cycles, each given by a set $A$. If the cycles of a permutation are given by $A$, then the fixed points of the permutation are given by $S\setminus A'$. Each $x\in S\setminus A'$ contributes $Y_{\{x\}}$ to the term in the determinant we get from this permutation. The cycles each contribute 1.
\end{proof}

\begin{Lemma}\label{lem:Y-term-containment}
Let $Y_S^{(n)}=\sum_{A\in A
(n)}\big(\prod_{x\in (S\setminus A')}Y_{\{x\}}\big)$. If $t$ is a monomial that appears as a term in $Y_S^{(m+1)}$, then $t$ also appears in $Y_S^{(m)}$.
\end{Lemma}

\begin{proof}
Let $A=\{(s_1,s_2),\dots,(s_{2m-1},s_{2m}),(p,q)\}\in A(m+1)$. Then $B=\{(s_1,s_2),\dots,(s_{2m-1},\allowbreak s_{2m})\}$ is in $A(m)$. We have the following equalities:
\begin{gather*}
\prod_{x\in (S\setminus A')}Y_{\{x\}} =\frac{\prod_{x\in (S\setminus A')}f(x)}{\prod_{x\in (S\setminus A')}Y_{\Gamma_{< x}^v}}=\frac{Y_{\Gamma_{< p}^v}Y_{\Gamma_{< q}^v}\prod_{x\in (S\setminus A')}f(x)}{\prod_{x\in (S\setminus B')}Y_{\Gamma_{< x}^v}}, \\
\prod_{x\in (S\setminus B')}Y_{\{x\}} =\frac{\prod_{x\in (S\setminus B')}f(x)}{\prod_{x\in (S\setminus B')}Y_{\Gamma_{< x}^v}}=\frac{f(p)f(q)\prod_{x\in (S\setminus A')}f(x)}{\prod_{x\in (S\setminus B')}Y_{\Gamma_{< x}^v}},\\
f(p)f(q)= \bigg(Y_{I_{p}}+\sum_{u\in\Gamma_{\lessdot p}^v}\big(Y_{\Gamma_{< p}^v\setminus\{u\}}\big)\bigg)\bigg(Y_{I_{q}}+\sum_{u\in\Gamma_{\lessdot q}^v}Y_{\Gamma_{< q}^v\setminus\{u\}}\bigg)\\
\hphantom{f(p)f(q)}{}
= \bigg(Y_{I_{p}}+\sum_{\substack{u\in\Gamma_{\lessdot p}^v\\u\neq q}}Y_{\Gamma_{< p}^v\setminus\{u\}}\bigg)\bigg(Y_{I_{q}}+\sum_{u\in\Gamma_{\lessdot q}^v}Y_{\Gamma_{< q}^v\setminus\{u\}}\bigg) \\
\hphantom{f(p)f(q)=}{}+ Y_{\Gamma_{< p}^v\setminus\{q\}}\bigg(Y_{I_{q}}+\sum_{u\in\Gamma_{\lessdot q}^v}Y_{\Gamma_{< q}^v\setminus\{u\}}\bigg)\\
\hphantom{f(p)f(q)}{} = \bigg(Y_{I_{p}}+\sum_{\substack{u\in\Gamma_{\lessdot p}^v\\u\neq q}}Y_{\Gamma_{< p}^v\setminus\{u\}}\bigg)\bigg(Y_{I_{q}}+\sum_{u\in\Gamma_{\lessdot q}^v}Y_{\Gamma_{< q}^v\setminus\{u\}}\bigg)\\
\hphantom{f(p)f(q)=}{} + Y_{\Gamma_{< p}^v}Y_{\Gamma_{< q}^v}+Y_{\Gamma_{< p}^v\setminus\{q\}}\sum_{u\in\Gamma_{\lessdot q}^v}Y_{\Gamma_{< q}^v\setminus\{u\}}.
\end{gather*}

Notice that $Y_{\Gamma_{< p}^v}Y_{\Gamma_{< q}^v}$ appears in $f(p)f(q)$. This means that if $t$ is a monomial that appears in $\prod_{x\in (S\setminus A')}Y_{\{x\}}$, it must also appear in $\prod_{x\in (S\setminus B')}Y_{\{x\}}$.
\end{proof}

\begin{Theorem}\label{thm:Y-formula}
Let $\Gamma$ be a tree rooted at $v$. For any subset $S$ of the vertices of $\Gamma$, we have
\[ Y_S=\sum_{\substack{O\subseteq S\text{ containing all}\\ \text{minimal elements of }\Gamma\text{ in }S}} \, \sum_{\substack{u\colon S\setminus O\to V(\Gamma)\\ u(x)\in\Gamma_{\lessdot x}^v\setminus O}}\frac{\left(\prod_{x\in O}Y_{I_x} \right)\left(\prod_{x\in S\setminus O}Y_{\Gamma_{< x}^v\setminus\{u(x)\}} \right)}{\prod_{x\in S}Y_{\Gamma_{< x}^v}}.\]
\end{Theorem}

\begin{proof}
By Lemma~\ref{lem:Y-term-containment}, if we expand the sum given in Lemma~\ref{lem:Y-sum}, the only monomials that will appear are those that appear in the expansion of $\prod_{x\in S}Y_{\{x\}}$. We will count how many times each of these monomials appear the sum in Lemma~\ref{lem:Y-sum} with signs.

Since
\[ \prod_{x\in S}Y_{\{x\}}=\frac{\prod_{x\in S}f(x)}{\prod_{x\in S}Y_{\Gamma_{< x}^v}},\]
we will consider a monomial in this expansion
\[ t=\frac{t'}{\prod_{x\in S}Y_{\Gamma_{< x}^v}}.\]
We obtain $t'$ by multiplying the terms $Y_{I_x}$ from $f(x)$ for some set $O_t$ and then multiplying the terms $Y_{\Gamma_{< x}^v\setminus\{u_t(x)\}}$ for some choice of $u_t(x)\in\Gamma_{\lessdot x}^v$ for each $x\in S\setminus O_t$. Then \[ t'=\bigg(\prod_{x\in O_t}Y_{I_x} \bigg)\bigg(\prod_{x\in S\setminus O_t}Y_{\Gamma_{< x}^v\setminus\Gamma_{\geq u_t(x)}^v} \bigg).\]
Notice that every minimal element of $\Gamma$ that is in $S$ must be in $O_t$ for every $t$.

Let $P_t$ be the set of maximal chains in $(S\setminus O_t)$ that follow $u_t$. That is,
\[ P_t=\{(x_1,\dots,x_n)\, |\, x_i\in S\setminus O_t;\ x_{i+1}=u_{x_i}(t);\ (x_1,\dots,x_n)\text{ is maximal}\}.\]
Every element of $S\setminus O_t$ is in exactly one chain of $P_t$. Let \[ P_t^e=\{(x,y)\, |\, \text{there is a chain in }P_t\text{ that ends with }x;\ u_t(x)=y\in S\}.\]
We can see that if $(x,y)\in P_t^e$, then $t'$ contains a factor $Y_{\Gamma_{< x}^v}Y_{\Gamma_{< y}^v}$. Further, if
\[
F_t=\{x\, |\, (x,y)\in P_t^e\text{ for some }y\} \qquad \text{and} \qquad L_t=\{y\, |\, (x,y)\in P_t^e\text{ for some }x\},
\] then we have
\[ t'=\bigg(\prod_{x\in F_t\cup L_t}Y_{\Gamma_{< x}^v}\bigg)\bigg(\prod_{x\in O_t\setminus L_t}Y_{I_x} \bigg)\bigg(\prod_{x\in S\setminus O_t\setminus F_t}Y_{\Gamma_{< x}^v\setminus\{u_t(x)\}} \bigg),\]
because the pairs in $P_t^e$ are disjoint by construction. This means $t'$ is a term in the product $\big(\prod_{x\in F_t\cup L_t}Y_{\Gamma_{< x}^v}\big)\big(\prod_{x\in S\setminus F_t\setminus L_t} f(x) \big)$.

Recall that for $A\in A(n)$,
\[ \prod_{x\in (S\setminus A')}Y_{\{x\}}=\frac{\prod_{x\in (S\setminus A')}f(x)}{\prod_{x\in (S\setminus A')}Y_{\Gamma_{< x}^v}}=\frac{\left(\prod_{S\in A'} Y_{\Gamma_{< x}^v}\right)\prod_{x\in (S\setminus A')}f(x)}{\prod_{x\in S}Y_{\Gamma_{< x}^v}}.\]
Since the numerator of this fraction contains $\prod_{S\in A'} Y_{\Gamma_{< x}^v}$, the term $t$ appears in $\prod_{x\in (S\setminus A')}Y_{\{x\}}$ exactly when $A'\subseteq F_t\cup L_t$. Since $F_t\cup L_t$ uniquely determines $P_t^e$, the term $t$ appears in $\prod_{x\in (S\setminus A')}Y_{\{x\}}$ exactly when $A\subseteq P_t^e$. Thus if $|P_t^e|=m$, when we count the number of times this~$t$ appears with sign, we get ${m \choose 0}-{m \choose 1}+\dots+(-1)^{m}{m \choose m}$. This is 0 if $m>0$ and 1 if $m=0$.

This means we want to sum over all possible choices of $O_t$ and $u_t$ that give rise to $P_t^e=\varnothing$. We have $P_t^e=\varnothing$, exactly when every chain in $P_t$ ends in an element $x$ such that $u_t(x)\not\in S$. Given $O_t$, this is equivalent to $u_t(x)\not\in O_t$ for every $x\in S\setminus O_t$. Summing over all subsets of $S$ that contain every minimal element and then all functions $u$ that meet these conditions gives us our formula.
\end{proof}

\begin{Example}Consider the following graph $\Gamma$ rooted at 1:
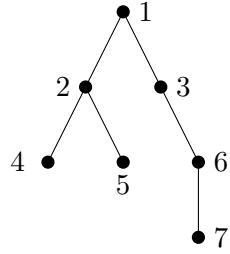
\begin{center}
\begin{tikzpicture}
\node at (0.3,0) {1};
\node at (-0.8,-1) {2};
\node at (0.8,-1) {3};
\node at (-1.4,-2) {4};
\node at (0,-2.3) {5};
\node at (1.3,-2) {6};
\node at (1.3,-3) {7};
\draw [line width=0.25mm, fill=black] (0,0) circle (0.75mm);
\draw [line width=0.25mm, fill=black] (-0.5,-1) circle (0.75mm);
\draw [line width=0.25mm, fill=black] (0.5,-1) circle (0.75mm);
\draw [line width=0.25mm, fill=black] (-1,-2) circle (0.75mm);
\draw [line width=0.25mm, fill=black] (0,-2) circle (0.75mm);
\draw [line width=0.25mm, fill=black] (1,-2) circle (0.75mm);
\draw [line width=0.25mm, fill=black] (1,-3) circle (0.75mm);
\draw (0,0) -- (-0.5,-1)
(0,0) -- (0.5,-1)
(-0.5,-1) -- (-1,-2)
(-0.5,-1) -- (0,-2)
(0.5,-1) -- (1,-2)
(1,-2) -- (1,-3);
\end{tikzpicture}
\end{center}

Let $S=\{1,3\}$. The possibilities for $O\subset S$ such that $O$ contains all leaves in $S$ are $\varnothing$,~$\{1\}$,~$\{3\}$, and $\{1,3\}$, as $S$ does not contain any minimal elements of the rooted tree. First consider when $O=\varnothing$. In this case, the domain of $u$ is $\{1,3\}$. Since 3 covers only 6, $u(3)=6$. Both~2 and~3 are covered by 1 and are not in $O$, so we could have $u(1)=2$ or $u(1)=3$. This gives us that $O=\varnothing$ contributes
\[ \frac{Y_{\Gamma_{< 1}^1\setminus\Gamma_{\geq 2}^1}Y_{\Gamma_{< 3}^1\setminus\Gamma_{\geq 6}^1}}{Y_{\Gamma_{< 1}^1}Y_{\Gamma_{< 3}^1}}+\frac{Y_{\Gamma_{< 1}^1\setminus\Gamma_{\geq 3}^1}Y_{\Gamma_{< 3}^1\setminus\Gamma_{\geq 6}^1}}{Y_{\Gamma_{< 1}^1}Y_{\Gamma_{< 3}^1}}\]
to $Y_S$.

Now we consider $O=\{1\}$. The domain of $u$ is then $\{3\}$ and $u(3)$ must be 6, as above. This means that $O=\{1\}$ contributes \[ \frac{Y_{I_1}Y_{\Gamma_{< 3}^1\setminus\Gamma_{\geq 6}^1}}{Y_{\Gamma_{< 1}^1}Y_{\Gamma_{< 3}^1}}.\]
If $O=\{3\}$ then the domain of $u$ is $\{1\}$. Although 1 covers both 2 and 3, $3\in O$, and therefore~$u(1)$ has to be~2. This means that $O=\{3\}$ contributes
\[ \frac{Y_{I_3}Y_{\Gamma_{< 1}^1\setminus\Gamma_{\geq 2}^1}}{Y_{\Gamma_{< 1}^1}Y_{\Gamma_{< 3}^1}}.\]
Finally, if $O=\{1,3\}$ then the domain of $u$ is $\varnothing$ and we get that $O=\{1,3\}$ contributes \[ \frac{Y_{I_1}Y_{I_3}}{Y_{\Gamma_{< 1}^1}Y_{\Gamma_{< 3}^1}}.\]

All together, this gives us
\begin{align*}
Y_S&=\frac{Y_{\Gamma_{< 1}^1\setminus\Gamma_{\geq 2}^1}Y_{\Gamma_{< 3}^1\setminus\Gamma_{\geq 6}^1}+Y_{\Gamma_{< 1}^1\setminus\Gamma_{\geq 3}^1}Y_{\Gamma_{< 3}^1\setminus\Gamma_{\geq 6}^1}+Y_{I_1}Y_{\Gamma_{< 3}^1\setminus\Gamma_{\geq 6}^1}+Y_{I_3}Y_{\Gamma_{< 1}^1\setminus\Gamma_{\geq 2}^1}+Y_{I_1}Y_{I_3}}{Y_{\Gamma_{< 1}^1}Y_{\Gamma_{< 3}^1}}\\
&=\frac{Y_{367}Y_4Y_5Y_7+Y_{245}Y_{67}Y_7+Y_{1234567}Y_7+Y_{367}Y_{367}Y_4Y_5+Y_{1234567}Y_{367}}{Y_{245}Y_{367}Y_{67}}.
\end{align*}
\end{Example}

As with the $X$-variables, we get positivity as a corollary.

\begin{Corollary}\label{cor:Y-pos}
Let $\Gamma$ be a tree and $\mathcal{C}$ a rooted cluster for $\Gamma$. For any set $S$ of vertices of $\Gamma$, $Y_S$ can be expressed as a Laurent polynomial in the elements of $\mathcal{C}$ with positive coefficients.
\end{Corollary}

Theorem~\ref{thm:main1} is a direct consequence of Corollaries~\ref{cor:X-pos} and~\ref{cor:Y-pos}.

\section[Hyper T-paths]{Hyper $\boldsymbol{T}$-paths}\label{sec:T-path}

We now develop the tools to prove a combinatorial formula for the $Y$-variables and use that to reprove Corollary~\ref{cor:Y-pos}. Our constructions are a generalization of \emph{$T$-paths}.
The notion of a $T$-path was introduced by Schiffler in \cite{S-08}. Later, Schiffler and Thomas \cite{S-10, ST-09} extended this work to give cluster expansion formulas for cluster algebras from unpunctured surfaces. Subsequently, Gunawan and Musiker~\cite{GM-15} used $T$-paths to give expansion formulas for type $D$ cluster algebras and to prove that in these algebras the cluster monomials form the atomic basis. More recently, a generalization of $T$-paths called \emph{super $T$-paths} were introduced by Ovenhouse, Musiker, and the last author in the context of supersymmetric cluster algebras~\cite{MOZ-21}.

In Section~\ref{sec:T-path-bg}, we review the original $T$-path construction for type $A$ cluster algebras. We then introduce our generalization of the $T$-path construction to \emph{hyper $T$-paths}
and give some motivating examples in Section~\ref{sec:gen-T-paths}. Section~\ref{sec:T-path-lems} gives some additional properties of hyper $T$-paths which are derived from our definition. In Sections~\ref{sec:single-sets}, \ref{sec:gen-sets}, and \ref{sec:pf-thm-2}, we prove that our hyper $T$-path construction gives expansion formulas for graph LP algebras in terms of rooted clusters (Theorem~\ref{thm:main2}).

\subsection[T-paths for type A cluster algebras]{$\boldsymbol{T}$-paths for type $\boldsymbol{A}$ cluster algebras}\label{sec:T-path-bg}
For the purposes of this paper, it will be most useful for us to review the definition of a \emph{complete} $T$-path.
Type $A$ cluster algebras are modeled by triangulations of an $(n+3)$-gon, with each initial seed corresponding to a unique initial triangulation. Consider an $(n+3)$-gon with vertices labeled $1, \dots, n+3$ and a fixed triangulation $T = \{ T_1, \dots, T_n, T_{n+1}, \dots, T_{2n+3} \}$ where $T_1, \dots, T_n$ are interior diagonals and $T_{n+1},\dots, T_{2n+3}$ are boundary edges. Let $i$ and $j$ be non-adjacent boundary vertices and let $M_{i,j}$ denote the interior diagonal connecting $i$ and $j$. Fix an orientation on $M_{i,j}$ and let $i = p_0, p_1, \dots, p_{d}, p_{d+1} = j$ be the ordered list of intersection points of $M_{i,j}$ and arcs of $T$. Then let $i_1, \dots, i_{d}$ be a list of indices such that intersection point $p_k$ lies on the arc $T_{i_k} \in T$.
For $k \in [d]$, let $M_k$ denote the segment of the diagonal $M_{i,j}$ between the intersection points $p_k$ and $p_{k+1}$.

In \cite{MS-09}, Musiker and Schiffler define a \emph{complete $T$-path from $i$ to $j$} as a sequence $\alpha = ( t_1,\dots,t_{\ell(\alpha)})$ such that
\begin{itemize}\itemsep=0pt
 \item[(T1)] $i = a_0,a_1, \dots,a_{\ell(\alpha)} = j$ are vertices of the $(n+3)$-gon,
 \item[(T2)] $t_k \in \alpha$ is an arc in the triangulation $T$ that connects vertices $a_{k-1}$ and $a_k$, and
 \item[(T3)] the even arcs are precisely the arcs crossed by $M_{i,j}$ in order, i.e., $t_{2k} = T_{i_k}$.
\end{itemize}
One immediate consequence of (T3) is that
\new{$\ell(\alpha)=2d+1$}.
It is possible for a complete $T$-path to contain backtracking, so there is consequently a natural notion of a \emph{reduced} $T$-path where such backtracking is removed. When we define hyper $T$-paths in Section~\ref{sec:gen-T-paths}, we will see that there are also similar notions of complete and reduced hyper $T$-paths.

\begin{Example}\label{ex:cluster_t_path}
Consider the triangulation $T$ and arc $M_{i,j}$ shown below:
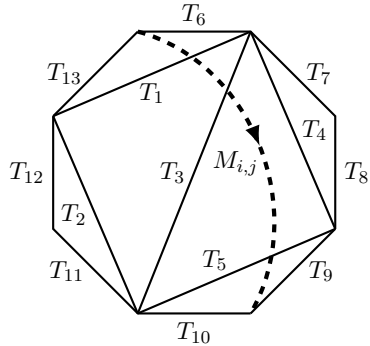
\begin{center}
\begin{tikzpicture}[scale=1.5]
 \draw[thick] (0,0) to node[scale=0.85,left,xshift=-2pt]{$T_{13}$} (0.75,0.75);
 \draw[thick] (0.75,0.75) to node[scale=0.85,above,yshift=-1pt]{$T_6$} (1.75,0.75) to node[scale=0.85,right,xshift=2pt]{$T_{7}$} (2.5,0) to node[scale=0.85,right]{$T_8$} (2.5,-1);
 \draw[thick] (2.5,-1) to node[scale=0.85,right,xshift=3pt]{$T_9$} (1.75,-1.75);
 \draw[thick] (1.75,-1.75) to node[scale=0.85,below]{$T_{10}$} (0.75,-1.75) to node[scale=0.85,left,xshift=-1pt]{$T_{11}$} (0,-1) to node[scale=0.85,left]{$T_{12}$}(0,0);

 \draw[thick] (0,0) to node[scale=0.85,below]{$T_1$} (1.75,0.75);
 \draw[thick] (1.75,0.75) to node[right,scale=0.85]{$T_4$} (2.5,-1);
 \draw[thick] (2.5,-1) to node[left,scale=0.85,yshift=5pt]{$T_5$} (0.75,-1.75);
 \draw[thick] (0.75,-1.75) to node[left,scale=0.85]{$T_2$} (0,0);
 \draw[thick] (0.75,-1.75) to node[left,scale=0.85]{$T_3$} (1.75,0.75);

 \draw[ultra thick,dashed,out=-15,in=60,->-] (0.75,0.75) to node[left,scale=0.85,yshift=-8pt,xshift=4pt]{$M_{i,j}$} (1.75,-1.75);
\end{tikzpicture}
\end{center}
Notice that the arcs crossed by $M_{i,j}$ are, in order, $T_1$, $T_3$, and $T_5$. The following is then a~complete $T$-path from $i$ to $j$. Odd arcs are depicted as blue and solid; even arcs are depicted as red and dashed:
\begin{center}
\begin{tikzpicture}[scale = 1.5]
\node at (-0.53,0) {\phantom{$a_0$}$i=$};
\node (0) at (0,0) {$a_0$};
\node (1) at (1,0) {$a_1$};
\node (2) at (2,0) {$a_2$};
\node (3) at (3,0) {$a_3$};
\node (4) at (4,0) {$a_4$};
\node (5) at (5,0) {$a_5$};
\node (6) at (6,0) {$a_6$};
\node (7) at (7,0) {$a_7$};
\node at (7.4,0) {$=j$.};
\draw[blue] (0) to node[above] {$T_{13}$} (1);
\draw[blue] (2) to node[above] {$T_{3}$} (3);
\draw[blue] (4) to node[above] {$T_{3}$} (5);
\draw[blue] (6) to node[above] {$T_{9}$} (7);
\draw[red,densely dashed] (1) to node[above] {$T_{1}$} (2);
\draw[red,densely dashed] (3) to node[above] {$T_{3}$} (4);
\draw[red,densely dashed] (5) to node[above] {$T_{5}$} (6);
\end{tikzpicture}
\end{center}
By removing the backtracking on the arc $T_3$, we obtain the following reduced $T$-path:
\begin{center}
\begin{tikzpicture}[scale = 1.5]
\node at (-0.53,0) {\phantom{$a_0$}$i=$};
\node (0) at (0,0) {$a_0$};
\node (1) at (1,0) {$a_1$};
\node (2) at (2,0) {$a_2$};
\node (3) at (3,0) {$a_3$};
\node (4) at (4,0) {$a_4$};
\node (5) at (5,0) {$a_5$};
\node at (5.4,0) {$=j$.};
\draw[blue] (0) to node[above] {$T_{13}$} (1);
\draw[red,densely dashed] (1) to node[above]{$T_{1}$} (2);
\draw[blue] (2) to node[above]{$T_{3}$} (3);
\draw[red, densely dashed] (3) to node[above]{$T_5$} (4);
\draw[blue] (4) to node[above]{$T_9$} (5);
\end{tikzpicture}
\end{center}
\end{Example}

We denote the set of all complete $T$-paths from $i$ to $j$ as $\mathcal{T}_{ij}$. Given a complete $T$-path $\alpha$, the \emph{weight} of $\alpha$ is defined to be the Laurent monomial
\[ \wt(\alpha) := \bigg( \prod_{i \textrm{ odd}} \wt(t_{i}) \bigg)\bigg( \prod_{i \textrm{ even}} \wt(t_i) \bigg)^{-1}, \]
where the weight of edge $t_i$ is given by $\wt(t_i) := x_{t_i}$. By summing over the set $\mathcal{T}_{ij}$, Schiffler \cite{S-10} then obtains an expansion formula for the cluster variable corresponding to $M_{i,j}$ in terms of the cluster seed corresponding to the triangulation $T$:
\[ x_{M_{i,j}} := \sum_{\alpha \in \mathcal{T}_{ij}} \wt(\alpha). \]
Although it is not immediately obvious, this cluster expansion formula is independent of the choice of orientation on $M_{i,j}$. This formula also holds if the summation is over all reduced $T$-paths rather than all complete $T$-paths.

\begin{Example}
The $T$-path in Example~\ref{ex:cluster_t_path} has weight \[\frac{x_{T_{13}}x_{T_3}x_{T_3}x_{T_9}}{x_{T_1}x_{T_3}x_{T_5}}=\frac{x_{T_{13}}x_{T_3}x_{T_9}}{x_{T_1}x_{T_5}}.\] In the case where the variables corresponding to boundary edges (the frozen variables) are given weight 1, the weight of the $T$-path becomes $\frac{x_{T_3}}{x_{T_1}x_{T_5}}.$ In this case, the summation over $\mathcal{T}_{ij}$ yields the cluster expansion formula
\[ x_{M_{i,j}} = \frac{x_{T_3}^2 + x_{T_3}x_{T_4} + x_{T_2}x_{T_3} + x_{T_2}x_{T_4} + x_{T_1}x_{T_5}}{x_{T_1}x_{T_3}x_{T_5}}.\]
\end{Example}

\subsection{Construction and examples}\label{sec:gen-T-paths}

In this section, we generalize the notion of $T$-paths to define \emph{hyper $T$-paths}. We first need the following construction of $\Gamma_{\mathcal{C}}$, an auxiliary graph.

Let $\Gamma$ be a tree and $\mathcal{C}$ be a rooted cluster for $\Gamma$. For each vertex $x$ of degree 1 in $\Gamma_{\mathcal{C}}$, we add an additional vertex $x'$ which is adjacent only to $x$. Call this extended graph $\Gamma'$. We will continue to think of $\Gamma'$ as a poset where $x'<x$ if $x$ is not the root and $x'>x$ if $x$ is the root. For every $S\in\mathcal{C}$, let $S'$ be the set of vertices in $\Gamma'$ that are adjacent to a vertex in $S$ but are themselves not in $S$. Add a~hyperedge labelled $S$ which connects all the vertices of $S'$. As a~convenient abuse of notation, we often refer to this hyperedge simply as $S$. We refer to this new hypergraph as~$\Gamma_{\mathcal{C}}$. See Figure~\ref{fig:example_hypergraph} for an example:

\begin{figure}[h]
 \centering
\begin{tikzpicture}
	\begin{pgfonlayer}{nodelayer}
		\node [style=circ] (0) at (0, 3) {};
		\node [style=circ] (1) at (-2, 1) {};
		\node [style=circ] (2) at (2, 1) {};
		\node [style=circ] (3) at (0, -1) {};
		\node [style=circ] (4) at (4, -1) {};
		\node [style=none] (5) at (0.25, 3.5) {$3$};
		\node [style=none] (6) at (-2.25, 1.5) {$4$};
		\node [style=none] (7) at (2.25, 1.5) {$2$};
		\node [style=none] (8) at (-0.25, -0.5) {$1$};
		\node [style=none] (9) at (4.25, -0.5) {$5$};
		\node [style=none] (10) at (0.75, -3) {$\Gamma$};
		\node [style=none] (11) at (0, 5.5) {$\color{white}yeah$};
	\end{pgfonlayer}
	\begin{pgfonlayer}{edgelayer}
		\draw (1) to (0);
		\draw (0) to (2);
		\draw (2) to (3);
		\draw (2) to (4);
	\end{pgfonlayer}
\end{tikzpicture}
\quad\quad
\begin{tikzpicture}
	\begin{pgfonlayer}{nodelayer}
		\node [style=circ] (1) at (-6, 0) {};
		\node [style=circ] (2) at (-3, 0) {};
		\node [style=circ] (3) at (0, 0) {};
		\node [style=circ] (4) at (2, 2) {};
		\node [style=circ] (5) at (2, -2) {};
		\node [style=circ] (7) at (4, 4) {};
		\node [style=circ] (8) at (4, -4) {};
		\node [style=circ] (9) at (-9, 0) {};
		\node [style=none] (10) at (4.5, 3.75) {};
		\node [style=none] (11) at (-2.5, -0.5) {};
		\node [style=none] (12) at (2.5, 1.5) {};
		\node [style=none] (13) at (3.75, 4.5) {};
		\node [style=none] (14) at (4.5, -3.5) {};
		\node [style=none] (15) at (3.25, -4) {};
		\node [style=none] (17) at (3.25, -0.25) {};
		\node [style=none] (18) at (1, 0.5) {};
		\node [style=none] (19) at (-3.25, 0.5) {};
		\node [style=none] (20) at (-0.25, 1.5) {};
		\node [style=label-2] (22) at (2.75, 0.25) {$I_2$};
		\node [style=none] (51) at (3.75, 4.5) {$5'$};
		\node [style=none] (54) at (1.75, 2.5) {$5$};
		\node [style=none] (55) at (-0.25, 0.5) {$2$};
		\node [style=none] (56) at (2.25, -1.5) {$1$};
		\node [style=none] (57) at (4.25, -3.5) {$1'$};
		\node [style=none] (61) at (-3.25, 0.5) {$3$};
		\node [style=none] (62) at (-6.25, 0.5) {$4$};
		\node [style=none] (63) at (-9.25, 0.5) {$4'$};
		\node [style=label] (64) at (1, 3) {$I_{5}$};
		\node [style=label] (65) at (-6, 1) {$I_4$};
		\node [style=label] (66) at (0.75, -3) {$I_{1}$};
		\node [style=none] (67) at (-1.75, -2.5) {$\Gamma_{\mathcal C_3}$};
		\node [style=none] (68) at (-9, 0.5) {};
		\node [style=none] (69) at (-9.25, -0.5) {};
		\node [style=none] (70) at (-0.75, -1.25) {};
		\node [style=none] (71) at (1, -2.5) {};
		\node [style=none] (72) at (3.25, 4) {};
		\node [style=none] (73) at (4.5, 3.75) {};
		\node [style=none] (74) at (4.25, -4) {};
		\node [style=none] (75) at (3.5, -4.5) {};
		\node [style=none] (79) at (1, -1.5) {};
		\node [style=none] (80) at (0.5, -1.75) {$I_3$};
	\end{pgfonlayer}
	\begin{pgfonlayer}{edgelayer}
		\draw (2) to (1);
		\draw (3) to (2);
		\draw (3) to (4);
		\draw (3) to (5);
		\draw (1) to (9);
		\draw (4) to (7);
		\draw (5) to (8);
		\draw [style=hyper1] (11.center)
			 to [in=-135, out=-150, looseness=2.00] (19.center)
			 to [in=180, out=45, looseness=0.75] (20.center)
			 to [in=-150, out=0, looseness=0.75] (12.center)
			 to [in=-150, out=30, looseness=0.75] (13.center)
			 to [in=90, out=30, looseness=1.25] (10.center)
			 to [in=90, out=-90, looseness=0.75] (17.center)
			 to [in=120, out=-90, looseness=0.75] (14.center)
			 to [in=-60, out=-60, looseness=2.25] (15.center)
			 to [in=-15, out=120] (18.center)
			 to [in=30, out=165, looseness=0.75] cycle;
		\draw [bend left, looseness=0.75] (64) to (7);
		\draw [bend right, looseness=0.75] (64) to (3);
		\draw [bend left=15] (9) to (65);
		\draw [bend left=15] (65) to (2);
		\draw [bend right, looseness=0.75] (3) to (66);
		\draw [bend right, looseness=0.75] (66) to (8);
		\draw [style=hyper2] (73.center)
			 to [in=105, out=-90, looseness=0.50] (79.center)
			 to [in=105, out=-75, looseness=0.50] (74.center)
			 to [in=-15, out=-75, looseness=1.25] (75.center)
			 to [in=-30, out=165, looseness=0.50] (71.center)
			 to [in=-15, out=150, looseness=0.50] (69.center)
			 to [in=165, out=165, looseness=2.00] (68.center)
			 to [in=-180, out=-15, looseness=0.25] (70.center)
			 to [in=-120, out=0, looseness=0.75] (72.center)
			 to [in=90, out=60, looseness=1.50] cycle;
	\end{pgfonlayer}
\end{tikzpicture}
\caption{A rooted tree and its associated hypergraph.}
\label{fig:example_hypergraph}
\end{figure}
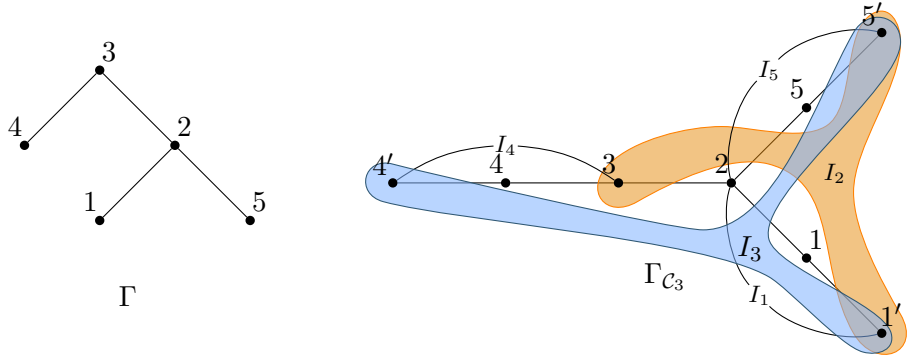

For $x$ a vertex in the rooted tree $\Gamma$, we will use $\mathcal{L}_x$ to denote to all minimal elements of $\Gamma'$ that are less than $x$. Equivalently, $\mathcal{L}_x$ is the elements of $\Gamma'\setminus\Gamma$ that are less than $x$.

\begin{Definition}\label{def:hyperTpath}
Let $S$ be a connected subset of $\Gamma$. A \emph{complete hyper $T$-path} for $S$ with respect to $\mathcal{C}$ is a set of nodes, labelled by vertices of $\Gamma_{\mathcal{C}}$, joined by connections labelled by hyperedges of $\Gamma_{\mathcal{C}}$ such that the diagram is connected and the following hold.
\begin{enumerate}\itemsep=0pt
 \item If a connection is labelled by hyperedge $e$, it joins nodes labelled by all the endpoints of $e$ with multiplicity~1.
 \item There are a distinguished set of \emph{boundary nodes} labelled by elements of $S'$ with multiplicity~1. Other nodes are called \emph{internal nodes}.
 \item Connections are specified to be \emph{even} or \emph{odd}.
 \item Boundary nodes are adjacent only to odd connections.
 \item Internal nodes labelled by elements of $S$ are adjacent to exactly one even and at least one odd connection.
 \item Internal nodes labelled by elements not in $S$ are adjacent to exactly one even and exactly one odd connection.
 \item If $x$, $y$ are below elements of $S$, any path in any complete hyper $T$-path from boundary node $x$ to boundary node $y$ uses even connections labelled, in order, by $I_x,I_{a_p},I_{a_{p-1}},\dots,I_{a_1},I_{b_1},\allowbreak I_{b_2},\dots,I_{b_q},I_y$ where the shortest path from $x$ to $y$ in $\Gamma'$ is $x,a_p,$ $a_{p-1},\dots,a_1,x\vee y,b_1,b_2,\dots,\allowbreak b_q,y$ for $p,q \geq 0$.
 \item If $x$ is below an element of $S$ and $y$ above the maximal element of $S$, any path in any complete hyper $T$-path from the boundary node $x$ to the boundary node $y$ uses even connections labelled, in order, by $I_x,I_{a_p},\ldots,I_{a_2}$, where the shortest path from $x$ to $y$ in $\Gamma'$ is $x,a_p,a_{p-1},\ldots,a_1,y$, $p \geq 1$. If $p = 1$, then a path from $x$ to $y$ uses the even connec\-tion~$I_x$.\looseness=1

 \item If $x$, $y$ are boundary nodes, where the shortest path from $x$ to $y$ in $\Gamma'$ is $x,a_p,\ldots,a_1,x\vee y$, $b_1, \ldots,b_q,y$, then any path in any complete hyper $T$-path from $x$ to $y$ uses nodes labelled by elements of $\mathcal{L}_{x\vee y}$ and $a_p,$ $a_{p-1},\dots,a_1,x\vee y,b_1,b_2,\dots,b_q$, with any multiplicity. If one of the nodes, say $y$, is adjacent to the maximal element of $S$, then $x \vee y = y$ and $q = 0$.
\end{enumerate}
\end{Definition}

When we draw complete hyper $T$-paths, we will always depict odd edges as blue/solid and even edges as red/dashed.

\begin{Example}\label{ex:hyper_T-path}
Consider the graph $\Gamma$ from Figure~\ref{fig:example_hypergraph} and cluster $\mathcal{C}_3$. Let $S$ be the set $\{2,3\}$.
The following are some examples of complete hyper $T$-paths for $S$:
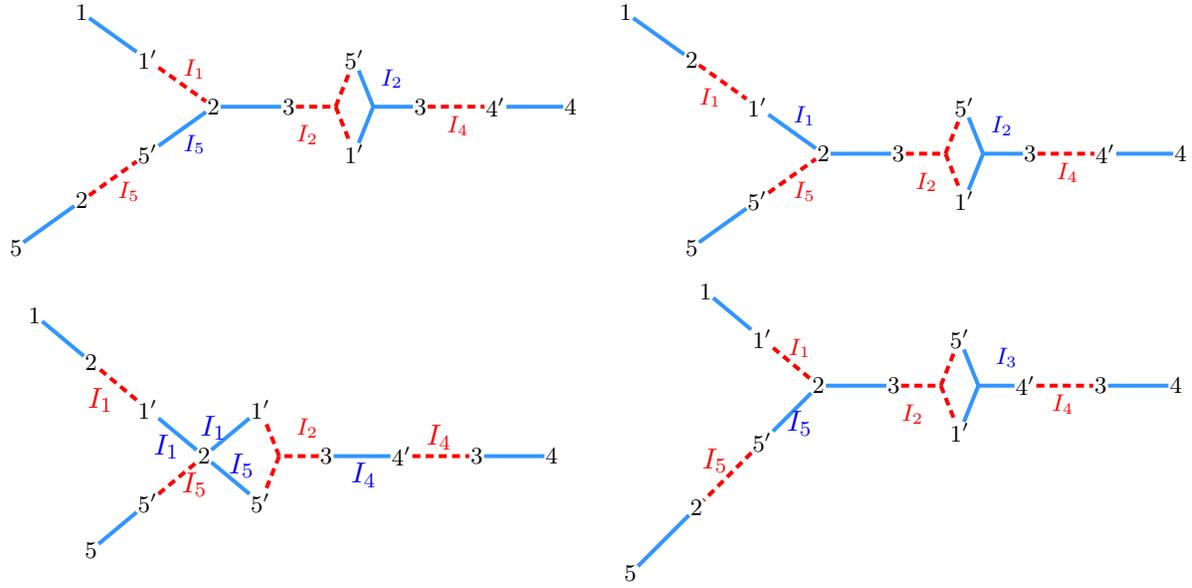
\begin{center}
\begin{tabular}{@{}cc@{}}
 \begin{tabular}{@{}c@{}}\begin{tikzpicture}
	\begin{pgfonlayer}{nodelayer}
		\node [style=label] (0) at (-7.25, 4.5) {$1$};
		\node [style=label] (1) at (-5.5, 3.25) {$1'$};
		\node [style=label] (2) at (-3.75, 2) {$2$};
		\node [style=label] (3) at (1.75, 2) {$3$};
		\node [style=label] (4) at (-5.5, 0.75) {$5'$};
		\node [style=label] (5) at (3.75, 2) {$4'$};
		\node [style=label] (6) at (5.75, 2) {$4$};
		\node [style=label] (8) at (-4.25, 3) {$\color{red}I_1$};
		\node [style=label] (9) at (2.75, 1.5) {$\color{red}I_4$};
		\node [style=label] (10) at (-7.25, -0.5) {$2$};
		\node [style=label] (11) at (-9, -1.75) {$5$};
		\node [style=label] (12) at (-6, -0.25) {$\color{red}I_5$};
		\node [style=label] (13) at (-4.25, 1) {$\color{blue}I_5$};
		\node [style=label] (14) at (0, 3.25) {$5'$};
		\node [style=label] (15) at (0, 0.75) {$1'$};
		\node [style=none] (16) at (-0.5, 2) {};
		\node [style=none] (17) at (0.5, 2) {};
		\node [style=label] (19) at (-1.25, 1.25) {$\color{red}I_2$};
		\node [style=label] (20) at (1, 2.75) {$\color{blue}I_2$};
		\node [style=label] (22) at (-1.75, 2) {$3$};
	\end{pgfonlayer}
	\begin{pgfonlayer}{edgelayer}
		\draw [style=blue] (5) to (6);
		\draw [style=blue] (2) to (4);
		\draw [style=blue] (0) to (1);
		\draw [style=red, densely dashed] (3) to (5);
		\draw [style=red, densely dashed] (1) to (2);
		\draw [style=red, densely dashed] (10) to (4);
		\draw [style=blue] (11) to (10);
		\draw [style=red, densely dashed] (16.center) to (14);
		\draw [style=red, densely dashed] (16.center) to (15);
		\draw [style=blue] (14) to (17.center);
		\draw [style=blue] (17.center) to (3);
		\draw [style=blue] (17.center) to (15);
		\draw [style=red, densely dashed, densely dashed] (22) to (16.center);
		\draw [style=blue] (2) to (22);
	\end{pgfonlayer}
\end{tikzpicture}\end{tabular} & \begin{tabular}{@{}c@{}}\begin{tikzpicture}
	\begin{pgfonlayer}{nodelayer}
		\node [style=label] (0) at (-7, -0.5) {$5$};
		\node [style=label] (1) at (-5.25, 0.75) {$5'$};
		\node [style=label] (2) at (-1.5, 2) {$3$};
		\node [style=label] (3) at (2, 2) {$3$};
		\node [style=label] (4) at (-5.25, 3.25) {$1'$};
		\node [style=label] (5) at (4, 2) {$4'$};
		\node [style=label] (6) at (6, 2) {$4$};
		\node [style=label] (8) at (-4, 1) {$\color{red}I_5$};
		\node [style=label] (9) at (3, 1.5) {$\color{red}I_4$};
		\node [style=label] (10) at (-7, 4.5) {$2$};
		\node [style=label] (11) at (-8.75, 5.75) {$1$};
		\node [style=label] (12) at (-6.5, 3.5) {$\color{red}I_1$};
		\node [style=label] (13) at (-4, 3) {$\color{blue}I_1$};
		\node [style=label] (14) at (0.25, 3.25) {$5'$};
		\node [style=label] (15) at (0.25, 0.75) {$1'$};
		\node [style=none] (16) at (-0.25, 2) {};
		\node [style=none] (17) at (0.75, 2) {};
		\node [style=label] (19) at (-0.75, 1.25) {$\color{red}I_2$};
		\node [style=label] (20) at (1.25, 2.75) {$\color{blue}I_2$};
		\node [style=label] (21) at (-3.5, 2) {$2$};
	\end{pgfonlayer}
	\begin{pgfonlayer}{edgelayer}
		\draw [style=blue] (5) to (6);
		\draw [style=blue] (0) to (1);
		\draw [style=red, densely dashed] (3) to (5);
		\draw [style=red, densely dashed] (10) to (4);
		\draw [style=blue] (11) to (10);
		\draw [style=red, densely dashed] (2) to (16.center);
		\draw [style=red, densely dashed] (16.center) to (14);
		\draw [style=red, densely dashed] (16.center) to (15);
		\draw [style=blue] (14) to (17.center);
		\draw [style=blue] (17.center) to (3);
		\draw [style=blue] (17.center) to (15);
		\draw [style=blue] (21) to (2);
		\draw [style=blue] (4) to (21);
		\draw [style=red, densely dashed] (1) to (21);
	\end{pgfonlayer}
\end{tikzpicture}\end{tabular} \\
 &\\[-0.3em]
 \begin{tabular}{@{}c@{}}\begin{tikzpicture}[rotate=180]
	\begin{pgfonlayer}{nodelayer}
		\node [style=label] (2) at (-3.75, 2) {$4'$};
		\node [style=label] (6) at (3, 3.25) {$5'$};
		\node [style=label] (14) at (0, 3.25) {$5'$};
		\node [style=label] (15) at (0, 0.75) {$1'$};
		\node [style=none] (16) at (-0.5, 2) {};
		\node [style=label] (17) at (1.5, 2) {$2$};
		\node [style=label] (19) at (-1.25, 1.25) {$\color{red}I_2$};
		\node [style=label] (22) at (-1.75, 2) {$3$};
		\node [style=label] (23) at (4.5, 4.5) {$5$};
		\node [style=label] (24) at (3, 0.75) {$1'$};
		\node [style=none] (25) at (0.5, 2.25) {$\color{blue}I_5$};
		\node [style=none] (27) at (1.75, 2.75) {$\color{red}I_5$};
		\node [style=label] (28) at (4.5, -0.5) {$2$};
		\node [style=label] (29) at (6, -1.75) {$1$};
		\node [style=none] (30) at (4.25, 0.5) {$\color{red}I_1$};
		\node [style=none] (31) at (2.5, 1.75) {$\color{blue}I_1$};
		\node [style=label] (32) at (-5.75, 2) {$3$};
		\node [style=label] (33) at (-7.75, 2) {$4$};
		\node [style=none] (34) at (-2.75, 2.5) {$\color{blue}I_4$};
		\node [style=none] (36) at (-4.75, 1.5) {$\color{red}I_4$};
		\node [style=none] (37) at (1.25, 1.25) {$\color{blue}I_1$};
	\end{pgfonlayer}
	\begin{pgfonlayer}{edgelayer}
		\draw [style=red, densely dashed] (16.center) to (14);
		\draw [style=red, densely dashed] (16.center) to (15);
		\draw [style=blue] (14) to (17);
		\draw [style=blue] (17) to (15);
		\draw [style=red, densely dashed] (22) to (16.center);
		\draw [style=blue] (2) to (22);
		\draw [style=red, densely dashed] (6) to (17);
		\draw [style=blue] (17) to (24);
		\draw [style=blue] (6) to (23);
		\draw [style=red, densely dashed] (24) to (28);
		\draw [style=blue] (28) to (29);
		\draw [style=red, densely dashed] (2) to (32);
		\draw [style=blue] (32) to (33);
	\end{pgfonlayer}
\end{tikzpicture}\end{tabular} & \begin{tabular}{@{}c@{}}\begin{tikzpicture}
	\begin{pgfonlayer}{nodelayer}
		\node [style=label] (2) at (-3.75, 2) {$2$};
		\node [style=label] (14) at (0, 3.25) {$5'$};
		\node [style=label] (15) at (0, 0.75) {$1'$};
		\node [style=none] (16) at (-0.5, 2) {};
		\node [style=label] (19) at (-1.25, 1.25) {$\color{red}I_2$};
		\node [style=label] (22) at (-1.75, 2) {$3$};
		\node [style=label] (23) at (-5.25, 3.25) {$1'$};
		\node [style=label] (26) at (-6.75, 4.5) {$1$};
		\node [style=label] (27) at (-5.25, 0.5) {$5'$};
		\node [style=label] (28) at (1.75, 2) {$4'$};
		\node [style=none] (29) at (0.5, 2) {};
		\node [style=label] (30) at (3.75, 2) {$3$};
		\node [style=label] (31) at (5.75, 2) {$4$};
		\node [style=label] (32) at (-4.25, 3) {$\color{red}I_1$};
		\node [style=label] (33) at (1.25, 2.75) {$\color{blue}I_3$};
		\node [style=label] (34) at (2.75, 1.5) {$\color{red}I_4$};
		\node [style=label] (35) at (-7, -1.25) {$2$};
		\node [style=label] (36) at (-8.75, -3) {$5$};
		\node [style=none] (37) at (-4.25, 1) {$\color{blue} I_5$};
		\node [style=none] (38) at (-6.5, 0) {$\color{red} I_5$};
	\end{pgfonlayer}
	\begin{pgfonlayer}{edgelayer}
		\draw [style=red, densely dashed] (16.center) to (14);
		\draw [style=red, densely dashed] (16.center) to (15);
		\draw [style=red, densely dashed] (22) to (16.center);
		\draw [style=blue] (2) to (22);
		\draw [style=red, densely dashed] (23) to (2);
		\draw [style=blue] (27) to (2);
		\draw [style=blue] (26) to (23);
		\draw [style=blue] (14) to (29.center);
		\draw [style=blue] (29.center) to (15);
		\draw [style=blue] (29.center) to (28);
		\draw [style=red, densely dashed] (28) to (30);
		\draw [style=blue] (30) to (31);
		\draw [style=red, densely dashed] (27) to (35);
		\draw [style=blue] (35) to (36);
	\end{pgfonlayer}
\end{tikzpicture}\end{tabular}
\end{tabular}
\end{center}
\end{Example}

\begin{Example}
The following is very similar to the first complete hyper $T$-path from the previous example:
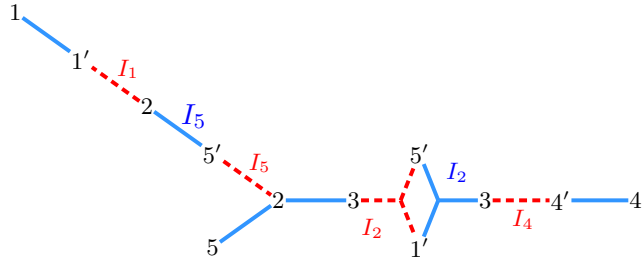
\begin{center}
\begin{tikzpicture}
	\begin{pgfonlayer}{nodelayer}
		\node [style=label] (-2) at (-10.75, 7) {$1$};
		\node [style=label] (-1) at (-9, 5.75) {$1'$};
		\node [style=label] (0) at (-7.25, 4.5) {$2$};
		\node [style=label] (1) at (-5.5, 3.25) {$5'$};
		\node [style=label] (2) at (-3.75, 2) {$2$};
		\node [style=label] (3) at (1.75, 2) {$3$};
		\node [style=label] (4) at (-5.5, 0.75) {$5$};
		\node [style=label] (5) at (3.75, 2) {$4'$};
		\node [style=label] (6) at (5.75, 2) {$4$};
		\node [style=label] (8) at (-4.25, 3) {$\color{red}I_5$};
		\node [style=label] (13) at (-7.75, 5.5) {$\color{red}I_1$};
		\node [style=label] (9) at (2.75, 1.5) {$\color{red}I_4$};
		\node [style=label] (14) at (0, 3.25) {$5'$};
		\node [style=label] (15) at (0, 0.75) {$1'$};
		\node [style=none] (16) at (-0.5, 2) {};
		\node [style=none] (17) at (0.5, 2) {};
		\node [style=label] (19) at (-1.25, 1.25) {$\color{red}I_2$};
		\node [style=label] (20) at (1, 2.75) {$\color{blue}I_2$};
		\node [style=label] (22) at (-1.75, 2) {$3$};
		\node [style=none] (23) at (-6, 4.25) {$\color{blue}I_5$};
	\end{pgfonlayer}
	\begin{pgfonlayer}{edgelayer}
		\draw [style=blue] (5) to (6);
		\draw [style=blue] (2) to (4);
		\draw [style=blue] (0) to (1);
		\draw [style=red, densely dashed] (3) to (5);
		\draw [style=red, densely dashed] (1) to (2);
		\draw [style=red, densely dashed] (0) to (-1);
		\draw [style=blue] (-1) to (-2);
		\draw [style=red, densely dashed] (16.center) to (14);
		\draw [style=red, densely dashed] (16.center) to (15);
		\draw [style=blue] (14) to (17.center);
		\draw [style=blue] (17.center) to (3);
		\draw [style=blue] (17.center) to (15);
		\draw [style=red, densely dashed, densely dashed] (22) to (16.center);
		\draw [style=blue] (2) to (22);
	\end{pgfonlayer}
\end{tikzpicture}
\end{center}
However, notice that the path from 1 to 4 uses the even edge $I_5$, which breaks Rule (7). Thus, this is not a complete hyper $T$-path.
\end{Example}

\begin{Example}
Consider the path graph on $[4]$ and cluster $\mathcal{C}_1$. Let $S = \{ 2, 3\}$. The following might initially appear to be a complete hyper $T$-path for $S$, but it actually violates Rule~(8):
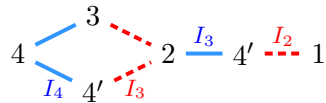
\begin{center}
 \begin{tikzpicture}
 \node (1) at (0,0) {$4$};
 \node (3) at (1,0.5) {$3$};
 \node (2) at (1,-0.5) {$4'$};
 \node (4) at (2,0) {$2$};
 \node (5) at (3,0) {$4'$};
 \node (6) at (4,0) {$1$};

 \draw[blue] (1) to (3);
 \draw[blue] (1) to node[midway,below,style=label,yshift=-2]{$\color{blue}{I_4}$} (2);
 \draw[red,densely dashed] (3) to (4);
 \draw[red, densely dashed] (2) to node[midway,below,style=label,yshift=-2,xshift=1]{$\color{red}{I_3}$} (4);
 \draw[blue] (4) to node[midway,above,style=label,yshift=2]{$\color{blue}{I_3}$} (5);
 \draw[red,densely dashed] (5) to node[above,midway,style=label,yshift=2] {$\color{red}{I_2}$}(6);
 \end{tikzpicture}
\end{center}
Because $1$ is adjacent to the maximal element of $S$ and $4$ is adjacent to the minimal element of~$S$, Rule~(8) requires that the even edges of any complete hyper $T$-path from 1 to 4 be labeled, in order, by \new{$I_{2}$} and $I_{3}$.
\end{Example}

\begin{Example}
Our definition of complete hyper $T$-paths is motivated by the definition of complete $T$-paths in the last section. If $\Gamma$ is a path graph on $[n]$, then the clusters consisting of only $Y$-variables form a type $A_{n-1}$ cluster algebra \cite[Corollary 6.2]{LP-16}. In this case, complete hyper $T$-paths are exactly complete $T$-paths.

Let $\Gamma$ be the path graph on $[6]$ with vertices numbered in order. We can construct $\Gamma_{\mathcal{C}_3}$ as follows (with edge labels omitted for clarity):
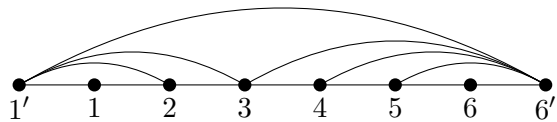
\begin{center}
\begin{tikzpicture}
\node at (0,-0.3) {$1'$};
\node at (1,-0.3) {1};
\node at (2,-0.3) {2};
\node at (3,-0.3) {3};
\node at (4,-0.3) {4};
\node at (5,-0.3) {5};
\node at (6,-0.3) {6};
\node at (7,-0.3) {$6'$};
\draw [line width=0.25mm, fill=black] (0,0) circle (0.75mm);
\draw [line width=0.25mm, fill=black] (1,0) circle (0.75mm);
\draw [line width=0.25mm, fill=black] (2,0) circle (0.75mm);
\draw [line width=0.25mm, fill=black] (3,0) circle (0.75mm);
\draw [line width=0.25mm, fill=black] (4,0) circle (0.75mm);
\draw [line width=0.25mm, fill=black] (5,0) circle (0.75mm);
\draw [line width=0.25mm, fill=black] (6,0) circle (0.75mm);
\draw [line width=0.25mm, fill=black] (7,0) circle (0.75mm);
\draw (0,0) -- (7,0)
(0,0) edge[bend left=30] (2,0)
(0,0) edge[bend left=30] (3,0)
(7,0) edge[bend right=30] (5,0)
(7,0) edge[bend right=30] (4,0)
(7,0) edge[bend right=30] (3,0)
(7,0) edge[bend right=30] (0,0);
\end{tikzpicture}
\end{center}
This can be redrawn as an octogon:
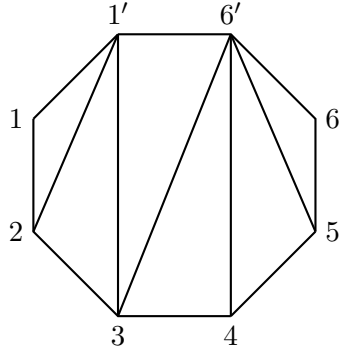
\begin{center}
\begin{tikzpicture}[scale=1.5]
 \draw[thick] (0,0) to (0.75,0.75);
 \draw[thick] (0.75,0.75) to (1.75,0.75) to (2.5,0) to (2.5,-1);
 \draw[thick] (2.5,-1) to (1.75,-1.75);
 \draw[thick] (1.75,-1.75) to (0.75,-1.75) to (0,-1) to (0,0);

 \draw[thick] (0.75,0.75) to (0,-1);
 \draw[thick] (0.75,0.75) to (0.75,-1.75);
 \draw[thick] (1.75,0.75) to (2.5,-1);
 \draw[thick] (1.75,0.75) to (1.75,-1.75);
 \draw[thick] (1.75,0.75) to (0.75,-1.75);

 \node[above] at (0.75,0.75) {$1'$};
 \node[left] at (0,0) {1};
 \node[left] at (0,-1) {2};
 \node[below] at (0.75,-1.75) {3};
 \node[below] at (1.75,-1.75) {4};
 \node[right] at (2.5,-1) {5};
 \node[right] at (2.5,0) {6};
 \node[above] at (1.75,0.75) {$6'$};
\end{tikzpicture}
\end{center}
Complete hyper $T$-paths on this octagon are the same as complete $T$-paths.
\end{Example}

\begin{Remark}
 Rules (7) and (8) can be relaxed to instead require even connections to be labelled by a subset of the listed hyperedges, arranged in the order consistent with the complete list. Doing so gives us a definition for the more general object of (not necessarily complete) \emph{hyper $T$-paths}. In particular, we can define a \emph{reduced hyper $T$-path} as a hyper $T$-path where there are no internal nodes of degree two where both adjacent connections are labelled by the same hyperedge. Although we state and prove Theorem~\ref{thm:main2} with complete hyper $T$-paths, it also holds for reduced hyper $T$-paths.
\end{Remark}

For the remainder of this paper, we focus solely on complete hyper $T$-paths. For readability, we will drop the prefix ``complete'' and simply refer to complete hyper $T$-paths as hyper $T$-paths.

\begin{Definition}
The \emph{weight} of a hyper $T$-path $\alpha$ is \[ \wt(\alpha)=\bigg(\prod_{\text{odd connections }c} \wt(c)\bigg)\bigg(\prod_{\text{even connections }c} \wt(c)\bigg)^{-1},\]
where the weight of a connection labelled by a set $I_x$ is $Y_{I_x}$ and the weight of a connection labelled by an edge in $\Gamma'$ is~1.
\end{Definition}

\begin{Example}
The first hyper $T$-path in Example~\ref{ex:hyper_T-path} has weight \[\frac{Y_{I_5}Y_{I_2}}{Y_{I_1}Y_{I_4}Y_{I_5}Y_{I_2}}=\frac{1}{Y_{I_1}Y_{I_4}}=\frac{1}{Y_1Y_4}.\]
\end{Example}

We have now introduced all of the necessary definitions to precisely understand the statement of Theorem~\ref{thm:main2}. The rest of Section~\ref{sec:T-path} will be dedicated to proving this theorem.

\subsection[Properties of hyper T-paths]{Properties of hyper $\boldsymbol{T}$-paths}\label{sec:T-path-lems}
Before proving Theorem~\ref{thm:main2}, we highlight some properties of hyper $T$-paths which are implied by the rules stated in Definition~\ref{def:hyperTpath} and will be useful to our proof.

The following condition is immediately implied by Rule (9).

\begin{Lemma}\label{lem:internalnodes}
Let $S$ be a connected subset of $\Gamma$. Then, internal nodes in a hyper $T$-path for $S$ will always be labeled either by elements of $S$ or elements of $\mathcal{L}_y$ for some $y \in S$.
\end{Lemma}

We next prove a lemma regarding the even connections that appear in a hyper $T$-path. Note that in the statement of the lemma below, $x^+$ only exists if $x$ is not the root or if $x$ is the root and $\deg(x) = 1$.

\begin{Lemma}\label{lem:evenconnections}
Given $S$, a connected subset of $\Gamma$ endowed with rooted cluster $\mathcal{C}_v$ let $x$ be the maximal element of $S$ $($it is possible that $x = v)$. Let $\alpha$ be a hyper $T$-path for $S$. Then, for each $y \in S \cup S'$ where $y \neq x,x^+$, $\alpha$ contains exactly one even connection labeled $I_y$.
\end{Lemma}

\begin{proof}
First, we show that for every $y \in S \cup S'$ such that $y \neq x,x^+$, the set $I_y$ must label at least one even connection. First, suppose $y \in S$; then, at least one vertex $c$ such that $c<y$ must be in $S'$. Moreover, $S'$ will have at least one element, $d$, such that $d \nless y$. It follows that $c \vee d > y$. By Rules (7) and (8), a path from $c$ to $d$ in a hyper $T$-path for $S$ must use $I_y$ as an even edge.
If $y \in S'$, then a path from $y$ to any other boundary node in a hyper $T$-path must use $I_y$ as an even edge.

Now, we show that $I_y$ cannot be used to label multiple even connections in a hyper $T$-path for $S$. Suppose by way of contradiction that the label $I_y$ appeared on two even connections in a hyper $T$-path. Because all hyper $T$-paths are connected, this would require that there be two boundary nodes connected by a path containing two even connections labeled $I_y$, violating Rules (7) and (8). Thus, we can conclude that any hyper $T$-path associated to $S$ must contain exactly one even connection labeled $I_y$.
\end{proof}

It follows from Rules (7) and (8) that Lemma \ref{lem:evenconnections} describes all even connections in a hyper $T$-path.
Finally we introduce a technical lemma which we will use heavily in the proofs for the rest of this section.

\begin{Lemma}\label{lem:NoExtendedVertexEvenStep}
Let $S$ be a connected subset of vertices of $\Gamma$ and let $u$ and $v$ be distinct elements of $S'$. Let $x$ be maximal in $S$ and let $a', b' \in \mathcal{L}_x$. Then, there does not exist a path between boundary nodes labeled $u$ and $v$ which takes a step along an even connection between $a'$ and $b'$ in any hyper $T$-path $\alpha$ associated to $S$.
\end{Lemma}

\begin{proof}
We assume for sake of contradiction that such a path exists in some $T$-path $\alpha$. First, suppose that $u$ and $v$ are adjacent to minimal elements of~$S$. Let the path between~$u$ and~$v$ in~$\Gamma$ be $u = w_0, w_1,\ldots,w_n,u \vee v, z_m,\ldots,z_1,v = z_0$. In order for a path from a node labeled $u$ to a~node labeled $v$ in a hyper $T$-path associated to~$S$ to have an even connection with nodes~$a'$,~$b'$, it must be that $a' < w_i$ and $b' < w_i$ or $a' < z_i$ and $b' < z_i$ for some $i$. Assume without of loss of generality that the former is true. Note that, if $0 \leq i \leq n$ is minimal such that $a' < w_i$ and $b' < w_i$, then $a',b' \in \mathcal{L}_{w_j}$ for all $i \leq j \leq n$.

Assume first that this even step is along a connection labeled $I_{w_k}$ for $k < n$. Also assume without loss of generality that this path, oriented from the node $u$ to the node $v$, goes through~$b'$ before it goes through~$a'$. Then, we know that $a'$ must connect to some endpoints of the next even connection on this path. By Rule~(7), this next connection is labeled $I_{w_{k+1}}$. It is possible that~$a'$ also connects to other even connections. By Rule~(7) and the fact that this connection is along a path using even connections labeled $I_u, I_{w_1},\ldots,I_{w_k}$, if the unique odd connection incident to~$a'$ connects to another even connection, it must be~$I_z$ for $z \lessdot w_{k+1}$. Thus, the odd connection incident to $a'$ can only have nodes $w_{k+1}$, $w_{k+2}$, and boundary nodes in $\mathcal{L}_{w_{k+2}}$; this implies that this odd connection must be labeled $I_{w_k}$ or $I_{w_{k+1}}$. Since $b' \in \mathcal{L}_{w_k} \subseteq \mathcal{L}_{w_{k+1}}$, this odd connection incident to~$a'$ will also have~$b'$ as a node. We know this is a distinct node from the~$b'$ incident to the even connection~$I_{w_k}$ since this would create a cycle, which would create paths between the boundary nodes $u$ and $v$ which violate Rule~(7). Since the underlying graph~$\Gamma$ is acyclic, $b' \notin \mathcal{L}_{z}$. It follows that the even edge incident to this node labeled~$b'$ must be labeled~$I_{w_{k+1}}$. Thus, if we have a path with even step from~$b'$ to $a'$ along $I_{w_k}$, we will also have a path with an even step from~$b'$ to~$a'$ along $I_{w_{k+1}}$:

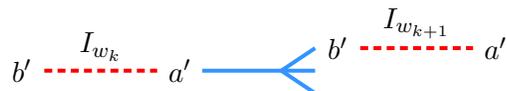
\begin{center}
\begin{tikzpicture}[scale = 1.5]
\node[] at (0,2) {$b'$};
\draw[red,densely dashed] (0.2,2) to (1.2,2);
\node[above] at (0.7,2) {$I_{w_k}$};
\node[] at (1.4,2){$a'$};
\draw[blue] (1.6,2) to (2.3,2);
\draw[blue] (2.3,2) to (2.6,2.2);
\draw[blue] (2.3,2) to (2.6,2);
\draw[blue] (2.3,2) to (2.6,1.8);
\node[] at (2.8,2.2){$b'$};
\draw[red,densely dashed] (3,2.2) to (4,2.2);
\node[] at (4.2,2.2){$a'$};
\node[above] at (3.5,2.2){$I_{w_{k+1}}$};
\end{tikzpicture}
\end{center}

It therefore suffices to consider the case when we have an even step along $I_{w_n}$ between~$b'$ and~$a'$. As before, we assume that when traveling along this path from $u$ to $v$, we visit $b'$ first. The odd connection incident to this node $a'$ can only have nodes which are incident to~$(I_{u \vee v})'$ or~$(I_z)'$ for $z \lessdot u \vee v$. Thus, the odd connection incident to $a'$ must be labeled $I_{(u \vee v)}$ or $I_{w_n}$. These odd connections will have $b'$ as a node as well. Note that $b' \in (I_{u \vee v})'$ but $b' \notin (I_z)'$ for $z \lessdot u \vee v$ and $z \neq w_n$ since our underlying graph is acyclic. Thus, either $u\vee v$ is maximal in $S$ and we have a contradiction since $b'$ cannot connect to any even connection, or $b'$ is incident to the even connection $I_{u \vee v}$:
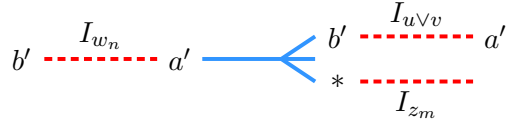
\begin{center}
\begin{tikzpicture}[scale = 1.5]
\node[] at (0,2) {$b'$};
\draw[red,densely dashed] (0.2,2) to (1.2,2);
\node[above] at (0.7,2) {$I_{w_n}$};
\node[] at (1.4,2){$a'$};
\draw[blue] (1.6,2) to (2.3,2);
\draw[blue] (2.3,2) to (2.6,2.2);
\draw[blue] (2.3,2) to (2.6,2);
\draw[blue] (2.3,2) to (2.6,1.8);
\node[] at (2.8,2.2){$b'$};
\draw[red,densely dashed] (3,2.2) to (4,2.2);
\node[] at (4.2,2.2){$a'$};
\node[above] at (3.5,2.2){$I_{u \vee v}$};
\node[] at (2.8,1.8){$*$};
\draw[red,densely dashed] (3,1.8) to (4,1.8);
\node[below] at (3.5,1.8){$I_{z_m}$};
\end{tikzpicture}
\end{center}

This logic continues; each new node $a'$ must have an incident odd connection which also has a~new node labeled $b'$. In order to preserve all rules, the node $b'$ must be incident to an even edge~$I_x$ for increasing vertices~$x$. Eventually, we will reach $x$ such that $x^+$ is maximal in~$S$. Since~$I_{x^+}$ does not appear as an even connection in a hyper $T$-path associated with $S$, we would be stuck at this point with a node~$b'$ incident to an odd connection but without any valid options for an incident even connection. Therefore, it is impossible to have a path in a hyper $T$-path between two boundary nodes which uses a step along an even connection between two extended vertices.

Now, we consider the case where $u \in S'$ is adjacent to a minimal element of $S$ and $x^+ = v \in S$ where $x$ is the maximal element of $S$. Let the path between $u$ and $v$ in $\Gamma$ be $u = w_0, w_1,\ldots, w_n, x, v = x^+$. In order for a path from a node labeled $u$ to a node labeled $v$ in a hyper $T$-path associated to $S$ to have an even connection with labels $a'$, $b'$, it must be that $a' < w_i$ and $b' < w_i$ for some $i$ so that $a',b' \in \mathcal{L}_{w_i}$. By the same argument as the previous case, if a path from the boundary node labeled $u$ to the boundary node labeled $v$ has a step along an even connection~$I_{w_k}$, $k < n$, between~$a'$ and~$b'$, then we can also find a path which has a step along an even connection $I_{w_{k+1}}$. Thus, it suffices to consider the case where the path uses a~step along an even connection labeled $I_{w_n}$ between $b'$ and $a'$; as before, suppose this path, when oriented from~$u$ to~$v$ passes through $b'$ before $a'$. By Rule (8), $I_{w_n}$ is the last even connection on this path when oriented from~$u$ to~$v$, so the node incident to \new{$I_{w_n}$} labeled $a'$ must be incident to an odd connection with a node labeled $v$. This connection must be $I_x$; accordingly, another node labeled~$b'$ will be incident to this odd connection. In order to preserve Rule (7) concerning paths from $u$ to other boundary nodes labeled by vertices of $S'$ incident to minimal elements of~$S$, this new node $b'$ must be adjacent to an even connection labeled $I_z$ for $z \lessdot x$ and $z \neq w_{n}$. However, since our underlying graph $\Gamma$ is acyclic, there is no vertex~$z$ with these properties such that $b' \in \mathcal{L}_z$. Thus, we have reached a contradiction and such a step along an even connection is not possible in this case.
\end{proof}

\subsection[Hyper T-paths for singleton sets]{Hyper $\boldsymbol{T}$-paths for singleton sets}\label{sec:single-sets}

Let $\Gamma$ be a tree with rooted cluster $\mathcal{C}_{v}$. We first describe the set of hyper $T$-paths for a singleton set $\{x\}$, where $x$ is a vertex of $\Gamma$.
If $\{x\} \in \mathcal{C}_v$ (equivalently, if $I_{x} = \{x\}$) then the only hyper $T$-path for $\{x\}$ consists of a single odd connection, labeled by $I_{x}$. Now, consider a vertex $x$ of $\Gamma$ such that $\{x\} \notin \mathcal{C}_v$; this includes the case where $x$ is the root $v$.
Let $\Gamma_{\lessdot x}^v$ be $\{c_0,\ldots,c_d\}$ where $d \geq 0$. We know this set is non-empty since $x$ is not a minimal element of $\Gamma$. If $x$ is the root of $\Gamma$ and $\deg(x) > 1$, then any hyper $T$-path associated to the set $\{x\}$ has endpoints $c_0,\ldots,c_d$. Otherwise, any hyper $T$-path associated to the set $\{x\}$ has endpoints $c_0,\ldots,c_d,x^+$.

Recall that $\mathcal{L}_{c_i}$ denotes the set of minimal elements of $\Gamma'$ which are less than $c_i$ (equivalently, the set of elements of $\Gamma' \setminus \Gamma$ that are less than $c_i$). For each $c_i$, there is one hyper $T$-path where $c_i$ is connected via a collection of odd connections to $\mathcal{L}_{c_i}$ and all other $c_j$, for $j \neq i$, are connected to $x$. We illustrate this below for $c_0$. Odd connections are shown as solid (blue) lines and even connections as dashed (red) lines:
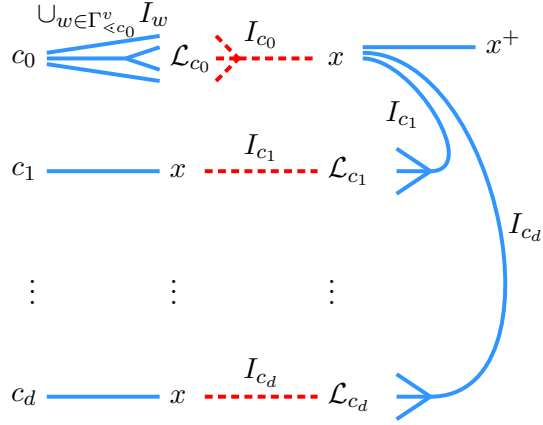
\begin{center}
\begin{tikzpicture}[scale = 1.5]
\node[left] at (0,2) {$c_0$};
\node[left] at (0,1) {$c_{1}$};
\node[left] at (0,0) {$\vdots$};
\node[left] at (0,-1) {$c_{d}$};
\draw[blue] (0,2) to (0.7,2);
\draw[blue] (0,2.05) to (1,2.2);
\draw[blue] (0.7,2) to (1,2.1);
\draw[blue] (0.7,2) to (1,1.9);
\draw[blue] (0,1.95) to (1,1.8);
\node[above] at (0.5,2.1) {$\cup_{w \in \Gamma_{\lessdot c_0}^v} I_{w}$};
\draw[blue] (0,1) to (1,1);
\draw[blue] (0,-1) to (1,-1);
\node[right] at (1,2) {$\mathcal{L}_{c_0}$};
\node[right] at (1,1) {$x$};
\node[right] at (1,0) {$\vdots$};
\node[right] at (1,-1) {$x$};
\draw[red,densely dashed] (1.5,2.2) to (1.7,2);
\draw[red,densely dashed] (1.5,2) to (1.7,2);
\draw[red,densely dashed] (1.5,1.8) to (1.7,2);
\draw[red,densely dashed] (1.7,2) to (2.4,2);
\node[above] at (1.9,2) {$I_{c_0}$};
\draw[red,densely dashed] (1.4,1) to (2.4,1);
\node[above] at (1.9,1) {$I_{c_{1}}$};
\draw[red,densely dashed] (1.4,-1) to (2.4,-1);
\node[above] at (1.9,-1) {$I_{c_{d}}$};
\node[right] at (2.4,2) {$x$};
\node[right] at (2.4,1) {$\mathcal{L}_{c_{1}}$};
\node[right] at (2.4,0) {$\vdots$};
\node[right] at (2.4,-1) {$\mathcal{L}_{c_{d}}$};
\draw[blue] (3.4,-1) to [out = 0, in = 0] (2.8,2.05);
\draw[blue] (3.1,-0.8) to (3.4,-1);
\draw[blue] (3.1,-1) to (3.4,-1);
\draw[blue] (3.1,-1.2) to (3.4,-1);
\draw[blue] (3.4,1) to [out = 0, in = 0] (2.8,2);
\draw[blue] (3.1,0.8) to (3.4,1);
\draw[blue] (3.1,1) to (3.4,1);
\draw[blue] (3.1,1.2) to (3.4,1);
\draw[blue] (2.8,2.1) to (3.8,2.1);
\node[right] at (3.8,2.15) {$x^+$};
\node[left] at (3.4,1.5){$I_{c_1}$};
\node[right] at (4,0.5) {$I_{c_d}$};
\end{tikzpicture}
\end{center}

Notice that for each element $w$ covered by $c_0$, there is one odd connection labeled $I_w$ with nodes $c_0$ and $\mathcal{L}_{w}$. The collection of all such connections will connect $c_0$ to all of the vertices of $\mathcal{L}_{c_0}$.
If $x^+$ does not exist, this hyper $T$-path can be updated by simply deleting that boundary node and its incident odd connection.
This hyper $T$-path satisfies the rules given in Definition~\ref{def:hyperTpath}. When $c_i$ is the distinguished boundary node, we refer to this hyper $T$-path as $T_{x}^{c_i}$.

There is one additional hyper $T$-path, shown below, where none of the $c_i$ boundary nodes connect to $\mathcal{L}_{c_i}$ via an odd connection:

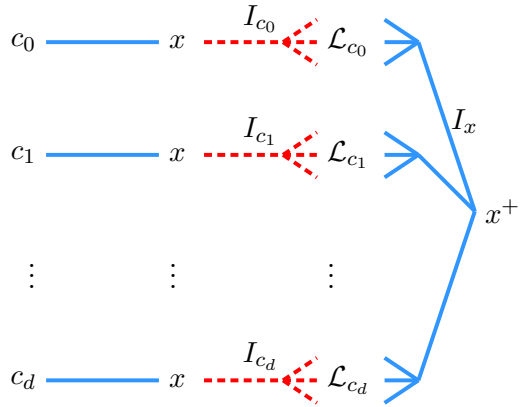
\begin{center}
\begin{tikzpicture}[scale = 1.5]
\node[left] at (0,2) {$c_0$};
\node[left] at (0,1) {$c_{1}$};
\node[left] at (0,0) {$\vdots$};
\node[left] at (0,-1) {$c_{d}$};
\draw[blue] (0,2) to (1,2);
\draw[blue] (0,1) to (1,1);
\draw[blue] (0,-1) to (1,-1);
\node[right] at (1,2) {$x$};
\node[right] at (1,1) {$x$};
\node[right] at (1,-1) {$x$};
\draw[red,densely dashed] (1.4,2) to (2.1,2);
\draw[red,densely dashed] (2.1,2) to (2.4,2);
\draw[red,densely dashed] (2.1,2) to (2.4,2.2);
\draw[red,densely dashed] (2.1,2) to (2.4,1.8);
\node[above] at (1.9,2) {$I_{c_0}$};
\draw[red,densely dashed] (1.4,1) to (2.1,1);
\draw[red,densely dashed] (2.1,1) to (2.4,1);
\draw[red,densely dashed] (2.1,1) to (2.4,1.2);
\draw[red,densely dashed] (2.1,1) to (2.4,0.8);
\node[above] at (1.9,1) {$I_{c_{1}}$};
\draw[red,densely dashed] (1.4,-1) to (2.1,-1);
\draw[red,densely dashed] (2.1,-1) to (2.4,-1);
\draw[red,densely dashed] (2.1,-1) to (2.4,-0.8);
\draw[red,densely dashed] (2.1,-1) to (2.4,-1.2);
\node[above] at (1.9,-1) {$I_{c_{d}}$};
\node[right] at (2.4,2) {$\mathcal{L}_{c_0}$};
\node[right] at (2.4,1) {$\mathcal{L}_{c_{1}}$};
\node[right] at (2.4,-1) {$\mathcal{L}_{c_{d}}$};
\draw[blue] (3.3,-1) to (3.8,0.5);
\draw[blue] (3,-0.8) to (3.3,-1);
\draw[blue] (3,-1) to (3.3,-1);
\draw[blue] (3,-1.2) to (3.3,-1);
\draw[blue] (3.3,1) to (3.8,0.5);
\draw[blue] (3,0.8) to (3.3,1);
\draw[blue] (3,1) to (3.3,1);
\draw[blue] (3,1.2) to (3.3,1);
\draw[blue] (3,1.8) to (3.3,2);
\draw[blue] (3,2) to (3.3,2);
\draw[blue] (3,2.2) to (3.3,2);
\draw[blue] (3.3,2) to (3.8,0.5);
\node[right] at (3.8,0.5) {$x^+$};
\node[right] at (3.5,1.3){$I_{x}$};
\node[right] at (1,0) {$\vdots$};
\node[right] at (2.4,0) {$\vdots$};
\end{tikzpicture}
\end{center}

If $x^+$ does not exist, then $x$ is the root and $I_x$ is the entire vertex set of $\Gamma$; therefore, the connection labeled $I_x$ has nodes $\mathcal{L}_{c_0} \sqcup \cdots \sqcup \mathcal{L}_{c_d}$. One can check that this hyper $T$-path also satisfies the rules of Definition~\ref{def:hyperTpath}. We use $T_x^{+}$ to denote this hyper $T$-path.

We now prove that this list of hyper $T$-paths is in fact a complete list for a singleton set. This will provide the base case of Theorem \ref{thm:CanPeelSingleton}, which describes all hyper $T$-paths for any connected set~$S$.

\begin{Theorem}\label{thm:TPathOneVertex}
Let $\Gamma$ be a tree with rooted cluster $\mathcal{C}_v$ and $y$ be an arbitrary vertex in $\Gamma$ with $\Gamma^{v}_{\lessdot y} = \{ c_0, \dots, c_d \}$. Then, with respect to $\mathcal{C}_v$, the collection of hyper $T$-paths for $\{ y \}$ consists exactly of $T_y^+$ and $T_{y}^{c_i}$ for each $c_i \in \Gamma^{v}_{\lessdot y}$.
\end{Theorem}
\begin{proof}
By Lemma~\ref{lem:internalnodes}, the internal nodes in a hyper $T$-path for $\{y\}$ can only be labelled by $y$ or elements of $\mathcal{L}_y$. Accordingly, the odd connections adjacent to the each boundary node $c_i$ can connect to nodes labeled by $y$ or by vertices in $\mathcal{L}_{c_i}$. By Rules (7) and (8), a path in the hyper $T$-path from $c_i$ to any other boundary node must first use the even connection $I_{c_i}$. In addition, by Lemma \ref{lem:NoExtendedVertexEvenStep}, we cannot have a path from $c_i$ to any other boundary node which uses an even step between two extended vertices. If $c_i$ connected to a set of nodes labeled by a~proper subset of $\mathcal{L}_{c_i}$, then since the even connection incident to these nodes must be labeled~$I_{c_i}$ and have all vertices from $\mathcal{L}_{c_i}$ as labels of nodes, we would create a path which Lemma \ref{lem:NoExtendedVertexEvenStep} forbids. Thus, if $c_i$ does not connect to a node labeled~$y$, it must connect to a collection of nodes labeled by all vertices in $\mathcal{L}_{c_i}$. To determine the set of valid hyper $T$-paths, we will consider how many of the nodes labeled by boundary vertices $c_i$ can connect directly to nodes labeled by~$y$.\looseness=1

First, consider the case where all of the nodes labeled by boundary vertices $c_i$ connect via an odd connection to nodes labeled by $y$. By Rules (7) and (8), this determines the label of the subsequent even connection in each branch and leaves us with the following configuration:\looseness=1

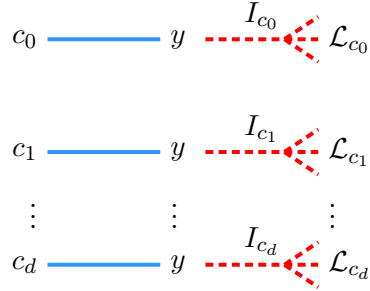
\begin{center}
\begin{tikzpicture}[scale = 1.5]
\node[left] at (0,2) {$c_0$};
\node[left] at (0,1) {$c_{1}$};
\node[left] at (0,0.5) {$\vdots$};
\node[left] at (0,0) {$c_{d}$};
\draw[blue] (0,2) to (1,2);
\draw[blue] (0,1) to (1,1);
\draw[blue] (0,0) to (1,0);
\node[right] at (1,2) {$y$};
\node[right] at (1,1) {$y$};
\node[right] at (1,0) {$y$};
\draw[red,densely dashed] (1.4,2) to (2.1,2);
\draw[red,densely dashed] (2.1,2) to (2.4,2);
\draw[red,densely dashed] (2.1,2) to (2.4,2.2);
\draw[red,densely dashed] (2.1,2) to (2.4,1.8);
\node[above] at (1.9,2) {$I_{c_0}$};
\draw[red,densely dashed] (1.4,1) to (2.1,1);
\draw[red,densely dashed] (2.1,1) to (2.4,1);
\draw[red,densely dashed] (2.1,1) to (2.4,1.2);
\draw[red,densely dashed] (2.1,1) to (2.4,0.8);
\node[above] at (1.9,1) {$I_{c_{1}}$};
\draw[red,densely dashed] (1.4,0) to (2.1,0);
\draw[red,densely dashed] (2.1,0) to (2.4,0);
\draw[red,densely dashed] (2.1,0) to (2.4,0.2);
\draw[red,densely dashed] (2.1,0) to (2.4,-0.2);
\node[above] at (1.9,0) {$I_{c_{d}}$};
\node[right] at (2.4,2) {$\mathcal{L}_{c_0}$};
\node[right] at (2.4,1) {$\mathcal{L}_{c_{1}}$};
\node[right] at (2.4,0) {$\mathcal{L}_{c_{d}}$};
\node[right] at (1,0.5) {$\vdots$};
\node[right] at (2.4,0.5) {$\vdots$};
\end{tikzpicture}
\end{center}

Because the sets $I_{c_0},\ldots,I_{c_d}$ are a complete set of the elements of the cluster $\mathcal{C}_v$ which are incompatible with \new{$\{y\}$}, by Lemma \ref{lem:evenconnections} we can no longer add even connections. This also means we cannot introduce more internal nodes, since each must have an incident even connection. All of the nodes in $\mathcal{L}_{c_0} \new{\sqcup \cdots \sqcup} \mathcal{L}_{c_d}$ can be connected by odd connections labeled either $I_y$ or $I_u$ where $u > y$. If we use $I_u$, then we would have to introduce the node $u^+$. Because $u^{+}$ is not in either~$S$ or~$S'$, however, it cannot label a node. Therefore, we must use the connection~$I_y$. If $y$ is not the root of $\Gamma$ or if $y$ is the root and $\textrm{deg}(y) = 1$, this connection will also have node~$y^+$. If $y$ is the root and $\deg(y) > 1$, then $S' = \Gamma_{\lessdot y}^v$, and a connection labeled $I_y$ connects the nodes labeled by all vertices in $\mathcal{L}_{c_0} \new{\sqcup \cdots \sqcup} \mathcal{L}_{c_d}$ to each other. Thus, this hyper $T$-path must be~$T_y^{+}$.

Next, we consider the case where all but one of the boundary nodes are connected to $y$. Without loss of generality, let $c_0$ be the unique boundary node which instead is connected via odd connections to the set $\mathcal{L}_{c_0}$. The following even connections on each branch are again forced by Rules~(7) and~(8):

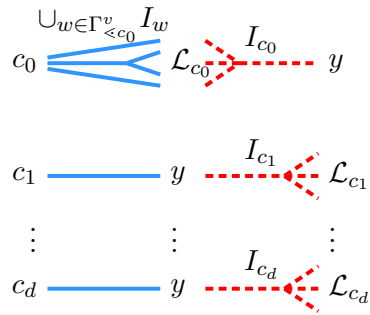
\begin{center}
\begin{tikzpicture}[scale = 1.5]
\node[left] at (0,2) {$c_0$};
\node[left] at (0,1) {$c_{1}$};
\node[left] at (0,0.5) {$\vdots$};
\node[left] at (0,0) {$c_{d}$};
\draw[blue] (0,2) to (0.7,2);
\draw[blue] (0,2.05) to (1,2.2);
\draw[blue] (0.7,2) to (1,2.1);
\draw[blue] (0.7,2) to (1,1.9);
\draw[blue] (0,1.95) to (1,1.8);
\node[above] at (0.5,2.1) {$\cup_{w \in \Gamma_{\lessdot c_0}^v} I_{w}$};
\draw[blue] (0,1) to (1,1);
\draw[blue] (0,0) to (1,0);
\node[right] at (1,2) {$\mathcal{L}_{c_0}$};
\node[right] at (1,1) {$y$};
\node[right] at (1,0.5) {$\vdots$};
\node[right] at (1,0) {$y$};
\draw[red,densely dashed] (1.4,2.2) to (1.7,2);
\draw[red,densely dashed] (1.4,2) to (1.7,2);
\draw[red,densely dashed] (1.4,1.8) to (1.7,2);
\draw[red,densely dashed] (1.7,2) to (2.4,2);
\node[above] at (1.9,2) {$I_{c_0}$};
\draw[red,densely dashed] (1.4,1) to (2.1,1);
\draw[red,densely dashed] (2.1,1) to (2.4,0.8);
\draw[red,densely dashed] (2.1,1) to (2.4,1);
\draw[red,densely dashed] (2.1,1) to (2.4,1.2);
\node[above] at (1.9,1) {$I_{c_{1}}$};
\draw[red,densely dashed] (1.4,0) to (2.1,0);
\draw[red,densely dashed] (2.1,0) to (2.4,0.2);
\draw[red,densely dashed] (2.1,0) to (2.4,0);
\draw[red,densely dashed] (2.1,0) to (2.4,-0.2);
\node[above] at (1.9,0) {$I_{c_{d}}$};
\node[right] at (2.4,2) {$y$};
\node[right] at (2.4,1) {$\mathcal{L}_{c_{1}}$};
\node[right] at (2.4,0.5) {$\vdots$};
\node[right] at (2.4,0) {$\mathcal{L}_{c_{d}}$};
\end{tikzpicture}
\end{center}

As in the previous case, we cannot introduce any more internal nodes since we have used the complete set of allowable even connections. A connection between $\cup_{i=1}^d \mathcal{L}_{c_i}$ and $y^+$ would also have vertices in $\mathcal{L}_{c_0}$ as nodes. We can neither introduce new nodes nor, by (6), connect this new connection to the existing set of nodes $\mathcal{L}_{c_0}$. Thus, we must connect each set $\mathcal{L}_{c_i}$ individually to $y$, using a connection labeled $I_{c_i}$. That is, any hyper $T$-path with this configuration which satisfies every rule must be $T_{y}^{c_0}$.

Finally, we consider the case where more than one of the boundary nodes $c_i$ connect to~$\mathcal{L}_{c_i}$ via an odd connection. If $y$ only covers one element, then this case is not possible. Without loss of generality, suppose $c_0,\ldots,c_k$ for $1 \leq k \leq d$ connect to $\mathcal{L}_{c_0},\ldots,\mathcal{L}_{c_k}$ respectively:

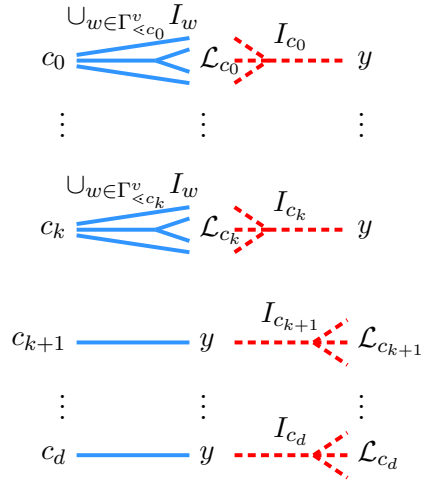
\begin{center}
\begin{tikzpicture}[scale = 1.5]
\node[left] at (0,1.5) {$c_0$};
\node[left] at (0,1) {$\vdots$};
\node[left] at (0,0) {$c_k$};
\node[left] at (0,-1) {$c_{k+1}$};
\node[left] at (0,-1.5) {$\vdots$};
\node[left] at (0,-2) {$c_d$};
\draw[blue] (0,1.5) to (0.7,1.5);
\draw[blue] (0,1.55) to (1,1.7);
\draw[blue] (0.7,1.5) to (1,1.6);
\draw[blue] (0.7,1.5) to (1,1.4);
\draw[blue] (0,1.45) to (1,1.3);
\node[above] at (0.5,1.6) {$\cup_{w \in \Gamma_{\lessdot c_0}^v} I_{w}$};
\draw[blue] (0,0) to (0.7,0);
\draw[blue] (0,0.05) to (1,0.2);
\draw[blue] (0.7,0) to (1,0.1);
\draw[blue] (0.7,0) to (1,-.1);
\draw[blue] (0,-0.05) to (1,-.2);
\node[above] at (0.5,0.1) {$\cup_{w \in \Gamma_{\lessdot c_k}^v} I_{w}$};
\draw[blue] (0,-1) to (1,-1);
\draw[blue] (0,-2) to (1,-2);
\node[right] at (1,1.5) {$\mathcal{L}_{c_0}$};
\node[right] at (1,1) {$\vdots$};
\node[right] at (1,0) {$\mathcal{L}_{c_k}$};
\node[right] at (1,-1) {$y$};
\node[right] at (1,-1.5) {$\vdots$};
\node[right] at (1,-2) {$y$};
\draw[red,densely dashed] (1.4,1.7) to (1.7,1.5);
\draw[red,densely dashed] (1.4,1.5) to (1.7,1.5);
\draw[red,densely dashed] (1.4,1.3) to (1.7,1.5);
\draw[red,densely dashed] (1.7,1.5) to (2.4,1.5);
\node[above] at (1.9,1.5) {$I_{c_0}$};
\draw[red,densely dashed] (1.4,0.2) to (1.7,0);
\draw[red,densely dashed] (1.4,0) to (1.7,0);
\draw[red,densely dashed,densely dashed] (1.4,-0.2) to (1.7,0);
\draw[red,densely dashed] (1.7,0) to (2.4,0);
\node[above] at (1.9,0) {$I_{c_k}$};
\draw[red,densely dashed] (1.4,-1) to (2.1,-1);
\draw[red,densely dashed] (2.1,-1) to (2.4,-0.8);
\draw[red,densely dashed] (2.1,-1) to (2.4,-1);
\draw[red,densely dashed] (2.1,-1) to (2.4,-1.2);
\node[above] at (1.9,-1) {$I_{c_{k+1}}$};
\draw[red,densely dashed] (1.4,-2) to (2.1,-2);
\draw[red,densely dashed] (2.1,-2) to (2.4,-1.8);
\draw[red,densely dashed] (2.1,-2) to (2.4,-2);
\draw[red,densely dashed] (2.1,-2) to (2.4,-2.2);
\node[above] at (1.9,-2) {$I_{c_{d}}$};
\node[right] at (2.4,1.5) {$y$};
\node[right] at (2.4,1) {$\vdots$};
\node[right] at (2.4,0) {$y$};
\node[right] at (2.4,-1) {$\mathcal{L}_{c_{k+1}}$};
\node[right] at (2.4,-1.5) {$\vdots$};
\node[right] at (2.4,-2) {$\mathcal{L}_{c_{d}}$};
\end{tikzpicture}
\end{center}

The same discussion as in the previous case holds. However, now there are multiple nodes labeled $y$ which need an adjacent odd connection. Each node in a set $\mathcal{L}_{c_i}$ can only be adjacent to one odd connection by Rule (6). Moreover, these nodes must connect to a node labeled $y$; otherwise we would create subpaths in the hyper $T$-path that violate Rule (9).
Thus, there is no way to connect all of these initial components while still satisfying the rules for a valid hyper $T$-path.
\end{proof}

\subsection[Hyper T-paths for general sets]{Hyper $\boldsymbol{T}$-paths for general sets}\label{sec:gen-sets}

The goal of this section is the proof of Theorem \ref{thm:CanPeelSingleton}, which will enable us to prove our second main result, Theorem~\ref{thm:main2}, in which we verify that any hyper $T$-path satisfying the rules given in Definition~\ref{def:hyperTpath} can be constructed by pasting together hyper $T$-paths for singleton sets, which we described in Section~\ref{sec:single-sets}.

The procedure for \emph{pasting} two hyper $T$-paths together is as follows. Let $u$, $v$ be vertices connected by an edge $(u,v)$ in $\Gamma$. Let $A$ and $B$ be disjoint subsets of the vertex set of $\Gamma$ such that $A$ contains $u$ but not $v$ and $B$ contains $v$ but not $u$. Then $v \in A'$ and $u \in B'$. Let~$T_A$,~$T_B$ be hyper $T$-paths for the sets $A$ and $B$, respectively. Suppose that in~$T_A$, the boundary node labeled $v$ is connected via an odd connection to a node labeled $u$ or in $T_B$ the boundary node labeled $u$ is connected via an odd connection to a node labeled $v$. By Rules~(7) and~(8), if this is not true for both hyper $T$-paths, then in one of them there are nodes labeled $u$ and $v$ connected by a sequence of an odd connection and an even connection. Then, we can paste $T_A$ and $T_B$ together as follows. Suppose without loss of generality that in $T_A$ there is an odd connection with nodes labeled~$u$ and~$v$. Then deleting that connection and the node $u$ and identifying the node $v$ in $T_A$ and $T_B$ gives us $T_A\oplus T_B$, a hyper $T$-path for $A \cup B$.

\begin{Example}
Consider the graph $\Gamma$ in Figure \ref{fig:example_hypergraph}. Let $A = \{2\}$ and $B = \{3\}$. Then, $u = 2$ and $v = 3$. Below we display a hyper $T$-path for $A$, call it $T_A$, and a hyper $T$-path for $B$, call it $T_B$. In the notation of Section~\ref{sec:single-sets}, $T_A = T_2^1$ and $T_B = T_3^4$. We can see that $T_A$ has the boundary node 3 joined to an internal node 2 by an odd edge and $T_B$ has a boundary node 2. We paste these two hyper $T$-paths together by deleting the boundary node 3 from $T_A$ and identifying the internal node 2 to which this boundary node was formerly connected with the boundary node 2 from $T_B$:
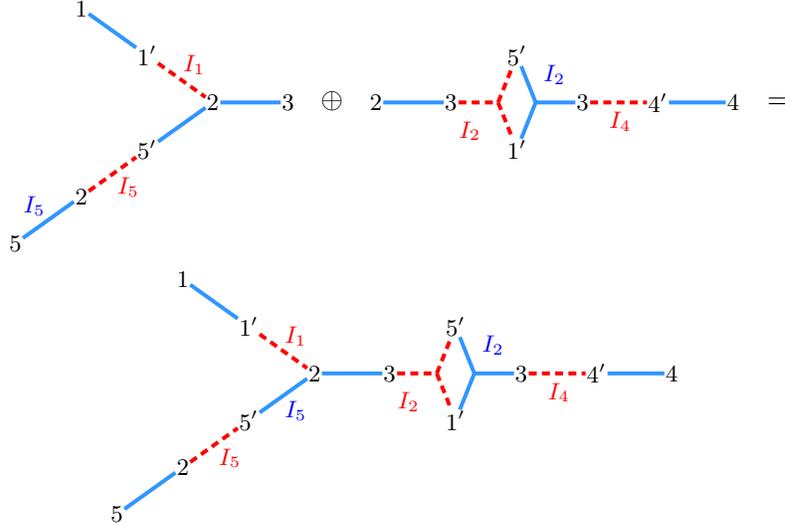
$$\raisebox{-0.4in}{\begin{tikzpicture}
	\begin{pgfonlayer}{nodelayer}
		\node [style=label] (0) at (-7.25, 4.5) {$1$};
		\node [style=label] (1) at (-5.5, 3.25) {$1'$};
		\node [style=label] (2) at (-3.75, 2) {$2$};
		\node [style=label] (4) at (-5.5, 0.75) {$5'$};
		\node [style=label] (8) at (-4.25, 3) {$\color{red}I_1$};
		\node [style=label] (10) at (-7.25, -0.5) {$2$};
		\node [style=label] (11) at (-9, -1.75) {$5$};
		\node [style=label] (12) at (-6, -0.25) {$\color{red}I_5$};
		\node [style=label] (13) at (-8.5, -0.75) {$\color{blue}I_5$};
		\node [style=label] (22) at (-1.75, 2) {$3$};
	\end{pgfonlayer}
	\begin{pgfonlayer}{edgelayer}
		\draw [style=blue] (2) to (4);
		\draw [style=blue] (0) to (1);
		\draw [style=red, densely dashed] (1) to (2);
		\draw [style=red, densely dashed] (10) to (4);
		\draw [style=blue] (11) to (10);
		\draw [style=blue] (2) to (22);
	\end{pgfonlayer}
\end{tikzpicture}}\ \oplus \ \raisebox{-0.4in}{\begin{tikzpicture}
	\begin{pgfonlayer}{nodelayer}
		\node [style=label] (2) at (-3.75, 2) {$2$};
		\node [style=label] (3) at (1.75, 2) {$3$};
		\node [style=label] (5) at (3.75, 2) {$4'$};
		\node [style=label] (6) at (5.75, 2) {$4$};
		\node [style=label] (9) at (2.75, 1.5) {$\color{red}I_4$};
		\node [style=label] (14) at (0, 3.25) {$5'$};
		\node [style=label] (15) at (0, 0.75) {$1'$};
		\node [style=none] (16) at (-0.5, 2) {};
		\node [style=none] (17) at (0.5, 2) {};
		\node [style=label] (19) at (-1.25, 1.25) {$\color{red}I_2$};
		\node [style=label] (20) at (1, 2.75) {$\color{blue}I_2$};
		\node [style=label] (22) at (-1.75, 2) {$3$};
	\end{pgfonlayer}
	\begin{pgfonlayer}{edgelayer}
		\draw [style=blue] (5) to (6);
		\draw [style=red, densely dashed] (3) to (5);
		\draw [style=red, densely dashed] (16.center) to (14);
		\draw [style=red, densely dashed] (16.center) to (15);
		\draw [style=blue] (14) to (17.center);
		\draw [style=blue] (17.center) to (3);
		\draw [style=blue] (17.center) to (15);
		\draw [style=red, densely dashed, densely dashed] (22) to (16.center);
		\draw [style=blue] (2) to (22);
	\end{pgfonlayer}
\end{tikzpicture}} \ =$$
$$\begin{tikzpicture}
	\begin{pgfonlayer}{nodelayer}
		\node [style=label] (0) at (-7.25, 4.5) {$1$};
		\node [style=label] (1) at (-5.5, 3.25) {$1'$};
		\node [style=label] (2) at (-3.75, 2) {$2$};
		\node [style=label] (3) at (1.75, 2) {$3$};
		\node [style=label] (4) at (-5.5, 0.75) {$5'$};
		\node [style=label] (5) at (3.75, 2) {$4'$};
		\node [style=label] (6) at (5.75, 2) {$4$};
		\node [style=label] (8) at (-4.25, 3) {$\color{red}I_1$};
		\node [style=label] (9) at (2.75, 1.5) {$\color{red}I_4$};
		\node [style=label] (10) at (-7.25, -0.5) {$2$};
		\node [style=label] (11) at (-9, -1.75) {$5$};
		\node [style=label] (12) at (-6, -0.25) {$\color{red}I_5$};
		\node [style=label] (13) at (-4.25, 1) {$\color{blue}I_5$};
		\node [style=label] (14) at (0, 3.25) {$5'$};
		\node [style=label] (15) at (0, 0.75) {$1'$};
		\node [style=none] (16) at (-0.5, 2) {};
		\node [style=none] (17) at (0.5, 2) {};
		\node [style=label] (19) at (-1.25, 1.25) {$\color{red}I_2$};
		\node [style=label] (20) at (1, 2.75) {$\color{blue}I_2$};
		\node [style=label] (22) at (-1.75, 2) {$3$};
	\end{pgfonlayer}
	\begin{pgfonlayer}{edgelayer}
		\draw [style=blue] (5) to (6);
		\draw [style=blue] (2) to (4);
		\draw [style=blue] (0) to (1);
		\draw [style=red, densely dashed] (3) to (5);
		\draw [style=red, densely dashed] (1) to (2);
		\draw [style=red, densely dashed] (10) to (4);
		\draw [style=blue] (11) to (10);
		\draw [style=red, densely dashed] (16.center) to (14);
		\draw [style=red, densely dashed] (16.center) to (15);
		\draw [style=blue] (14) to (17.center);
		\draw [style=blue] (17.center) to (3);
		\draw [style=blue] (17.center) to (15);
		\draw [style=red, densely dashed, densely dashed] (22) to (16.center);
		\draw [style=blue] (2) to (22);
	\end{pgfonlayer}
\end{tikzpicture}$$

In this case, since $T_B$ has a connection between the boundary node 3 and an internal node 2, we could have instead thought about this pasting as deleting the boundary node 2 from $T_B$ and identifying the formerly connected node 3 with the boundary node 3 from $T_A$. Both processes give the same result.

In this case, any pair of $T$-paths for $A $ and $B$ can be pasted except the pair $T_2^+$ and $T_3^2$.
\end{Example}

\begin{Proposition}\label{prop:pasting}
Let $A$ and $B$ be disjoint connected subsets of $\Gamma$ such that there exists an edge $(u,v)$ in $\Gamma$ where $A$ contains $u$ and not $v$ and $B$ contains $v$ but not $u$. If $T_A$ is a hyper $T$-path for $A$ and $T_B$ is a hyper $T$-path for $B$ such that $u$ and $v$ are joined by a connection in at least one of $T_A$, $T_B$, then $T_A\oplus T_B$ is a hyper $T$-path for $A \sqcup B$.
\end{Proposition}

\begin{proof}
It is easy to check that this construction follows all of the rules listed in Defini\-tion~\ref{def:hyperTpath}.\looseness=1
\end{proof}

\begin{Theorem}\label{thm:CanPeelSingleton}
Let $S$ be a connected set of vertices of the graph $\Gamma$. Let $\mathcal{C}_v$ be a rooted cluster on~$\Gamma$. Let $y$ be a minimal element of $S$; that is, there are no elements in $S$ less than $y$. Then, any hyper $T$-path associated to $S$ can be constructed by pasting a hyper $T$-path for $\{ y \}$, $T_{y}^*$, together with a hyper $T$-path for $S \backslash \{ y \}$.
\end{Theorem}

\begin{proof}
We proceed by induction on $\vert S \vert$. The $\vert S \vert = 1$ case is covered by Theorem~\ref{thm:TPathOneVertex}, so let $\vert S \vert > 1$, and suppose we have proved the statement of the theorem for all sets of size $\vert S \vert-1$. Let~$y$ be a minimal element of $S$ and let $\Gamma_{\lessdot y}^v = \{c_0,\ldots,c_d\}$ with $d \geq 0$. As in the singleton case, each boundary node $c_i$ must connect via odd connection to either~$y$ or all of $\mathcal{L}_{c_i}$. We consider the same cases as in the singleton case.

First, consider the case where all $c_i$ connect directly to $y$ via an odd connection. Rules~(7) and~(8) then determine the label attached to the even connections adjacent to each of these internal nodes labeled $y$. By Rule~(7), any subpath in the hyper $T$-path from $c_i$ to $c_j$ for $i \neq j$ must use even connections labeled $I_{c_i}$ and $I_{c_j}$, in this order. Because hyper $T$-paths must be connected, the nodes labeled by vertices in $\mathcal{L}_{c_0}, \dots, \mathcal{L}_{c_d}$ must all be joined by a series of odd connections. Moreover, the first two even connections on a path from $c_i$ to an element of $S'$ which is not labelled by an element of $\Gamma_{\lessdot y}^v$ must be $I_{c_i}$ and $I_y$. The nodes in $\mathcal{L}_{c_0} \new{\sqcup \cdots \sqcup} \mathcal{L}_{c_d}$ must then all connect to one or multiple endpoints of $I_y$. Recall that the endpoints of~$I_y$ are~$y^+$ and~$\mathcal{L}_{c_0} \new{\sqcup \cdots \sqcup} \mathcal{L}_{c_d}$ and, by Lemma~\ref{lem:evenconnections}, this hyper $T$-path can only use one even connection labeled~$I_y$. So, we must connect all the nodes in $\mathcal{L}_{c_i}$ to one node labeled~$y^+$:
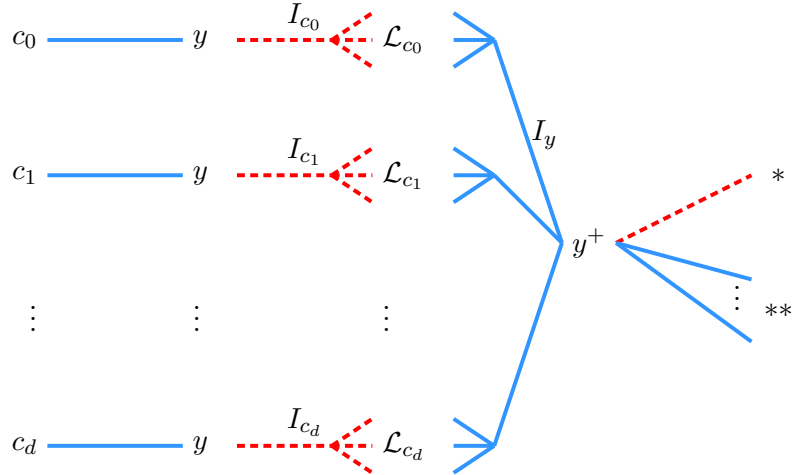
\begin{center}
\begin{tikzpicture}[scale = 1.8]
\node[left] at (0,2) {$c_0$};
\node[left] at (0,1) {$c_{1}$};
\node[left] at (0,0) {$\vdots$};
\node[left] at (0,-1) {$c_{d}$};
\draw[blue] (0,2) to (1,2);
\draw[blue] (0,1) to (1,1);
\draw[blue] (0,-1) to (1,-1);
\node[right] at (1,2) {$y$};
\node[right] at (1,1) {$y$};
\node[right] at (1,-1) {$y$};
\draw[red,densely dashed] (1.4,2) to (2.1,2);
\draw[red,densely dashed] (2.1,2) to (2.4,2);
\draw[red,densely dashed] (2.1,2) to (2.4,2.2);
\draw[red,densely dashed] (2.1,2) to (2.4,1.8);
\node[above] at (1.9,2) {$I_{c_0}$};
\draw[red,densely dashed] (1.4,1) to (2.1,1);
\draw[red,densely dashed] (2.1,1) to (2.4,1);
\draw[red,densely dashed] (2.1,1) to (2.4,1.2);
\draw[red,densely dashed] (2.1,1) to (2.4,0.8);
\node[above] at (1.9,1) {$I_{c_{1}}$};
\draw[red,densely dashed] (1.4,-1) to (2.1,-1);
\draw[red,densely dashed] (2.1,-1) to (2.4,-1);
\draw[red,densely dashed] (2.1,-1) to (2.4,-0.8);
\draw[red,densely dashed] (2.1,-1) to (2.4,-1.2);
\node[above] at (1.9,-1) {$I_{c_{d}}$};
\node[right] at (2.4,2) {$\mathcal{L}_{c_0}$};
\node[right] at (2.4,1) {$\mathcal{L}_{c_{1}}$};
\node[right] at (2.4,-1) {$\mathcal{L}_{c_{d}}$};
\node[right] at (1,0) {$\vdots$};
\node[right] at (2.4,0) {$\vdots$};
\draw[blue] (3.3,-1) to (3.8,0.5);
\draw[blue] (3,-0.8) to (3.3,-1);
\draw[blue] (3,-1) to (3.3,-1);
\draw[blue] (3,-1.2) to (3.3,-1);
\draw[blue] (3.3,1) to (3.8,0.5);
\draw[blue] (3,0.8) to (3.3,1);
\draw[blue] (3,1) to (3.3,1);
\draw[blue] (3,1.2) to (3.3,1);
\draw[blue] (3,1.8) to (3.3,2);
\draw[blue] (3,2) to (3.3,2);
\draw[blue] (3,2.2) to (3.3,2);
\draw[blue] (3.3,2) to (3.8,0.5);
\node[right] at (3.5,1.3){$I_{y}$};
\node[right] at (3.8,0.5) {$y^+$};
\draw[red, densely dashed] (4.2,0.5) to (5.2,1);
\node[] at (5.4,1) {$*$};
\draw[blue] (4.2,0.5) to (5.2,0.23);
\draw[blue] (4.2,0.5) to (5.2,-0.23);
\node[] at (5.4,0) {$**$};
\node[left] at (5.2,0.15) {$\vdots$};
\end{tikzpicture}
\end{center}

We can decompose this $T$-path as $T_{y}^{+}$ and a $T$-path for $S - \{y\}$. In the latter $T$-path, $y$ is a~boundary node and is connected via an odd connection to $y^+$:
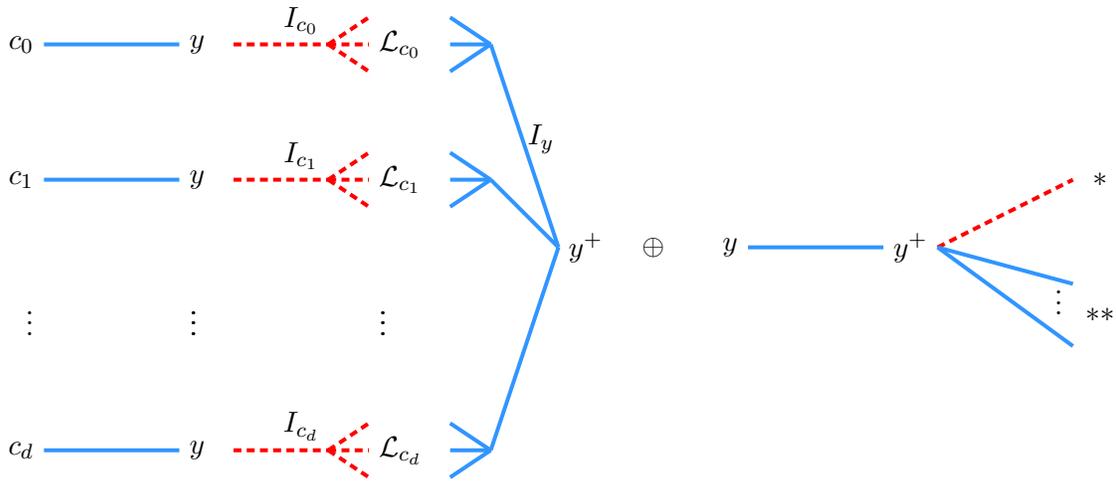
\begin{center}
\begin{tikzpicture}[scale = 1.8]
\node[left] at (0,2) {$c_0$};
\node[left] at (0,1) {$c_{1}$};
\node[left] at (0,0) {$\vdots$};
\node[left] at (0,-1) {$c_{d}$};
\draw[blue] (0,2) to (1,2);
\draw[blue] (0,1) to (1,1);
\draw[blue] (0,-1) to (1,-1);
\node[right] at (1,2) {$y$};
\node[right] at (1,1) {$y$};
\node[right] at (1,-1) {$y$};
\draw[red,densely dashed] (1.4,2) to (2.1,2);
\draw[red,densely dashed] (2.1,2) to (2.4,2);
\draw[red,densely dashed] (2.1,2) to (2.4,2.2);
\draw[red,densely dashed] (2.1,2) to (2.4,1.8);
\node[above] at (1.9,2) {$I_{c_{0}}$};
\draw[red,densely dashed] (1.4,1) to (2.1,1);
\draw[red,densely dashed] (2.1,1) to (2.4,1);
\draw[red,densely dashed] (2.1,1) to (2.4,1.2);
\draw[red,densely dashed] (2.1,1) to (2.4,0.8);
\node[above] at (1.9,1) {$I_{c_{1}}$};
\draw[red,densely dashed] (1.4,-1) to (2.1,-1);
\draw[red,densely dashed] (2.1,-1) to (2.4,-1);
\draw[red,densely dashed] (2.1,-1) to (2.4,-0.8);
\draw[red,densely dashed] (2.1,-1) to (2.4,-1.2);
\node[above] at (1.9,-1) {$I_{c_{d}}$};
\node[right] at (2.4,2) {$\mathcal{L}_{c_0}$};
\node[right] at (2.4,1) {$\mathcal{L}_{c_{1}}$};
\node[right] at (2.4,-1) {$\mathcal{L}_{c_{d}}$};
\node[right] at (1,0) {$\vdots$};
\node[right] at (2.4,0) {$\vdots$};
\draw[blue] (3.3,-1) to (3.8,0.5);
\draw[blue] (3,-0.8) to (3.3,-1);
\draw[blue] (3,-1) to (3.3,-1);
\draw[blue] (3,-1.2) to (3.3,-1);
\draw[blue] (3.3,1) to (3.8,0.5);
\draw[blue] (3,0.8) to (3.3,1);
\draw[blue] (3,1) to (3.3,1);
\draw[blue] (3,1.2) to (3.3,1);
\draw[blue] (3,1.8) to (3.3,2);
\draw[blue] (3,2) to (3.3,2);
\draw[blue] (3,2.2) to (3.3,2);
\draw[blue] (3.3,2) to (3.8,0.5);
\node[right] at (3.8,0.5) {$y^+$};
\node[right] at (3.5,1.3){$I_{y}$};
\node[right] at (1,0) {$\vdots$};
\node[right] at (2.4,0) {$\vdots$};
\node[] at (4.5,0.5){$\oplus$};
\node[left] at (5.2,0.5) {$y$};
\draw[blue] (5.2,0.5) -- (6.2,0.5);
\node[right] at (6.2,0.5) {$y^+$};
\draw[red, densely dashed] (6.6,0.5) to (7.6,1);
\node[] at (7.8,1) {$*$};
\draw[blue] (6.6,0.5) to (7.6,0.23);
\draw[blue] (6.6,0.5) to (7.6,-0.23);
\node[] at (7.8,0) {$**$};
\node[left] at (7.6,0.15) {$\vdots$};
\end{tikzpicture}
\end{center}

Now, we consider the case where exactly one boundary node, call it $c_0$, is not connected to~$y$ by an odd connection. Recall from the proof of Theorem~\ref{thm:TPathOneVertex} that we must then connect the node $c_0$ to all vertices in~$\mathcal{L}_{c_0}$.
As discussed in the previous case, these branches must again meet via odd connections, and the next even connection on each branch must be $I_y$. In order to connect, via an odd connection, the nodes labeled by $\mathcal{L}_{c_i}$ to nodes labeled by $\mathcal{L}_{c_j}$, for $i \neq j$, we would need to use a connection labeled either $I_y$ or $I_u$, where $u\in\Gamma_{> y}^v$. This would require introducing more internal nodes; at very least, we would have to introduce nodes labeled by vertices in $\mathcal{L}_{c_0}$. Then, by Rules \new{(7) and (8)}, the nodes labeled by vertices in $\mathcal{L}_{c_0}$ would be incident to an even connection with label $I_y$. Since this even connection would also have nodes labeled by elements in $\mathcal{L}_{c_i}$, we would create paths which violate Lemma~\ref{lem:NoExtendedVertexEvenStep}. Thus, we must connect each $\mathcal{L}_{c_i}$ for $1 \leq i\leq d$ directly to~$y$:
\begin{center}
\begin{tikzpicture}[scale = 1.8]
\node[left] at (0,2) {$c_0$};
\node[left] at (0,1) {$c_{1}$};
\node[left] at (0,0) {$\vdots$};
\node[left] at (0,-1) {$c_{d}$};
\draw[blue] (0,2) to (0.7,2);
\draw[blue] (0,2.05) to (1,2.2);
\draw[blue] (0.7,2) to (1,2.1);
\draw[blue] (0.7,2) to (1,1.9);
\draw[blue] (0,1.95) to (1,1.8);
\node[above] at (0.5,2.1) {$\bigcup_{w \in \Gamma_{\lessdot c_0}^v} I_{w}$};
\draw[blue] (0,1) to (1,1);
\draw[blue] (0,-1) to (1,-1);
\node[right] at (1,2) {$\mathcal{L}_{c_0}$};
\node[right] at (1,1) {$y$};
\node[right] at (1,0) {$\vdots$};
\node[right] at (1,-1) {$y$};
\draw[red,densely dashed] (1.4,2.2) to (1.7,2);
\draw[red,densely dashed] (1.4,2) to (1.7,2);
\draw[red,densely dashed] (1.4,1.8) to (1.7,2);
\draw[red,densely dashed] (1.7,2) to (2.4,2);
\node[above] at (1.9,2) {$I_{c_0}$};
\draw[red,densely dashed] (1.4,1) to (2.4,1);
\node[above] at (1.9,1) {$I_{c_{1}}$};
\draw[red,densely dashed] (1.4,-1) to (2.4,-1);
\node[above] at (1.9,-1) {$I_{c_{d}}$};
\node[right] at (2.4,2) {$y$};
\node[right] at (2.4,1) {$\mathcal{L}_{c_{1}}$};
\node[right] at (2.4,0) {$\vdots$};
\node[right] at (3.2,0) {$\vdots$};
\node[right] at (2.4,-1) {$\mathcal{L}_{c_{d}}$};
\draw[blue] (3.4,-1) to [out = 0, in = 0] (2.8,2.05);
\draw[blue] (3.1,-0.8) to (3.4,-1);
\draw[blue] (3.1,-1) to (3.4,-1);
\draw[blue] (3.1,-1.2) to (3.4,-1);
\draw[blue] (3.4,1) to [out = 0, in = 0] (2.8,2);
\draw[blue] (3.1,0.8) to (3.4,1);
\draw[blue] (3.1,1) to (3.4,1);
\draw[blue] (3.1,1.2) to (3.4,1);
\node[left] at (3.4,1.5){$I_{c_1}$};
\node[right] at (4,0.5) {$I_{c_d}$};
\end{tikzpicture}
\end{center}

 The node $y$ can either connect via an odd connection to $y^+$ or $\mathcal{L}_{c_0} \new{\sqcup \cdots \sqcup} \mathcal{L}_{c_d}$. In this case, we can decompose the hyper $T$-path into $T_y^{c_0}$ and a hyper $T$-path associated to $S - \{y\}$. Unlike in the previous case, the boundary node $y$ in the hyper $T$-path associated to $S - \{y\}$ does not necessarily connect to $y^+$ by an odd connection:
\begin{center}
\begin{tikzpicture}[scale = 1.5]
\node[left] at (0,2) {$c_0$};
\node[left] at (0,1) {$c_{1}$};
\node[left] at (0,0) {$\vdots$};
\node[left] at (0,-1) {$c_{d}$};
\draw[blue] (0,2) to (0.7,2);
\draw[blue] (0,2.05) to (1,2.2);
\draw[blue] (0.7,2) to (1,2.1);
\draw[blue] (0.7,2) to (1,1.9);
\draw[blue] (0,1.95) to (1,1.8);
\node[above] at (0.5,2.1) {$\prod_{w \in \Gamma_{\lessdot c_0}^v} I_{w}$};
\draw[blue] (0,1) to (1,1);
\draw[blue] (0,-1) to (1,-1);
\node[right] at (1,2) {$\mathcal{L}_{c_0}$};
\node[right] at (1,1) {$y$};
\node[right] at (1,0) {$\vdots$};
\node[right] at (1,-1) {$y$};
\draw[red,densely dashed] (1.4,2.2) to (1.7,2);
\draw[red,densely dashed] (1.4,2) to (1.7,2);
\draw[red,densely dashed] (1.4,1.8) to (1.7,2);
\draw[red,densely dashed] (1.7,2) to (2.4,2);
\node[above] at (1.9,2) {$I_{c_0}$};
\draw[red,densely dashed] (1.4,1) to (2.1,1);
\draw[red,densely dashed] (2.1,1) to (2.4,0.8);
\draw[red,densely dashed] (2.1,1) to (2.4,1);
\draw[red,densely dashed] (2.1,1) to (2.4,1.2);
\node[above] at (1.9,1) {$I_{c_{1}}$};
\draw[red,densely dashed] (1.4,-1) to (2.1,-1);
\draw[red,densely dashed] (2.1,-1) to (2.4,-0.8);
\draw[red,densely dashed] (2.1,-1) to (2.4,-1);
\draw[red,densely dashed] (2.1,-1) to (2.4,-1.2);
\node[above] at (1.9,-1) {$I_{c_{d}}$};
\node[right] at (2.4,2) {$y$};
\node[right] at (2.4,1) {$\mathcal{L}_{c_{1}}$};
\node[right] at (2.4,0) {$\vdots$};
\node[right] at (3.2,0) {$\vdots$};
\node[right] at (2.4,-1) {$\mathcal{L}_{c_{d}}$};
\draw[blue] (3.4,-1) to [out = 0, in = 0] (2.8,2.05);
\draw[blue] (3.1,-0.8) to (3.4,-1);
\draw[blue] (3.1,-1) to (3.4,-1);
\draw[blue] (3.1,-1.2) to (3.4,-1);
\draw[blue] (3.4,1) to [out = 0, in = 0] (2.8,2);
\draw[blue] (3.1,0.8) to (3.4,1);
\draw[blue] (3.1,1) to (3.4,1);
\draw[blue] (3.1,1.2) to (3.4,1);
\draw[blue] (2.8,2.1) to (3.8,2.1);
\node[right] at (3.8,2.15) {$y^+$};
\node[left] at (3.4,1.5){$I_{c_1}$};
\node[right] at (4,0.5) {$I_{c_d}$};
\node[] at (5,0.5){$\oplus$};
\node[left] at (6,0.5) {$y$};
\draw[blue] (6,0.5) -- (7,0.5);
\node[] at (7.3,0.5) {$\cdots$};
\end{tikzpicture}
\end{center}

The final case is when more than one of the boundary nodes $c_i$ are connected via odd connections to $\mathcal{L}_{c_i}$. If $y$ only covers one element, this case is not possible. Without loss of generality, label the elements covered by $y$ so that $c_0,\ldots,c_k$ are connected $\mathcal{L}_{c_0}, \ldots, \mathcal{L}_{c_k}$ respectively for $0 < k \leq d$ and $c_{k+1},\ldots,c_d$ are connected to $y$:
\begin{center}
\begin{tikzpicture}[scale = 1.5]
\node[left] at (0,1.5) {$c_0$};
\node[left] at (0,1) {$\vdots$};
\node[left] at (0,0) {$c_k$};
\node[left] at (0,-1) {$c_{k+1}$};
\node[left] at (0,-1.5) {$\vdots$};
\node[left] at (0,-2) {$c_d$};
\draw[blue] (0,1.5) to (0.7,1.5);
\draw[blue] (0,1.55) to (1,1.7);
\draw[blue] (0.7,1.5) to (1,1.6);
\draw[blue] (0.7,1.5) to (1,1.4);
\draw[blue] (0,1.45) to (1,1.3);
\node[above] at (0.5,1.6) {$\prod_{w \in \Gamma_{\lessdot c_0}^v} I_{w}$};
\draw[blue] (0,0) to (0.7,0);
\draw[blue] (0,0.05) to (1,0.2);
\draw[blue] (0.7,0) to (1,0.1);
\draw[blue] (0.7,0) to (1,-.1);
\draw[blue] (0,-0.05) to (1,-.2);
\node[above] at (0.5,0.1) {$\prod_{w \in \Gamma_{\lessdot c_k}^v} I_{w}$};
\draw[blue] (0,-1) to (1,-1);
\draw[blue] (0,-2) to (1,-2);
\node[right] at (1,1.5) {$\mathcal{L}_{c_0}$};
\node[right] at (1,1) {$\vdots$};
\node[right] at (1,0) {$\mathcal{L}_{c_k}$};
\node[right] at (1,-1) {$y$};
\node[right] at (1,-1.5) {$\vdots$};
\node[right] at (1,-2) {$y$};
\draw[red,densely dashed] (1.4,1.7) to (1.7,1.5);
\draw[red,densely dashed] (1.4,1.5) to (1.7,1.5);
\draw[red,densely dashed] (1.4,1.3) to (1.7,1.5);
\draw[red,densely dashed] (1.7,1.5) to (2.4,1.5);
\node[above] at (1.9,1.5) {$I_{c_0}$};
\draw[red,densely dashed] (1.4,0.2) to (1.7,0);
\draw[red,densely dashed] (1.4,0) to (1.7,0);
\draw[red,densely dashed] (1.4,-0.2) to (1.7,0);
\draw[red,densely dashed] (1.7,0) to (2.4,0);
\node[above] at (1.9,0) {$I_{c_k}$};
\draw[red,densely dashed] (1.4,-1) to (2.1,-1);
\draw[red,densely dashed] (2.1,-1) to (2.4,-0.8);
\draw[red,densely dashed] (2.1,-1) to (2.4,-1);
\draw[red,densely dashed] (2.1,-1) to (2.4,-1.2);
\node[above] at (1.9,-1) {$I_{c_{k+1}}$};
\draw[red,densely dashed] (1.4,-2) to (2.1,-2);
\draw[red,densely dashed] (2.1,-2) to (2.4,-1.8);
\draw[red,densely dashed] (2.1,-2) to (2.4,-2);
\draw[red,densely dashed] (2.1,-2) to (2.4,-2.2);
\node[above] at (1.9,-2) {$I_{c_{d}}$};
\node[right] at (2.4,1.5) {$y$};
\node[right] at (2.4,1) {$\vdots$};
\node[right] at (2.4,0) {$y$};
\node[right] at (2.4,-1) {$\mathcal{L}_{c_{k+1}}$};
\node[right] at (2.4,-1.5) {$\vdots$};
\node[right] at (2.4,-2) {$\mathcal{L}_{c_{d}}$};
\end{tikzpicture}
\end{center}

Since the hyper $T$-path is connected, there must be a path between the boundary nodes~$c_i$ and $c_j$ for $i \neq j$; by Rule~(9), this path can only pass through nodes labeled by $y$ and vertices in~$\mathcal{L}_{y}$. Each node in $\mathcal{L}_{c_i}$ can only be incident to one odd connection, so connecting these nodes to one of the nodes labeled $y$ would produce a disconnected hyper $T$-path. There is no way for all of these $d+1$ branches to come together in the next odd connection without producing extra branches. As discussed in the previous case, connecting the nodes in $\mathcal{L}_{c_{k+1}} \new{\sqcup \cdots \sqcup} \mathcal{L}_{c_d}$ would create paths between boundary nodes which violate Lemma~\ref{lem:NoExtendedVertexEvenStep}. Hence, this case does not produce any valid hyper $T$-paths.

We have therefore shown that all valid hyper $T$-paths associated to $S$ can be constructed via pasting in the desired manner.
\end{proof}

\begin{Theorem}
The hyper $T$-paths for a connected set $S$ are exactly obtained by pasting hyper $T$-paths for each element of $S$ together.
\end{Theorem}

\begin{proof}
By Proposition~\ref{prop:pasting}, pasting hyper $T$-paths for each element of $S$ together will always give us a hyper $T$-path for $S$. We also know by Theorem~\ref{thm:CanPeelSingleton} that all hyper $T$-paths for $S$ are formed this way.
\end{proof}

\subsection{Proof of Theorem~\ref{thm:main2}}\label{sec:pf-thm-2}

We begin by proving Theorem~\ref{thm:main2} when $|S|=1$.

\begin{Lemma}
Let $\Gamma$ be a tree and $\calC_v$ a rooted cluster for $\Gamma$. Then the cluster variable $Y_{\{i\}}$ has the combinatorial formula \[ Y_{\{i\}}=\sum_{\substack{\text{complete hyper}\\T\text{-paths }\gamma\text{ for }\{i\}}}\wt(\gamma).\]
\end{Lemma}

\begin{proof}
From Theorem \ref{thm:TPathOneVertex}, the hyper $T$-paths for $\{i\}$ are $T_i^{+}$ and $T_i^x$ for each $x\in\Gamma_{\lessdot i}^v$. The hyper $T$-path $T_i^{+}$ has weight \[ \frac{Y_{I_{i}}}{\prod_{w\in\Gamma_{\lessdot i}^v} Y_{I_{w}}}=\frac{Y_{I_{i}}}{Y_{\Gamma_{< i}^v}}.\]
The hyper $T$-path $T_i^x$ has weight
\[ \frac{\prod_{w \in \Gamma_{\lessdot x}^v} Y_{I_w}}{Y_{I_{x}}} = \frac{\Big(\prod_{w \in \Gamma_{\lessdot x}^v} Y_{I_w}\Big)\Big(\prod_{w\in\Gamma_{\lessdot i}^v\setminus\{x\}} Y_{I_{w}}\Big)}{\prod_{w\in\Gamma_{\lessdot i}^v} Y_{I_{w}}}= \frac{Y_{\Gamma_{< i}^v\setminus\{x\}} }{Y_{\Gamma_{< i}^v}}.\]
Added together, this is the same as the formula for $Y_{\{i\}}$ from Proposition~\ref{prop:Y-single-formula}.
\end{proof}

We can now prove our the theorem for general sets.

\begin{proof}[Proof of Theorem~\ref{thm:main2}]
To get a hyper $T$-path for $S$, we choose a hyper $T$-path for each $x\in S$ such that we are able to paste them all together. Let $O\subset S$ be the set of elements $x$ where we choose $T_x^{+}$. Notice that if $x$ is a minimal element of $\Gamma$, we must choose $T_x^{+}$ and therefore~$x$ must be in $O$. Define $u\colon S\setminus O\to V(\Gamma)$ so that we choose $T_x^{u(x)}$ for each $x\in S\setminus O$. Since~$T_x^{u(x)}$ does not have an connection joining nodes $x$ and $u(x)$, if $u(x)\in S$ we must have chosen a~hyper $T$-path for $u(x)$ that has such a connection. That means we must not have chosen $T_{u(x)}^+$. Equivalently, $u(x)$ must not be in $O$.

Thus, summing over all possible choices of hyper $T$-paths for the elements of $S$, we find
\begin{align*}
\sum_{\substack{\text{complete hyper}\\T\text{-paths }\gamma\text{ for }S}}\wt(\gamma)&=\sum_{\substack{O\subseteq S\text{ containing all}\\ \text{minimal elements of }\Gamma\text{ in }S}} \, \sum_{\substack{u\colon S\setminus O\to V(\Gamma)\\ u(x)\in\Gamma_{\lessdot x}^v\setminus O}}\bigg(\prod_{x\in O}\wt(T_x^{+})\bigg)\bigg(\prod_{x\in S\setminus O}\wt(T_x^{u(x)})\bigg)\\
&=\sum_{\substack{O\subseteq S\text{ containing all}\\ \text{minimal elements of }\Gamma\text{ in }S}} \, \sum_{\substack{u\colon S\setminus O\to V(\Gamma)\\ u(x)\in\Gamma_{\lessdot x}^v\setminus O}}\bigg(\prod_{x\in O}\frac{Y_{I_{x}}}{Y_{\Gamma_{< x}^v}}\bigg)\bigg(\prod_{x\in S\setminus O}\frac{Y_{\Gamma_{< x}^v\setminus\{u(x)\}} }{Y_{\Gamma_{< x}^v}}\bigg).
\end{align*}

By Theorem~\ref{thm:Y-formula}, this is $Y_S$.
\end{proof}

\section{Future directions}\label{sec:future}
\subsection{From rooted clusters to other clusters}

We would like to extend our results to other clusters for trees.

Our definition of rooted clusters comes from the algebraic formulas for the exchange relations. That is, rooted clusters are exactly the ones where the formulas in Proposition~\ref{prop:exchange-rels} give expansions in terms of the cluster variables. Proving formulas algebraically for other types of clusters will likely require an inductive argument. This induction seems to be easiest in star graphs because of their symmetry.

\begin{Conjecture}
Let $S_n$ denote the star graph on $n$ vertices whose central vertex is labeled by $1$. Let $\mathcal{C}$ be the cluster $\{ \{3 \}, \{ 4 \}, \dots, \{ n \}, \{ 1, 3, 4, \dots n \} \}$. For any vertex subset $S$ such that $1 \in S$, we conjecture that
\begin{align*}
 Y_S = \begin{cases}
 \displaystyle\frac{\prod\limits_{i \in S \backslash \{ 1, 2 \}} Y_i^2}{Y_{[n] \backslash \{ 1, 2 \}}} \bigg( \sum\limits_{i \not\in S} \prod\limits_{j \neq 1, 2, i} Y_j \bigg) + \frac{Y_{[n]}}{\prod\limits_{i \in [n] \backslash S} Y_i} + \frac{Y_{[n]}}{Y_{[n] \backslash \{ 2 \}}} \bigg( \sum\limits_{i \in [n] \backslash S} \frac{1}{Y_i} \bigg) & \textrm{if } 2 \not\in S, \vspace{2mm}\\
 \displaystyle \frac{Y_{[n] \backslash \{ 2 \}}}{\prod\limits_{i \in [n] \backslash (S \cup \{ 2 \})} Y_i} + \bigg(\prod\limits_{i \in S \backslash \{ 1 \}} Y_i \bigg) \bigg( \sum\limits_{j \in [n] \backslash (S \cup \{ 2 \})} \frac{1}{Y_i} \bigg) & \textrm{if } 2 \in S.
 \end{cases}
\end{align*}
If $1 \not\in S$, then $S$ either consists of a set of disconnected leaves or a single leaf. Because $\{ i \} \in \mathcal{C}$ for all $i \neq 2$, we then have
\begin{align*}
 Y_s = \begin{cases}
 Y_{i} & \textrm{if }S = \{ i \}, \\
 \prod\limits_{i \in S} Y_i & \textrm{else}.
 \end{cases}
\end{align*}
\end{Conjecture}

We are hopeful that we can extend our hyper $T$-path expansion formula to other types of clusters. For a~type~$A$ cluster algebra, $T$-paths can be used to find expansions for cluster variables in terms of any cluster. Because the definitions of $T$-paths and hyper $T$-paths for path graphs are similar, this suggests that it might be possible to use hyper $T$-paths for other clusters when~$\Gamma$ is an arbitrary tree.

Unfortunately, our current hyper $T$-path construction does not work for arbitrary clusters. One immediate problem is that Rules (7) and (8) would need to be rewritten to allow the even edges to be labelled by any set in the cluster that is incompatible with $S$. However, that change still would not be sufficient because we still wouldn't have ``enough'' valid hyper $T$-paths. For example, suppose $\Gamma$ is the graph from Figure~\ref{fig:example_hypergraph} and $\mathcal{C}$ is the cluster $\{Y_5,y_{25},Y_{125}Y_{1235},Y_{12345}\}$. Then the term $\frac{Y_5^2}{Y_{125}}$ appears in the expansion of $Y_{235}$, but we have been unable to find a hyper $T$-path with that weight. Further, the expansions of some cluster variables with respect to certain clusters contain monomials with squared terms in the denominator. This is not possible with our current definition, as we see from Lemma~\ref{lem:evenconnections}. Thus, it is clear that there is some other substantial change required for us to be able to extend our construction to other clusters.

\subsection{Snake graphs}
In \cite{MSW}, Musiker, Schiffler and Williams provided an alternative combinatorial formula for type $A$ cluster algebras using perfect matchings (also known as dimer models) on certain \emph{snake graphs}.

For an arc $\gamma$ in a triangulation $T$, the snake graph $G_{\gamma, T}$ is built up from quadrilateral tiles whose diagonals correspond to arcs of $T$ that intersect with $\gamma$. We illustrate the definition via the following example and refer to~\cite{MSW} for a complete exposition on snake graphs.

The snake graph corresponding to the arc $M_{i,j}$ in the triangulation $T$ from Example~\ref{ex:cluster_t_path} is as follows:
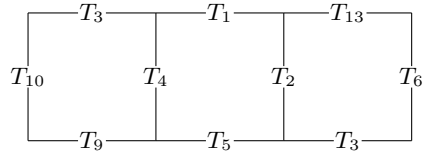
\begin{center}
\begin{tikzpicture}[scale=0.85]
	\node [style=none] (0) at (-6, 4) {};
		\node [style=none] (1) at (-6, 2) {};
		\node [style=none] (2) at (-4, 2) {};
		\node [style=none] (3) at (-4, 4) {};
		\node [style=label] (4) at (-6, 3) {$T_{10}$};
		\node [style=label] (5) at (-5, 2) {$T_9$};
		\node [style=label] (6) at (-4, 3) {$T_4$};
		\node [style=label] (7) at (-5, 4) {$T_3$};
		\node [style=none] (8) at (-2, 2) {};
		\node [style=none] (9) at (-2, 4) {};
		\node [style=label] (10) at (-3, 4) {$T_1$};
		\node [style=label] (11) at (-3, 2) {$T_5$};
		\node [style=label] (12) at (-2, 3) {$T_2$};
		\node [style=none] (13) at (0, 4) {};
		\node [style=none] (14) at (0, 2) {};
		\node [style=label] (15) at (-1, 4) {$T_{13}$};
		\node [style=label] (16) at (-1, 2) {$T_3$};
		\node [style=label] (17) at (0, 3) {$T_6$};
		\draw (0.center) to (4);
		\draw (4) to (1.center);
		\draw (1.center) to (5);
		\draw (5) to (2.center);
		\draw (2.center) to (6);
		\draw (3.center) to (6);
		\draw (3.center) to (7);
		\draw (7) to (0.center);
		\draw (2.center) to (11);
		\draw (11) to (8.center);
		\draw (12) to (8.center);
		\draw (12) to (9.center);
		\draw (9.center) to (10);
		\draw (10) to (3.center);
		\draw (9.center) to (15);
		\draw (15) to (13.center);
		\draw (13.center) to (17);
		\draw (17) to (14.center);
		\draw (14.center) to (16);
		\draw (16) to (8.center);
\end{tikzpicture}
\end{center}

A \emph{perfect matching} of a graph is a collection $M$ of its edges such that every vertex in the graph is incident to exactly one edge in $M$. The \emph{weight} of a perfect matching is the product of all of the cluster variables associated to edges in $M$. For example, the set of edges with labels $T_3$, $T_3$, $T_9$, $T_{13}$ is a perfect matching with weight $x_{3}^2x_{9}x_{13}$ for the previous snake graph:
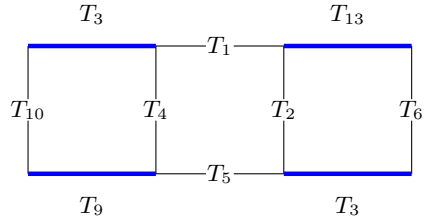
\begin{center}
\begin{tikzpicture}[scale=0.85]
 \node [style=none] (0) at (-6, 4) {};
		\node [style=none] (1) at (-6, 2) {};
		\node [style=none] (2) at (-4, 2) {};
		\node [style=none] (3) at (-4, 4) {};
		\node [style=label] (4) at (-6, 3) {$T_{10}$};
		\node [style=label] (5) at (-5, 1.5) {$T_9$};
		\node [style=label] (6) at (-4, 3) {$T_4$};
		\node [style=label] (7) at (-5, 4.5) {$T_3$};
		\node [style=none] (8) at (-2, 2) {};
		\node [style=none] (9) at (-2, 4) {};
		\node [style=label] (10) at (-3, 4) {$T_1$};
		\node [style=label] (11) at (-3, 2) {$T_5$};
		\node [style=label] (12) at (-2, 3) {$T_2$};
		\node [style=none] (13) at (0, 4) {};
		\node [style=none] (14) at (0, 2) {};
		\node [style=label] (15) at (-1, 4.5) {$T_{13}$};
		\node [style=label] (16) at (-1, 1.5) {$T_3$};
		\node [style=label] (17) at (0, 3) {$T_6$};
		\draw (0.center) to (4);
		\draw (4) to (1.center);
		\draw (2.center) to (6);
		\draw (3.center) to (6);
		\draw (2.center) to (11);
		\draw (11) to (8.center);
		\draw (12) to (8.center);
		\draw (12) to (9.center);
		\draw (9.center) to (10);
		\draw (10) to (3.center);
		\draw (13.center) to (17);
		\draw (17) to (14.center);
		\draw [style=blue thick] (9.center) to (13.center);
		\draw [style=blue thick] (8.center) to (14.center);
		\draw [style=blue thick] (0.center) to (3.center);
		\draw [style=blue thick] (1.center) to (2.center);
\end{tikzpicture}
\end{center}

The following theorem of \cite{MSW} gives an explicit combinatorial formula for the cluster variable associated to the arc $\gamma=M_{i,j}$.
\begin{Theorem}[{\cite[Theorem~4.9]{MSW}}]
\label{thm:snake_graph}
The cluster variable $x_\gamma$ of an arc $\gamma$ has a Laurent expansion in terms of the cluster corresponding to $T$ which is given by
	\[x_{\gamma}=
	{1\over\textup{cross}(\gamma)}
	\sum_{\substack{M\text{ is a perfect}\\ \text{matching of }G_{\gamma, T}}} \wt(M),\]
	where $\textup{cross}(\gamma)$ is the weighted product of all diagonals in $T$ which $\gamma$ crosses.
\end{Theorem}

There is a weight-preserving bijection between the set of perfect matching of $G_{\gamma,T}$ and the set of all (usual) $T$-paths from $i$ to $j$: by adding diagonals representing arcs in $\textup{cross}(\gamma)$ to a~perfect matching $M$, one obtains a complete $T$-path where the edges in $M$ are the odd steps and the diagonals in $\textup{cross}(\gamma)$ are the even steps. For example, adding diagonals to the above perfect matching, we obtain the same complete $T$-path as in Example~\ref{ex:cluster_t_path}:

\begin{center}
\begin{tikzpicture}[scale=0.85]
	\begin{pgfonlayer}{nodelayer}
		\node [style=none] (0) at (-6, 4) {};
		\node [style=none] (1) at (-6, 2) {};
		\node [style=none] (2) at (-4, 2) {};
		\node [style=none] (3) at (-4, 4) {};
		\node [style=label] (4) at (-6, 3) {$T_{10}$};
		\node [style=label] (5) at (-5, 1.5) {$T_9$};
		\node [style=label] (6) at (-4, 3) {$T_4$};
		\node [style=label] (7) at (-5, 4.5) {$T_3$};
		\node [style=none] (8) at (-2, 2) {};
		\node [style=none] (9) at (-2, 4) {};
		\node [style=label] (10) at (-3, 4) {$T_1$};
		\node [style=label] (11) at (-3, 2) {$T_5$};
		\node [style=label] (12) at (-2, 3) {$T_2$};
		\node [style=none] (13) at (0, 4) {};
		\node [style=none] (14) at (0, 2) {};
		\node [style=label] (15) at (-1, 4.5) {$T_{13}$};
		\node [style=label] (16) at (-1, 1.5) {$T_3$};
		\node [style=label] (17) at (0, 3) {$T_6$};
		\node [style=label] (18) at (-5, 3) {$T_5$};
		\node [style=label] (19) at (-3, 3) {$T_3$};
		\node [style=label] (20) at (-1, 3) {$T_1$};
	\end{pgfonlayer}
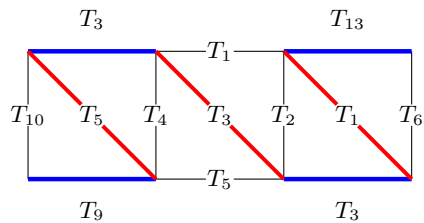
	\begin{pgfonlayer}{edgelayer}
		\draw (0.center) to (4);
		\draw (4) to (1.center);
		\draw (2.center) to (6);
		\draw (3.center) to (6);
		\draw (2.center) to (11);
		\draw (11) to (8.center);
		\draw (12) to (8.center);
		\draw (12) to (9.center);
		\draw (9.center) to (10);
		\draw (10) to (3.center);
		\draw (13.center) to (17);
		\draw (17) to (14.center);
		\draw [style=blue thick] (9.center) to (13.center);
		\draw [style=blue thick] (8.center) to (14.center);
		\draw [style=blue thick] (0.center) to (3.center);
		\draw [style=blue thick] (1.center) to (2.center);
		\draw [style=red] (0.center) to (18);
		\draw [style=red] (18) to (2.center);
		\draw [style=red] (3.center) to (19);
		\draw [style=red] (19) to (8.center);
		\draw [style=red] (9.center) to (20);
		\draw [style=red] (20) to (14.center);
	\end{pgfonlayer}
\end{tikzpicture}
\end{center}

We hope that a graph-theoretic formula similar to the cluster expansion formula of Musiker, Schiffler, and Williams exists for graph LP algebras from generic trees; this would provide another approach to prove the positivity conjecture for more general clusters of LP algebras from trees. The rough idea is explained in the following example.

\begin{Example} We draw a snake graph associated to the set $\{2,3\}$ of the graph $\Gamma$ in Figure~\ref{fig:example_hypergraph} as follows.
Note that the two edges $1'-2$ and $5'-2$ are considered as one single edge when we take a matching:
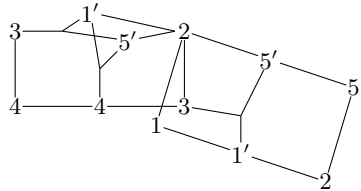
\begin{center}
\begin{tikzpicture}
	\begin{pgfonlayer}{nodelayer}
		\node [style=label] (0) at (-7, 3) {$3$};
		\node [style=label] (1) at (-7, 1) {$4$};
		\node [style=label] (2) at (-4, 2.75) {$5'$};
		\node [style=label] (3) at (-4.75, 1) {$4$};
		\node [style=label] (4) at (-2.5, 3) {$2$};
		\node [style=label] (5) at (-2.5, 1) {$3$};
		\node [style=label] (6) at (-3.25, 0.5) {$1$};
		\node [style=label] (7) at (-0.25, 2.25) {$5'$};
		\node [style=label] (8) at (-1, -0.25) {$1'$};
		\node [style=label] (9) at (2, 1.5) {$5$};
		\node [style=label] (10) at (1.25, -1) {$2$};
		\node [style=none] (11) at (-1, 0.75) {};
		\node [style=label] (12) at (-5, 3.5) {$1'$};
		\node [style=none] (13) at (-5.75, 3) {};
		\node [style=none] (14) at (-4.75, 2) {};
		\node [style=none] (15) at (-2.75, 3) {};
	\end{pgfonlayer}
	\begin{pgfonlayer}{edgelayer}
		\draw (0) to (1);
		\draw (1) to (3);
		\draw (6) to (4);
		\draw (3) to (5);
		\draw (5) to (4);
		\draw (4) to (7);
		\draw (7) to (9);
		\draw (9) to (10);
		\draw (10) to (8);
		\draw (6) to (8);
		\draw (7) to (11.center);
		\draw (11.center) to (5);
		\draw (11.center) to (8);
		\draw (0) to (13.center);
		\draw (13.center) to (12);
		\draw (13.center) to (2);
		\draw (12) to (14.center);
		\draw (14.center) to (3);
		\draw (14.center) to (2);
		\draw (12) to (15.center);
		\draw (15.center) to (2);
	\end{pgfonlayer}
\end{tikzpicture}
\end{center}

Similar to the cluster case (Theorem~\ref{thm:snake_graph}), the cluster variable $Y_{23}$ is given by the weighted sum of all matchings divided by a monomial given by the elements of the cluster incompatible with $\{2,3\}$.

The following are several matchings of the above snake graph, where vertex $2$ is allowed (but not required) to have 2 adjacent matched edges:
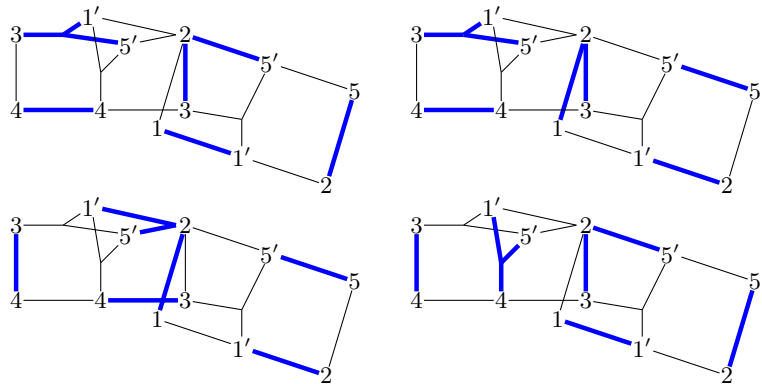
\begin{center}
 \begin{tikzpicture}
	\begin{pgfonlayer}{nodelayer}
		\node [style=label] (0) at (-7, 3) {$3$};
		\node [style=label] (1) at (-7, 1) {$4$};
		\node [style=label] (2) at (-4, 2.75) {$5'$};
		\node [style=label] (3) at (-4.75, 1) {$4$};
		\node [style=label] (4) at (-2.5, 3) {$2$};
		\node [style=label] (5) at (-2.5, 1) {$3$};
		\node [style=label] (6) at (-3.25, 0.5) {$1$};
		\node [style=label] (7) at (-0.25, 2.25) {$5'$};
		\node [style=label] (8) at (-1, -0.25) {$1'$};
		\node [style=label] (9) at (2, 1.5) {$5$};
		\node [style=label] (10) at (1.25, -1) {$2$};
		\node [style=none] (11) at (-1, 0.75) {};
		\node [style=label] (12) at (-5, 3.5) {$1'$};
		\node [style=none] (13) at (-5.75, 3) {};
		\node [style=none] (14) at (-4.75, 2) {};
		\node [style=none] (15) at (-2.75, 3) {};
	\end{pgfonlayer}
	\begin{pgfonlayer}{edgelayer}
		\draw (0) to (1);
		\draw [style=blue thick] (1) to (3);
		\draw (6) to (4);
		\draw (3) to (5);
		\draw [style=blue thick] (5) to (4);
		\draw [style=blue thick] (4) to (7);
		\draw (7) to (9);
		\draw [style=blue thick] (9) to (10);
		\draw (10) to (8);
		\draw [style=blue thick] (6) to (8);
		\draw (7) to (11.center);
		\draw (11.center) to (5);
		\draw (11.center) to (8);
		\draw [style=blue thick] (0) to (13.center);
		\draw [style=blue thick] (13.center) to (12);
		\draw [style=blue thick] (13.center) to (2);
		\draw (12) to (14.center);
		\draw (14.center) to (3);
		\draw (14.center) to (2);
		\draw (12) to (15.center);
		\draw (15.center) to (2);
	\end{pgfonlayer}
\end{tikzpicture}\quad\begin{tikzpicture}
	\begin{pgfonlayer}{nodelayer}
		\node [style=label] (0) at (-7, 3) {$3$};
		\node [style=label] (1) at (-7, 1) {$4$};
		\node [style=label] (2) at (-4, 2.75) {$5'$};
		\node [style=label] (3) at (-4.75, 1) {$4$};
		\node [style=label] (4) at (-2.5, 3) {$2$};
		\node [style=label] (5) at (-2.5, 1) {$3$};
		\node [style=label] (6) at (-3.25, 0.5) {$1$};
		\node [style=label] (7) at (-0.25, 2.25) {$5'$};
		\node [style=label] (8) at (-1, -0.25) {$1'$};
		\node [style=label] (9) at (2, 1.5) {$5$};
		\node [style=label] (10) at (1.25, -1) {$2$};
		\node [style=none] (11) at (-1, 0.75) {};
		\node [style=label] (12) at (-5, 3.5) {$1'$};
		\node [style=none] (13) at (-5.75, 3) {};
		\node [style=none] (14) at (-4.75, 2) {};
		\node [style=none] (15) at (-2.75, 3) {};
	\end{pgfonlayer}
	\begin{pgfonlayer}{edgelayer}
		\draw (0) to (1);
		\draw [style=blue thick] (1) to (3);
		\draw [style=blue thick] (6) to (4);
		\draw (3) to (5);
		\draw [style=blue thick] (5) to (4);
		\draw (4) to (7);
		\draw [style=blue thick] (7) to (9);
		\draw (9) to (10);
		\draw [style=blue thick] (10) to (8);
		\draw (6) to (8);
		\draw (7) to (11.center);
		\draw (11.center) to (5);
		\draw (11.center) to (8);
		\draw [style=blue thick] (0) to (13.center);
		\draw [style=blue thick] (13.center) to (12);
		\draw [style=blue thick] (13.center) to (2);
		\draw (12) to (14.center);
		\draw (14.center) to (3);
		\draw (14.center) to (2);
		\draw (12) to (15.center);
		\draw (15.center) to (2);
	\end{pgfonlayer}
\end{tikzpicture}

 \begin{tikzpicture}
	\begin{pgfonlayer}{nodelayer}
		\node [style=label] (0) at (-7, 3) {$3$};
		\node [style=label] (1) at (-7, 1) {$4$};
		\node [style=label] (2) at (-4, 2.75) {$5'$};
		\node [style=label] (3) at (-4.75, 1) {$4$};
		\node [style=label] (4) at (-2.5, 3) {$2$};
		\node [style=label] (5) at (-2.5, 1) {$3$};
		\node [style=label] (6) at (-3.25, 0.5) {$1$};
		\node [style=label] (7) at (-0.25, 2.25) {$5'$};
		\node [style=label] (8) at (-1, -0.25) {$1'$};
		\node [style=label] (9) at (2, 1.5) {$5$};
		\node [style=label] (10) at (1.25, -1) {$2$};
		\node [style=none] (11) at (-1, 0.75) {};
		\node [style=label] (12) at (-5, 3.5) {$1'$};
		\node [style=none] (13) at (-5.75, 3) {};
		\node [style=none] (14) at (-4.75, 2) {};
		\node [style=none] (15) at (-2.75, 3) {};
	\end{pgfonlayer}
	\begin{pgfonlayer}{edgelayer}
		\draw [style=blue thick] (0) to (1);
		\draw (1) to (3);
		\draw [style=blue thick] (6) to (4);
		\draw [style=blue thick] (3) to (5);
		\draw (5) to (4);
		\draw (4) to (7);
		\draw [style=blue thick] (7) to (9);
		\draw (9) to (10);
		\draw [style=blue thick] (10) to (8);
		\draw (6) to (8);
		\draw (7) to (11.center);
		\draw (11.center) to (5);
		\draw (11.center) to (8);
		\draw (0) to (13.center);
		\draw (13.center) to (12);
		\draw (13.center) to (2);
		\draw (12) to (14.center);
		\draw (14.center) to (3);
		\draw (14.center) to (2);
		\draw [style=blue thick] (12) to (15.center);
		\draw [style=blue thick] (2) to (15.center);
	\end{pgfonlayer}
\end{tikzpicture}\quad\begin{tikzpicture}
	\begin{pgfonlayer}{nodelayer}
		\node [style=label] (0) at (-7, 3) {$3$};
		\node [style=label] (1) at (-7, 1) {$4$};
		\node [style=label] (2) at (-4, 2.75) {$5'$};
		\node [style=label] (3) at (-4.75, 1) {$4$};
		\node [style=label] (4) at (-2.5, 3) {$2$};
		\node [style=label] (5) at (-2.5, 1) {$3$};
		\node [style=label] (6) at (-3.25, 0.5) {$1$};
		\node [style=label] (7) at (-0.25, 2.25) {$5'$};
		\node [style=label] (8) at (-1, -0.25) {$1'$};
		\node [style=label] (9) at (2, 1.5) {$5$};
		\node [style=label] (10) at (1.25, -1) {$2$};
		\node [style=none] (11) at (-1, 0.75) {};
		\node [style=label] (12) at (-5, 3.5) {$1'$};
		\node [style=none] (13) at (-5.75, 3) {};
		\node [style=none] (14) at (-4.75, 2) {};
		\node [style=none] (15) at (-2.75, 3) {};
	\end{pgfonlayer}
	\begin{pgfonlayer}{edgelayer}
		\draw [style=blue thick] (0) to (1);
		\draw (1) to (3);
		\draw (6) to (4);
		\draw (3) to (5);
		\draw [style=blue thick] (5) to (4);
		\draw [style=blue thick] (4) to (7);
		\draw (7) to (9);
		\draw [style=blue thick] (9) to (10);
		\draw (10) to (8);
		\draw [style=blue thick] (6) to (8);
		\draw (7) to (11.center);
		\draw (11.center) to (5);
		\draw (11.center) to (8);
		\draw (0) to (13.center);
		\draw (13.center) to (12);
		\draw (13.center) to (2);
		\draw [style=blue thick] (12) to (14.center);
		\draw [style=blue thick] (14.center) to (3);
		\draw [style=blue thick] (14.center) to (2);
		\draw (12) to (15.center);
		\draw (2) to (15.center);
	\end{pgfonlayer}
\end{tikzpicture}

\end{center}
Note that these matchings have the same weight as the complete hyper $T$-paths in Example~\ref{ex:hyper_T-path}.
\end{Example}

Observe that in the preceding example, all vertices in the underlying graph $\Gamma$ had degree three or less. When the vertices in $\Gamma$ have higher degree, it is unclear how to draw the snake graphs.

\subsection*{Acknowledgements}

We would like to thank Pavlo Pylyavskyy for suggesting this problem, Trevor Karn for organizing the 2021 Minnesota Combinatorics Working Group, and Kayla Wright and Libby Farrell for generating initial computational data. We thank the anonymous referees for their thoughtful comments, which significantly improved the paper.
This material is based upon work supported by the National Science Foundation under Grant No.~DMS-1439786 while the second author was in residence at the Institute for Computational and Experimental Research in Mathematics in Providence, RI, during the Combinatorial Algebraic Geometry program. The third author was partially supported by NSF Grant No. DMS-1937241.


\pdfbookmark[1]{References}{ref}
\LastPageEnding

\end{document}